\documentclass[en,oneside,bbfont,article]{amsart} 
\usepackage[lmargin = 30mm, rmargin = 30mm, bmargin = 25mm, tmargin=25mm]{geometry}
\usepackage[utf8]{inputenc}
\usepackage{xcolor}

\usepackage{mathrsfs}
\usepackage{mathtools}
\usepackage{enumitem}
\usepackage{mathabx}
\usepackage{mdframed}
\usepackage{amssymb}
\usepackage{dsfont}
\usepackage{amsthm}
\usepackage{bm}
\usepackage{amsmath,stackrel}
\usepackage{comment}
\usepackage{marginnote}
\usepackage{caption,subcaption}
\usepackage{scalefnt}
\usepackage{float,graphicx,verbatim}
\usepackage{tikz}
\usetikzlibrary{calc}
\usepackage{pgfplots}
\pgfplotsset{compat=1.16}
\RequirePackage[colorlinks=true,linkcolor=black,
	citecolor=blue,urlcolor=blue]{hyperref}

\usepackage{mathtools}
\mathtoolsset{showonlyrefs}
\numberwithin{equation}{section}

\providecommand{\examplename}{Example}

\newtheorem{theorem}{Theorem}[section]
\newtheorem{proposition}[theorem]{Proposition}
\newtheorem{corollary}[theorem]{Corollary}
\newtheorem{lemma}[theorem]{Lemma}
\theoremstyle{remark}
\newtheorem{remarkx}[theorem]{Remark}
\newenvironment{remark}
    {\pushQED{\qed}\remarkx}
    {\popQED\endremarkx}
\theoremstyle{definition}
\newtheorem*{example*}{\protect\examplename}

\theoremstyle{plain}

\newtheorem*{assumption*}{Assumption}

\usepackage{comment}

\usepackage{array}

\newcolumntype{C}{>{\centering\arraybackslash}m{3cm}}
\newcolumntype{T}{>{\centering\arraybackslash}m{5.65cm}}
\newcolumntype{L}{>{\centering\arraybackslash}m{8.5cm}}

\newcommand\N{\mathbb{N}}
\newcommand\Z{\mathbb{Z}}
\newcommand\R{\mathbb{R}}
\newcommand\E{\mathds{E}}
\newcommand\p{\mathds{P}}
\newcommand\1{\mathds{1}}
\newcommand\ld{,\ldots,}
\newcommand\eqd{\overset{\mathscr{D}}{=}}

\newcommand{\D}{\mathrm{d}}
\newcommand\cid{\xrightarrow{\mathscr{D}}}

\newcommand{\wt}[1]{\widetilde{#1}}
\newcommand{\wh}[1]{\widehat{#1}}

\newcommand\var{\mathrm{Var}}
\newcommand\cov{\mathrm{Cov}}
\newcommand\cor{\mathrm{Corr}}
\newcommand\Dom{\mathrm{Dom}}
\newcommand\Diag{\mathrm{Diag}}
\newcommand\erf{\mathrm{erf}}
\newcommand\KS{\mathrm{KS}}
\newcommand\Kol{\mathrm{Kol}}

\newcommand{\mW}{\mathcal{W}}
\newcommand{\HR}{\mathcal{R}}
\newcommand{\CD}{\mathscr{C}}
\newcommand{\tra}{{\scalebox{0.6}{$\top$}}}
\newcommand\nf[1]{\normalfont{#1}}

\newcommand\mH{\mathfrak{H}}
\newcommand{\OPn}[1]{\|{#1}\|_{\mathrm{op}}}

\newcommand{\floor}[1]{\lfloor {#1}\rfloor }

\newcommand{\mD}{\mathbb{D}^{1,2}}

\usepackage[backend=biber,style=ieee,firstinits=true,citestyle=numeric,sorting=nyt,maxbibnames=99,dashed=false,sortcites=true]{biblatex}
\renewbibmacro{in:}{}
\appto{\bibsetup}{\sloppy}
\usepackage[foot]{amsaddr}

\title[Quantitative bounds for high-dimensional non-linear functionals]{Quantitative bounds for high-dimensional non-linear functionals of Gaussian processes}
\author{Andreas Basse-O'Connor$^*$ \& David Kramer-Bang$^{\dag}$}

\address{$^{*\dag}$Department of Mathematics, Aarhus University, DK}

\email{$^*$basse@math.au.dk}
\email{$^\dag$bang@math.au.dk}

\begin{document}

\begin{abstract}
In this paper, we establish explicit quantitative Berry--Esseen bounds in the hyper-rectangle distance $d_\HR$, the convex distance $d_\CD$ and the $1$-Wasserstein distance $d_\mW$ for high-dimensional, non-linear functionals of Gaussian processes, allowing for strong dependence between variables. Our main result demonstrates that, under a smoothness assumption, the convergence rate under $d_\HR$ is sub-polynomial in the dimension and polynomial under $d_\CD$ and $d_\mW$. To the best of our knowledge, our results under $d_\HR$ provide the first explicit sub-polynomial bound for high-dimensional, non-linear functionals of Gaussian processes beyond the i.i.d. setting. Building on this, we derive explicit Berry--Esseen bounds under both $d_\HR$ and $d_\CD$ for multiple statistical examples, such as the method of moments, empirical characteristic functions, empirical moment-generating functions, and functional limit theorems in high-dimensional settings.
\end{abstract}



\maketitle

\section{Introduction}\label{sec:intro}

Central limit theorems (CLTs) are some of the most important cornerstones in our understanding of randomness. They provide a universality result showing that the normal distribution emerges under rescaled aggregation of randomness, regardless of the initial distribution or the details of the model under consideration, in a broad setting. Key functionals of interest often have highly intractable distributions. However, CLTs show that, for a sufficiently large sample size, the limiting distribution is asymptotically Gaussian. Because of their widespread applications in the applied sciences, both univariate and multivariate CLTs have been extensively studied for decades. Justifying the application of CLTs requires quantifying when the sample size is ``large enough''. Such results are usually referred to under the umbrella of Berry--Esseen bounds, and they quantify the finite sample rate of convergence in the CLT based on a chosen metric. In this paper, we focus on the hyper-rectangle distance $d_\HR$ and the convex distance $d_\CD$ defined in~\eqref{eq:d_R_and_d_C_defn} below, whenever working with multivariate data. In the case of independent and identically distributed (i.i.d.) univariate random variables $(X_k)_{k \in \N}$ with finite third moment, we have for $S_n=\sqrt{n}(n^{-1}\sum_{k=1}^n X_k - \E[X_1])$ and $Z\sim \mathcal N(0,\var(X_1))$, that
\begin{equation}\label{sljdfljsdlfjl}
    d_\Kol(S_n,Z)\coloneqq \sup_{x \in \R}|\p(S_n \le x)-\p(Z\le x)|\leq C\frac{\E[|X_1|^3]}{\sqrt{n}},\qquad  \text{for all }n\geq 1,
\end{equation} where $C>0$ is a universal constant and $d_\Kol(S_n,Z)$ denotes the Kolmogorov distance between the random variables $S_n$ and $Z$. That is,~\eqref{sljdfljsdlfjl} shows that the convergence rate in the central limit theorem is of the order $n^{-1/2}\E[|X_1|^3]$; moreover, this rate is optimal in the sense that symmetric Bernoulli random variables achieve a lower bound with the same rate, cf.~\cite{MR90166}. This result gives an explicit formula for how the third absolute moment $\E[|X_1|^3]$ and the number of samples $n$ affect the rate of convergence in the CLT.  However, this fundamental result is stated in a highly simplistic model. 
Many statistical problems involve both (a) high-dimensional data and (b) strong dependence, both of which typically slow down convergence rates in the CLT compared to the classical univariate i.i.d.\ case. While there has been significant progress in understanding the effects of dimensionality on convergence rates in i.i.d.\ settings, much less is known about how strong dependence further impacts these rates; see Subsections~\ref{sec_i.i.d._intro} \&~\ref{sec_non-lin_univ_intro}, and Section~\ref{sec:connec_lit} for references. To the best of our knowledge, explicit quantitative bounds that capture the interplay between sample size, dependence, and dimensionality in high-dimensional dependent models remain largely unexplored.

In this paper, we construct explicit quantitative bounds for high-dimensional non-linear functionals of Gaussian processes, a setting that naturally encompasses both high-dimensionality and strong dependence. In this multivariate setting,  we will work with three metrics, namely the convex distance $d_\CD$, the hyper-rectangle distance $d_\HR$ and the $1$-Wasserstein distance $d_\mW$. For general random vectors $\bm{X},\,\bm{Y}$, these are defined by
\begin{equation}\label{eq:d_R_and_d_C_defn}
\begin{aligned}
    d_\HR(\bm{X},\bm{Y})
    &\coloneqq 
    \sup_{A \in \HR}|\p(\bm{X} \in A)-\p(\bm{Y} \in A)|, \quad &&\text{where } \HR=\{\bigtimes_{i=1}^d (a_i,b_i): -\infty \le a_i \le  b_i \le \infty\}, \\ 
    d_\CD(\bm{X},\bm{Y})
    &\coloneqq 
    \sup_{A \in \CD}|\p(\bm{X} \in A)-\p(\bm{Y} \in A)|,\quad &&\text{where }\CD=\{A \in \mathcal{B}(\R^d): A\text{ is convex}\},\\
    d_{\mW}(\bm{X},\bm{Y}) &\coloneqq \sup_{f \in \mW}|\E[f(\bm{X})]-\E[f(\bm{Y})]|,\quad &&\text{where }\mW=\{ f:\R^d \to \R\,:\, f\text{ is  $1-$Lipschitz}\},
\end{aligned}
\end{equation} where $d_\HR(\bm{X},\bm{Y}) \le d_\CD(\bm{X},\bm{Y})$. Our main result, Theorem~\ref{thm:main_theorem}, establishes the first such explicit quantitative bounds in the general setting of Section~\ref{sec:high_dim_non_lin_subsec}, with a framework that can generalise to i.i.d.\ cases (see Proposition~\ref{prop:Phi_G_recovers_all_dist}). Despite their fundamental role in statistical methodologies, these functionals have received little attention in terms of explicit quantitative bounds for the fluctuations. This paper fills that gap, with statistical applications detailed in Section~\ref{sec:applications_MoM_ChaFct_fdd}. 

\subsection{I.I.D.\ multivariate setting}\label{sec_i.i.d._intro}

Extending the CLT to the multivariate setting (i.e.\ $d \ge 2$), we consider a sequence of i.i.d.\ $d$-dimensional random vectors $(\bm{X}_k)_{k \in \N}$ with $\E[\bm{X}_1]=\bm{0}$ and covariance matrix $\bm{\Sigma}$. Then, the multivariate CLT states that $\bm{S}_n = n^{-1/2}\sum_{k=1}^n\bm{X}_k \cid \bm{Z}$ as $n \to \infty$, where $\bm{Z}$ is a standard multivariate Gaussian random vector in $\R^d$ with mean $\bm{0}$ and covariance matrix $\bm{\Sigma}$ (i.e.\ $\bm{Z} \sim \mathcal{N}_d(\bm{0},\bm{\Sigma})$). As in the univariate case, understanding the rate of convergence in the multivariate CLT has been of interest for multiple decades; see, e.g.~\cite{MR4152649,MR2962301,MR2144310,MR4003566,MR4488569,MR4312842,bong2022highdimensional,cammarota2023quantitative,MR4505371,MR4583674} and references herein. The topic has especially gained traction recently with an additional focus on high-dimensional CLTs and quantitative bounds with an explicit dimensional dependence. In the i.i.d.\ setting, and under the additional assumption that $\E[|\bm{X}_1|^3]<\infty$ and $\cov(\bm{X}_1)=\bm{I}_d$,~\cite[Thm~1.1]{MR2144310} implies the existence of a universal constant $C>0$, such that 
\begin{equation}\label{eq:hyper_rect_defn}
    d_\CD(\bm{S}_n,\bm{Z})
    \le  
    C\frac{ d^{1/4}\E[|\bm{X}_1|^3]}{\sqrt{n}}, \quad \text{ for all }n,d \in \N.
\end{equation} 

A central goal in high-dimensional Berry--Esseen bounds is to control the dependence on dimension $d$ while maintaining an optimal (or near-optimal) rate in $n$. For i.i.d.\ summands, an extensive amount of literature exists for high-dimensional problems with bounds explicit in dimension, see Section~\ref{sec:connec_lit} and references therein. In the i.i.d.\ setting, stronger bounds than~\eqref{eq:hyper_rect_defn} in $d$ exist under $d_\HR$; however, such bounds all require strict assumptions, such as the observations being bounded, sub-Gaussian, sub-exponential, or having a $\log$-concave density, see Section~\ref{sec:connec_lit}. Outside of such strict assumptions, the best bound is in the setting of~\eqref{eq:hyper_rect_defn} and depends on $d$ only through $d^{1/4}$ and $\E[|\bm{X}_1|^3]$.

\subsection{Univariate non-linear functionals of Gaussians}\label{sec_non-lin_univ_intro}
The study of \emph{partial sums of
non-linear functionals of Gaussian processes} play a prominent role in the literature due to their fundamental role in many applications. This framework considers functionals of the form
$$S_n = \frac{1}{\sqrt{n}}\sum_{k=1}^n \Phi(G_k),\quad  \text{ for all }n \in \N,$$ where $\Phi:\R\to\R$ is measurable and non-linear, and $(G_k)_{k\in \N}$ is a centred stationary Gaussian sequence with autocovariance function $\rho$. If 
$\Phi\in L^2(\gamma,\R)$ (where $\gamma(\D x)\coloneqq (2\pi)^{-1/2}e^{-x^2/2}\D x$) has Hermite rank $r\ge 2$, i.e.\ has no affine part (meaning that $\Phi$ has no constant nor linear part), and $\sum_{k = 0}^\infty |\rho(k)|^2 <\infty$, then Breuer--Major theorem~\cite{MR716933} shows that $S_n \cid \sigma Z$ as $n \to \infty$ with $Z \sim \mathcal{N}(0,1)$, for a $\sigma^2<\infty$ which depends on the Hermite expansion of $\Phi$ (see Section~\ref{sec:preliminaries}). Under minimal regularity assumptions, a quantitative bound for the rate of convergence exists for $d_\Kol$~\cite[Thm~1.2]{MR4151212} (see also~\cite{MR3978683}) in the setting of the Breuer--Major theorem from above, and is given by
\begin{equation}
    d_\Kol\left(\frac{S_n}{\sqrt{\var(S_n)}},Z\right)
    \le \frac{C}{\sqrt{n}}\left(\|\rho_n\|_{\ell^1(\Z)}^{1/2}+ \|\rho_n\|_{\ell^{4/3}(\Z)}^2\right),
\end{equation} for all $n \ge 1$, where $\|\cdot\|_{\ell^p(\Z)}$ and $\rho_n$ are defined in Section~\ref{sec:preliminaries} below and $C>0$ is a constant independent of $n$. Note that~\cite[Thm~1.2]{ MR4151212} is stated for the total variation distance, which directly implies a bound for the Kolmogorov distance. Moreover, another strong general bound for the Kolmogorov distance is found in~\cite[Thm~2.1(3)]{MR2770907}, however, this bound has a less explicit dependence on $n$. We also refer to Theorem~\ref{thm:main_theorem} below, which also holds in the one-dimensional case and is completely explicit. Berry--Esseen type bounds characterise the finite sample error bound in the Gaussian approximation, making this bound crucially important in the literature and applied areas of statistics.

\subsection{High-dimensional non-linear functionals of Gaussians}\label{sec:high_dim_non_lin_subsec}
The fluctuation of high-dimensional non-linear functionals of Gaussian stochastic processes plays a significant role in various areas of statistics, see Section~\ref{sec:applications_MoM_ChaFct_fdd}. In fact, statistical procedures often rely on the Gaussian approximation; therefore, it is crucial to assess its finite sample accuracy. In particular, in the case of a test involving a high number of parameters, it is often unclear how much data is necessary for a Gaussian approximation to be satisfiable. This paper provides explicit quantitative bounds for time series data, allowing for strong correlations between observations in a setting with few known results.
Thus, quantitative bounds are obtained for the Gaussian approximation in the case of partial sums of \emph{high-dimensional, non-linear functionals} $\bm{S}_n\in \R^d$ in the setting 
\begin{equation}\label{eq:defn_S_n}
    \bm{S}_n\coloneqq \frac{1}{\sqrt{n}}\sum_{k=1}^n \Phi(G_k), \, \text{ for }n \in \N, \text{ where }\Phi:\R\to\R^d, 
\end{equation} and $(G_k)_{k\in \N}$ is a centred stationary Gaussian sequence with autocovariance function $\rho$. Throughout the paper, we assume that $\Phi$ is measurable with $\E[|\Phi(G_1)|^2]<\infty$. We obtain explicit non-asymptotic bounds for $d_\HR(\bm{S}_n,\bm{Z}_n)$, $d_\CD(\bm{S}_n,\bm{Z}_n)$ and $d_\mW (\bm{S}_n,\bm{Z}_n)$ between $\bm{S}_n$ and $\bm{Z}_n \sim \mathcal{N}_d(\bm{0},\bm{\Sigma}_n)$, for a general covariance matrix $\bm{\Sigma}_n$. The bounds constructed in this paper are explicit in the following factors: \textbf{I}) the numbers of samples $n$; \textbf{II}) the dimension $d$ of the codomain of $\Phi$; \textbf{III}) the dependence between the observations, described by the autocorrelation function $\rho$; \textbf{IV}) the regularity of $\Phi$, described by a multivariate parameter $\bm{\theta}$; \textbf{V}) the level of degenerativeness of the law of $\bm{Z}_n$ on $\R^d$ described by the minimum eigenvalue $\sigma_*^2$ of the correlation matrix $\cor(\bm{Z}_n)$; \textbf{VI}) how much $\bm{\Sigma}_n$ vary from $\cov(\bm{S}_n)$, measured by the max norm on the difference (under $d_\HR$) and measued by the Hilbert--Schmidt norm (under $d_\CD$ and $d_\mW$). The strengths and differences between the two distances are discussed below in Section~\ref{sec:main_results}. Here, we see that $d_\HR(\bm{S}_n,\bm{Z}_n)$ depends sub-polynomially on dimension $d$, under certain regularity assumptions of $\Phi$, and otherwise depends polynomially on $d$. We will also see that $d_\CD(\bm{S}_n,\bm{Z}_n)$ and $d_\mW(\bm{S}_n,\bm{Z}_n)$ always depends polynomially on $d$, and that this polynomial grows faster in $d$ than the one attained under $d_\HR$, when the main result Theorem~\ref{thm:main_theorem} is applicable. This leads us to an interesting property. Indeed, the regularity of $\Phi$ determines the dimensional dependence under $d_\HR$, meaning that a smoother function $\Phi$ results in a better dimensional dependence. Such a behaviour is \emph{not} evident under $d_\CD$ nor $d_\mW$, which has a fixed dimensional dependence no matter the regularity of $\Phi$. On the other hand, we see that the dependence on the number of observations $n$ is optimal under $d_\CD$ and $d_\mW$ (i.e.\ $1/\sqrt{n}$ for i.i.d.), where it is slightly sub-optimal for $d_\HR$ (e.g.\ $\log(n)/\sqrt{n}$ for i.i.d.).

This paper establishes the first explicit quantitative bounds for high-dimensional non-linear functionals of Gaussian processes, capturing both high-dimensionality and strong dependence. Our main result provides a general framework that also covers many i.i.d.\ settings, which are discussed further in Section~\ref{subsec:generality_of_setting} below. Beyond the i.i.d.\ setting, the framework~\eqref{eq:defn_S_n} also encompasses a broad class of partial sums of highly dependent $d$-dimensional vectors. In the latter case, which is the main topic of this paper, there is a limited understanding of the Gaussian fluctuation of the $\bm{S}_n$. When observations have strong dependence, it is out of reach to obtain results for general summands, and we therefore specialise to the setting of subordinated Gaussian observations $\Phi(G_k)$. 
As special cases of~\eqref{eq:defn_S_n}, this paper also considers important examples in Section~\ref{sec:applications_MoM_ChaFct_fdd}: method of moments, empirical characteristic functions, empirical moment generating functions and functional limit theorems. All these cases lie in the high-dimensional setting $d\gg 1$. Note that existing i.i.d.\ results do \emph{not} apply in our setting~\eqref{eq:defn_S_n} due to the (potential) strong dependence between observations. Moreover, the techniques used in the i.i.d.\ settings do not work for the dependent cases, and therefore, alternative methods are needed. See Subsection~\ref{sec:methodology} for a discussion of the methodology of this paper.

\subsection{Preliminaries}\label{sec:preliminaries}

Let $\N = \{1,2,\dots\}$, $\N_0=\{0,1,\dots\}$ and $\Z=\{\dots, -1,0,1,\dots\}$. For a matrix $\bm{A}=(\bm{A}_{i,j})_{1 \le i,j \le d} \in \R^{d \times d}$, we define $\sigma_*^2(\bm{A})$ as the smallest eigenvalue of $\bm{A}$, and $\sigma_\dag^2(\bm{A})$ denotes the largest eigenvalue of $A$. Moreover, we also define $
\underline{\sigma}^2(\bm{A})\coloneqq \min_{1 \le i \le d} \bm{A}_{i,i}$ and $
\overline{\sigma}^2(\bm{A})\coloneqq \max_{1 \le i \le d} \bm{A}_{i,i}$. In a setting where $\bm{\Sigma}$ is fixed, we will use the notation $\underline{\sigma}^2$, $\overline{\sigma}^2$, $\sigma_*^2$ and $\sigma_\dag^2$, otherwise, we use the more explicit form $\underline{\sigma}^2(\bm{\Sigma})$, $\overline{\sigma}^2(\bm{\Sigma})$, $\sigma_*^2(\bm{\Sigma})$ and $\sigma_\dag^2(\bm{\Sigma})$.

Define the Hilbert--Schmidt norm $\| \cdot \|_{\mathrm{H.S.}}$, for any matrix $\bm{A} \in \R^{d \times d}$ by $\|\bm{A}\|_{\mathrm{H.S.}}^2 = \langle \bm{A},\bm{A} \rangle_{\mathrm{H.S.}}= \sum_{1 \le i,j \le d} \bm{A}_{i,j}^2$. For a sequence $s=(s_k)_{k \in \Z} \subset \R$ we define $\|s\|_{\ell^p(\Z)}^p \coloneqq \sum_{k \in \Z}|s_k|^p$ for $p \in \N$, and let $\ell^p(\Z)$ be the space of sequences $s$ such that $\sum_{k\in \Z}|s_k|^p<\infty$. Define $\log_+(x)\coloneqq |\log(x)|\vee 1$ for all $x \in (0,\infty)$. We say a function $L$ is \emph{slowly varying} at $\infty$, denoted $L \in \mathrm{SV}_\infty$, if $L$ is positive in a neighbourhood of $\infty$ and $\lim_{k \to \infty} L(ck)/L(k)=1$ for all $c\in(0,\infty)$.

Let $(G_k)_{k \in \N}$ be a centred stationary Gaussian sequence and let $\rho$ and $\rho_n$ denote the autocovariance and truncated autocovariance functions, respectively, given by 
\begin{equation}
    \rho(k)\coloneqq \E[G_{|k|+1}G_{1}]\qquad \text{and} \qquad \rho_n(k)\coloneqq |\rho(k)|\1_{\{|k|<n\}}, \quad \text{ for }k \in \Z \text{ and }n \in \N. 
\end{equation}
Throughout the paper, we assume $\rho(0)=1$ implying $G_k \sim \mathcal{N}(0,1)$ for all $k \in \N$. We will let $\gamma$ denote the standard Gaussian measure $\gamma$ on $\R$ given by $\gamma(\D x)\coloneqq (2\pi)^{-1/2}e^{-x^2/2}\D x$. For $d\geq 1$ let $ L^p(\gamma,\R^d)$ be the set of measurable functions $\Phi:\R\to\R^d$ satisfying $\int_\R |\Phi(x)|^p \gamma(\D x) <\infty$. For $\varphi \in L^2(\gamma,\R)$, the process $(\varphi(G_k))_{k \in \N}$ is called a \emph{Gaussian subordinated process}, and allows for a strong dependence structure. 
In this paper, the focus is on the Gaussian fluctuations of multivariate subordinated Gaussian processes given in~\eqref{eq:defn_S_n}, i.e.\ $\bm{S}_n\coloneqq n^{-1/2} \sum_{k=1}^n \Phi(G_k)$, where $\Phi\in L^2(\gamma, \R^d)$. The coordinates of $\Phi$ are denoted as $\Phi=(\varphi_1,\dots, \varphi_d)$. For all $q \in \N_0$, let  $H_q$ denote the $q$th \emph{Hermite polynomial} given by $H_q(x) \coloneqq (-1)^qe^{x^2/2}\frac{d^q}{dx^q} e^{-x^2/2}$ for all $x \in \R$ and $q \ge 1$, and $H_0(x) \coloneqq 1$. Hence, $H_1(x)=x$, $H_2(x)=x^2-1$ etc. Any function $\varphi_i \in  L^2(\gamma,\R)$ has a \emph{Hermite expansion}, and in particular, for all $i=1,\dots, d$ there exist unique sequences $(a_{i,q})_{q\in \N_0}$ with $\sum_{q=0}^\infty q! a_{i,q}^2 <\infty$ such that  
\begin{equation}\label{eq:defn_Hermite_expansion}
    \varphi_i(x)=\sum_{q=0}^\infty a_{i,q}H_q(x), \quad \text{ for all }x \in \R.
\end{equation}  
The \emph{Hermite rank} of $\varphi_i$, denoted by $m_i\ge 0$, is defined as $m_i\coloneqq \inf\{q \ge 0:a_{i,q} \ne 0\}$. Hence, $\varphi_i$ has \emph{Hermite rank} $2$ if $a_{i,0}=a_{i,1}=0$, corresponding to $\varphi_i$ not having an affine part.

In the last part of this subsection, the main functions and constants needed for the paper's main results are introduced. For $\beta \in [1/2,1]$ and $\kappa\in\R$, define
\begin{equation}\label{eq:min_dim_dependence}
    \psi_{\beta,\kappa}(d)\coloneqq
        \log_+^{1/(2\beta)}(d)
        e^{r \log_+^{1/(2\beta)}(d)}, 
    \quad \text{ for all }d \in \N,
\end{equation} 
where $r=r_{\beta,\kappa}$ is given by 
    \begin{equation}\label{defn:r_constant-2}
        r\coloneqq 2e^{1/(2e)}\beta e^{(\kappa+\log(24)/2+5/(4e))/\beta}2^{1/(2\beta)}.
    \end{equation}
For all $\beta \in (1/2,1]$ and $\kappa \in \R$, $\psi_{\beta, \kappa}$
is asymptotically larger than any power $p>0$ of the  logarithm function but asymptotically smaller than any power functions with exponent $\epsilon>0$, i.e.\    
\begin{equation}\label{eq:bounds_psi}
\log^p(d) \le \psi_{\beta,\kappa}(d) \le  d^\epsilon, \quad \text{ for all large enough }d.
\end{equation} 
Equation~\eqref{eq:bounds_psi} follows directly, by noting that $d^\epsilon=e^{\epsilon\log(d)}$ and $\log^p(d)=e^{p\log(\log(d))}$ for all $d \ge 2$. In particular, equation~\eqref{eq:bounds_psi} shows that $\psi_{\beta,\kappa}$ has sub-polynomial growth in $d$ for $\beta>1/2$.
On the other hand, for $\beta=1/2$, 
$\psi_{\beta,\kappa}(d)=\log(d)d^{r}$ for all $d \ge 2$. Define 
\begin{equation}\label{eq:defn_Upsilon_constant}
 \Upsilon \coloneqq \log\big(W(e^{-1 + 1/(2 e)})/e^{1/(2e)}\big)/2-\log(24)/2-5/(4e)=-2.709\ldots, 
\end{equation}
where $W(x)$ is the product log function (also called the Lambert $W$-function).

\section{Main Results}\label{sec:main_results}

The ensuing theorem obtains quantitative bounds for Gaussian fluctuations of high-dimensional non-linear functionals. For simplicity, all results in this section are stated in the setting where $\bm{\Sigma}_n=\cov(\bm{S}_n)$; however, all results of the section have analogous results for the case where $\bm{\Sigma}_n$ is a general invertible covariance matrix. In the general case, an additional term appears which increases the complexity of the bound, see Theorem~\ref{thm:main_mult_clt_techncial_thm} and Theorem~\ref{thm:Ext_to_convex_dist}.

\begin{theorem}\label{thm:main_theorem}
    For $d,n \in \N$ let $\bm{S}_n$ be a $d$-dimensional random vector given by~\eqref{eq:defn_S_n}. Assume $\bm{\Sigma}_n\coloneqq  \cov(\bm{S}_n)$ is invertible, and let $\bm{Z}_n \sim \mathcal{N}_d(\bm{0},\bm{\Sigma}_n)$ and $\sigma_*^2 = \sigma_*^2(\cor(\bm{S}_n))$. For $i=1,\dots, d$ suppose that $\varphi_i\in L^2(\gamma,\R)$ has Hermite rank at least 2 and Hermite coefficients $(a_{i,q})_{q \ge 2}$. Assume there exists $\bm{\theta}=(\beta,\kappa,c)\in [1/2,1]\times \R \times (0,\infty)$ such that
\begin{equation}\label{eq:main_assump_a_i,q}
        |a_{i,q}| \le c \frac{e^{\kappa q} }{(q!)^{\beta}}, \qquad \text{ for all }q \ge 2 \text{ and }i=1\ld  d.
    \end{equation} If $\beta=1/2$ assume additionally that $\kappa <  \Upsilon$ in the case of~\eqref{eq:main_result_inequality} and $\kappa < -\log(3)/2$ in the case of~\eqref{eq:main_result_inequality_d_C}. Then, there exists a finite constant $C_{\bm{\theta}}$, only depending on 
$\bm{\theta}$, such that, for all $n,d \in \N$,
\begin{align}
    d_\HR(\bm{S}_n,\bm{Z}_n)
    &\le 
    C_{\bm{\theta}} \psi_{\beta,\kappa}(d) n^{-1/2}\|\rho_n\|_{\ell^1(\Z)}^{3/2} \log_+\left(n^{-1/2}\|\rho_n\|_{\ell^1(\Z)}^{3/2}\right)  \frac{\log_+(\sigma^2_*)}{\sigma^2_*}, \label{eq:main_result_inequality}\\
    d_\CD(\bm{S}_n,\bm{Z}_n) 
    &\le 
    C_{\bm{\theta}} d^{65/24}   \frac{\|\rho_n\|_{\ell^1(\Z)}^{3/2}}{\sqrt{n}} \frac{1}{(\sigma_*^2)^{3/2}}, \, \text{ and }\, d_\mW(\bm{S}_n,\bm{Z}_n) \le C_{\bm{\theta}} d^{3/2} \frac{\|\rho_n\|_{\ell^1(\Z)}^{3/2}}{\sqrt{n}}  \frac{\sigma_\dag(\bm{\Sigma}_n)}{\sigma_*^2(\bm{\Sigma}_n)}.\label{eq:main_result_inequality_d_C}
\end{align}
\end{theorem} 

We apply the general bounds~\eqref{eq:main_result_inequality} \&~\eqref{eq:main_result_inequality_d_C} from Theorem~\ref{thm:main_theorem} to obtain quantitative bounds in various statistical settings in Section~\ref{sec:applications_MoM_ChaFct_fdd}, such as method of moments, empirical characteristic functions, empirical moment generating functions and functional limit theorems. Note that~\eqref{eq:main_result_inequality} and~\eqref{eq:main_result_inequality_d_C} do \emph{not} require $d$ to be fixed, but hold for all $n,d \in \N$. This means that one may take $n$ and $d$ to infinity independently or dependently as desired, and the bounds still hold; see Corollary~\ref{cor:main_theorem_n_dep_d} for an example.

\begin{remark}\label{rem:main_inf_factors}
There are four main factors that influence the bounds in~\eqref{eq:main_result_inequality} \&~\eqref{eq:main_result_inequality_d_C}:

\noindent\nf{\textbf{1.}} \emph{The universal constant} $C_{\bm{\theta}}>0$ depending only on $\bm{\theta}$ from~\eqref{eq:main_assump_a_i,q}.

\noindent\nf{\textbf{2.}} \emph{The dimensional dependence}, described in the case of $d_\HR$ by the function $\psi_{\kappa,\beta}(d)$ describing the dimensional dependence, which is sub-polynomial for $\beta\in (1/2,1]$, and polynomial for $\beta=1/2$. For $d_\CD$ (resp. $d_\mW$), the dimensional dependence is always $d^{65/24}$ (resp. $d^{3/2}$), and does not depend on $\beta$. Hence, choosing a smoother function $\Phi$, does not in the case of~\eqref{eq:main_result_inequality_d_C} improve the $d$-dependence, as was the case of~\eqref{eq:main_result_inequality} under $d_\HR$.

\noindent\nf{\textbf{3.}} \emph{The dependence 
on number of observations $n$} and the dependence of the underlying Gaussian process, is controlled by $n^{-1/2}\|\rho_n\|_{\ell^1(\Z)}^{3/2}$. For~\eqref{eq:main_result_inequality}, the lag-dependence is controlled by the function $x\mapsto x\log_+(x)$ applied to $n^{-1/2}\|\rho_n\|_{\ell^1(\Z)}^{3/2}$, and in the case of~\eqref{eq:main_result_inequality_d_C}, it is purely determined by $n^{-1/2}\|\rho_n\|_{\ell^1(\Z)}^{3/2}$.

\noindent\nf{\textbf{4.}} \emph{How close $\bm{\Sigma}_n$ is from being singular}, measured measured by how close $\sigma_*^2$ is to $0$ as a function of $n$ and $d$. If $\sigma_*^2$ is bounded uniformly away from $0$, then $\sigma_*^{2}$ reduces to a universal constant, and hence components containing this quantity can be taken into the universal constant $C_{\bm{\theta}}$. 
\end{remark}

From Remark~\ref{rem:main_inf_factors} it is evident that there is a trade-off when choosing between $d_\HR$ and $d_\CD$ or $d_\mW$. Indeed, where $d_\HR$ always leads to a better dimensional dependence (as discussed in Remark~\ref{rem:min_conv_rate_a_i_q} below), it has a sub-optimal dependence on $n$ via an additional multiplicative $\log$-term. However, $d_\CD$ and $d_\mW$ have optimal $n$-dependence but a worse dimensional dependence. The proof of~\eqref{eq:main_result_inequality} from Theorem~\ref{thm:main_theorem} relies heavily on the results from~\cite{MR4312842}, where the proof of~\eqref{eq:main_result_inequality_d_C} depends mainly on~\cite[Thm~1.2]{MR4488569}. The ideas behind the proofs are explained in greater detail in Section~\ref{sec:methodology}. 

For the ensuing theorem and corollary, define $\underline{\underline{\sigma}}=((\underline{\sigma}\wedge 1)\1_{\{d=1,2\}}+\underline{\sigma} \,\1_{\{d >2\}})$, $\overline{\overline{\sigma}}=((\overline{\sigma}\vee 1)\1_{\{d=1,2\}}+\overline{\sigma}\, \1_{\{d >2\}})$, and $\wt\sigma_*^2=((\sigma_*^2\wedge 1)\1_{\{d=1,2\}}+\sigma_*^2 \,\1_{\{d >2\}})$. In the ensuing theorem and corollary, we will apply these functions to $\bm{\Sigma}$, given in the statements.
\begin{theorem}\label{thm:fixed_d}
    For $d,n \in \N$ let $\bm{S}_n$ be a $d$-dimensional random vector given by~\eqref{eq:defn_S_n}. For $i=1,\dots, d$ suppose that $\varphi_i\in L^2(\gamma,\R)$ has Hermite rank $m \ge 2$ and Hermite coefficients $(a_{i,q})_{q \ge m}$. Assume there exists $\bm{\theta}=(\beta,\kappa,c)\in [1/2,1]\times \R \times (0,\infty)$ such that~\eqref{eq:main_assump_a_i,q} holds, with the additional assumption for $\beta=1/2$ that $\kappa <  \Upsilon$ under $d_\HR$ and $\kappa < -\log(3)/2$ under $d_\CD$. Let $\bm{Z}\sim \mathcal{N}_d(\bm{0},\bm{\Sigma})$, for $\bm{\Sigma}_{i,j}=\sum_{\ell=m}^\infty \ell! a_{i,\ell}a_{j,\ell}\sum_{j \in \Z}\rho(j)^\ell$ for $i,j=1\ld d$, where we assume that $\sum_{k\in \Z} |\rho(k)|^m<\infty$. Then, there exist constants $C_{\bm{\theta}}>0$, only dependent on $\bm{\theta}$, such that for all $n,d \in \N$,
    \begin{align*}
        d_\HR(\bm{S}_n,\bm{Z})&\le 
    C_{\bm{\theta}} \log_+(d)\Delta(\bm{S}_n,\bm{\Sigma})
    \log_+\left(\Delta(\bm{S}_n,\bm{\Sigma})\right)\frac{\log_+\left(\,\overline{\overline{\sigma}} \,\wt\sigma_{*}^2/\,\underline{\underline{\sigma}}\,\right)}{\wt\sigma_*^2}, \, \text{ where}\\
    \Delta(\bm{S}_n,\bm{\Sigma})&=  \frac{\|\rho_n\|_{\ell^1(\Z)}^{3/2}}{\sqrt{n}}\frac{e^{r\log_+^{1/(2\beta)}(d)}}{\log_+(d)}+\sum_{|k| \ge n} |\rho(k)|^m + \sum_{|k|<n} \frac{|k|}{n}|\rho(k)|^m, \\
    d_\CD(\bm{S}_n,\bm{Z}) &\le C_{\bm{\theta}} d^{65/24} \left( \frac{ \|\rho_n\|_{\ell^1(\Z)}^{3/2}}{\sqrt{n}} + \sum_{|k| \ge n} |\rho(k)|^m + \sum_{|k|<n} \frac{|k|}{n}|\rho(k)|^m\right)\big((\sigma_*^2)
    ^{-3/2}+1\big), \quad \text{ and }\\
    d_\mW(\bm{S}_n,\bm{Z}) &\le C_{\bm{\theta}} d^{3/2} \left( \frac{ \|\rho_n\|_{\ell^1(\Z)}^{3/2}}{\sqrt{n}} + \sum_{|k| \ge n} |\rho(k)|^m + \sum_{|k|<n} \frac{|k|}{n}|\rho(k)|^m\right)\frac{\sigma_\dag}{\sigma_*^2}.
    \end{align*}
\end{theorem}

The ensuing corollary shows that the bounds from Theorem~\ref{thm:fixed_d} simplify if $(G_k)_{k \in \N}$ is assumed to have long-range or short-range dependence. A similar result can be constructed under the assumptions from~\eqref{eq:main_result_inequality} \&~\eqref{eq:main_result_inequality_d_C} in Theorem~\ref{thm:main_theorem}.

\begin{corollary}\label{cor:LRD_SRD_fixed cov:mat}
    Assume we are in the setting of Theorem~\ref{thm:fixed_d}, then the following statements hold:

    \noindent{\nf{(i)}} \underline{Short-range dependence:} Assume $|\rho(k)| \le \wt c \,|k|^{-\mu}L(|k|)$ for all $k\in \Z\setminus \{0\}$ for $\mu \ge 1$ and $L \in \mathrm{SV}_\infty$. If $\mu=1$ assume additionally $\|\rho\|_{\ell^1(\Z)}<\infty$. Then there exists $C_{\bm{\theta},\rho,m}\in (0,\infty)$, depending only on $\bm{\theta}$, $\|\rho\|_{\ell^1(\Z)}$, $\mu$, $\wt c$ and $m$, such that for all $n,d \in \N$,
such that
    \begin{align}
    d_\HR(\bm{S}_n,\bm{Z}_n)&\le 
    C_{\bm{\theta},\rho,m} \psi_{\beta,\kappa}(d)
    \frac{\log_+(n)}{\sqrt{n}} \frac{\log_+(\overline{\overline{\sigma}}\, \wt \sigma^2_*/\underline{\underline{\sigma}})}{\wt \sigma^2_*}, \label{eq:cor_main_srd_d_R_fix_sigma}\\
    d_\CD(\bm{S}_n,\bm{Z}_n) &\le C_{\bm{\theta},\rho,m} d^{65/24}  \frac{1}{\sqrt{n}}  \big((\sigma_*^2)^{-3/2}+1\big), \, \text{ and } \, d_\mW(\bm{S}_n,\bm{Z}_n) \le C_{\bm{\theta},\rho,m} d^{3/2}  \frac{1}{\sqrt{n}}  \frac{\sigma_\dag}{\sigma_*^2}.\label{eq:cor_main_srd_d_C_fix_sigma}
    \end{align}
    
\noindent{\nf{(ii)}} \underline{Long-range dependence:} If $|\rho(k)| \le \wt c \,|k|^{-\mu}L(|k|)$ for all $k \in \Z \setminus\{0\}$ for $\mu \in (2/3,1)$ and $L \in \mathrm{SV}_\infty$, then, there exists $C_{\bm{\theta},\wt c,\mu}\in (0,\infty)$, depending only on $\bm{\theta}$, $\wt c$ and $\mu$, 
such that for all $n,d \in \N$,
    \begin{align}
    d_\HR(\bm{S}_n,\bm{Z})\!
    &\le C_{\bm{\theta},\wt c,\mu} 
        \psi_{\beta,\kappa}(d)\frac{L(n)^{3/2}}{n^{(3\mu-2)/2}}\log_+\bigg(\frac{L(n)^{3/2}}{n^{(3\mu-2)/2}}\bigg) \frac{\log_+\left(\,\overline{\overline{\sigma}} \, \wt\sigma_{*}^2/\,\underline{\underline{\sigma}}\,)\right)}{\wt\sigma_*^2},
      \label{eq:cor_main_lrd_d_R_fix_sigma} \\
      d_\CD(\bm{S}_n,\bm{Z}) \!
      &\le C_{\bm{\theta},\wt c,\mu} d^{65/24} \frac{L(n)^{3/2}}{n^{(3\mu-2)/2}} \big((\sigma_*^2)^{-3/2}+1\big), \, \text{ and }\, d_\mW(\bm{S}_n,\bm{Z}) \!
      \le C_{\bm{\theta},\wt c,\mu} d^{3/2} \frac{L(n)^{3/2}}{n^{(3\mu-2)/2}} \frac{\sigma_\dag}{\sigma_*^2}.\label{eq:cor_main_lrd_d_C_fix_sigma}
    \end{align}
\end{corollary}

\begin{remark}
    Note that the case If $|\rho(k)| \le \wt c \,|k|^{-1}L(|k|)$ for all $k \in \Z \setminus\{0\}$ and $L \in \mathrm{SV}_\infty$, then~\eqref{eq:cor_main_lrd_d_R_fix_sigma} and~\eqref{eq:cor_main_lrd_d_C_fix_sigma} still holds, for any $\mu <1$ and $L\equiv 1$. This is the case, since $|\rho(k)| \le \wt c \,|k|^{-1}L(|k|)$ for all $k \in \Z \setminus\{0\}$ implies that $|\rho(k)| \le \wh c \,|k|^{-\mu}$ for all $k \in \Z \setminus\{0\}$.
\end{remark}

To the best of our knowledge, the strong dimensional control in~\eqref{eq:main_result_inequality} from Theorem~\ref{thm:main_theorem} and the $d_\HR$ bound from Theorem~\ref{thm:fixed_d} is new even in the widely studied case where $\Phi(G_k)$ are i.i.d.\ random vectors and outside of restrictive assumptions such as log-concavity, boundedness, sub-Gaussianity or sub-exponentiality (see~\cite{MR4505371,MR4583674,MR4312842}). By Proposition~\ref{prop:Phi_G_recovers_all_dist}, any i.i.d.\ sequence of $d$-dimensional random vectors $(\bm{X}_k)_{k \in \N}$ can be represented as $\bm{X}_k = \Phi(G_k)$ for a suitable function $\Phi:\R\to\R^d$ and an independent standard Gaussian sequence $(G_k)_{k \in \N}$. Thus, Corollary~\ref{cor:LRD_SRD_fixed cov:mat}(i) covers the case of all i.i.d.\ random vectors satisfying Assumption~\eqref{eq:main_assump_a_i,q}. The generality of the class of these i.i.d.\ vectors is discussed in depth in Section~\ref{subsec:generality_of_setting} below. In the i.i.d.\ case (covered by Corollary~\ref{cor:LRD_SRD_fixed cov:mat}(i)) the $n^{-1/2}\log(n)$ dependence in $n$ under $d_\HR$ is slightly sub-optimality, with the optimal dependence being $n^{-1/2}$ which is the case under $d_\CD$. However, using the methods from~\cite{MR4312842}, yields a structure of the form $\Delta \log_+(\Delta)$ (where $\Delta$ contains both $n$ and $d$ dependence), implying that a $\log_+(n)$ term is unavoidable when aiming for sub-polynomial $d$-dependence. Such a $\log$ error in optimality is often expected from the literature and is a tractable trade-off to gain strong control of the dimensional dependence, see e.g.\ Corollary~\ref{cor:main_theorem_n_dep_d} for an example on why this is the case.

Throughout the remainder of the paper, we will focus on the distances $d_\HR$ and $d_\CD$, and the differences which these distances impose. The reason is that both $d_\mW$ and $d_\CD$ are polynomial in $d$, and have very similar upper bounds, since they rely on the same Hilbert--Schmidt norm. Hence, all of the following results on the convex distance, can similarly be constructed for the $1$-Wasserstein distance. Some remarks for Theorem~\ref{thm:main_theorem} are in order. 

\begin{remark}\label{rem:min_conv_rate_a_i_q}
\noindent\nf{(a)} Since $\varphi_i \in L^2(\gamma,\R)$ for $i=1\ld d$, it follows that~\eqref{eq:main_assump_a_i,q} always holds for $\beta=1/2$, $\kappa=0$ and a $c>0$. As Lemma~\ref{lem:c_infty_main_result} shows, the smoothness of $\varphi_i$ influences the behaviour of $|a_{i,q}|$, and thus the choice of $\bm{\theta}$.

\noindent\nf{(b)} In Theorem~\ref{thm:main_theorem}, and throughout the paper, we assume that $\bm{\Sigma}_n$ is invertible (equivalent to $\sigma_*^2(\bm{\Sigma}_n)>0$). This assumption is standard in the literature when considering Gaussian fluctuations in the multivariate setting, see e.g.~\cite{MR4488569,MR4312842}. 

\noindent\nf{(c)} We want to emphasise again that the dimensional dependence is crucially determined by $\beta$ under $d_\HR$ in a type of critical value behaviour. Indeed, when $\beta=1/2$ then $\psi_{\beta,\kappa}(d)$ bahaves like $\log(d)d^r$ for $r$ given in~\eqref{defn:r_constant-2} which is determined by $(\beta,\kappa)$. We note that for all $\kappa<\Upsilon$, that $r<3/2<65/24$, and hence $\psi_{\beta,\kappa}(d)$ has a better dimensional dependence than $d^{65/24}$ from $d_\CD$ and $d^{3/2}$ from $d_\mW$. Moreover, when $\beta>1/2$, it follows that $\psi_{\beta,\kappa}(d)$ is sub-polynomial and hence grows slower than any power-function, especially slower than $d^{65/24}$ and $d^{3/2}$. Hence, in all cases where the bound~\eqref{eq:main_result_inequality} from Theorem~\ref{thm:main_theorem} is applicable, it will have a better dimensional dependence (i.e.\ grow slower in $d$) than the bound under $d_\CD$ and $d_\mW$, given in~\eqref{eq:main_result_inequality_d_C}. When $\beta=1/2$, we assume $\kappa<-\log(3)/2$ under $d_\CD$ and $d_\mW$, which is a slightly weaker assumption than $\kappa<\Upsilon<-\log(3)/2$ (needed under $d_\HR$). Hence, less regularity is needed of $\Phi$ to apply bounds under $d_\CD$ and $d_\mW$. However, when we can apply bounds for $d_\HR$, this has a better dimensional dependence, since $r<3/2$. In the case $\beta=1/2$, we would require $\kappa=(\log(65/1152)-3/e)/2>\Upsilon$ for $r=65/24$ and $\kappa=-(\log(32)+3/e)/2>\Upsilon$ for $r=3/2$.

\noindent\nf{(d)} Concrete examples of $\bm{\theta}=(\beta,\kappa,c)$ and $\psi_{\beta,\kappa}(d)$, are given in Sections~\ref{sec:applications_MoM_ChaFct_fdd} \&~\ref{sec:ext_results_examples} below. Here, we construct statistical examples such as the method of Moments, empirical and moment-generating functions and finite-dimensional convergence in Breuer--Major theorems. Moreover, examples for both $\beta=1/2$ and $\beta=1$ are provided, thereby establishing concrete examples that yield both sub-polynomial bounds and polynomial bounds under $d_\HR$. 
\end{remark}

\subsection{Statistical Applications}\label{sec:applications_MoM_ChaFct_fdd}

In statistics, many methods base their conclusions on Gaussian approximations. This applies, in particular, to the generalised method of moments, a key method used for estimation and hypothesis testing. Suppose that $G_1 \ld G_n$ is a stationary real-valued time series described by a statistical model $\{\p_\xi:\xi\in \Xi\}$. To estimate $\xi$, likelihood-based methods are often infeasible, and the generalised method of moments is the main methodology. 
This methodology relies on estimating equations, that is, to find a class of functions $\Phi:\R\to\R^d$ for $d\geq 1$ (depending on $\xi$) such that $\E_{\xi}[ \Phi(G_1) ] =0$ if and only if $\xi = \xi_0$, with $\xi_0$ denoting the true parameter of the model. To  construct a statistical test for the null hypotheses H$_0$: $\xi =\xi_0$ against the alternative hypotheses H$_1$: $\xi\neq \xi_0$, we 
reject the H$_0$ if $\bm{S}_n\coloneqq n^{-1/2}\sum_{k=1}^n \Phi(G_k)$ is large is an appropriate sense.

Assuming that $\bm{S}_n$ is close in distribution to $\bm{Z}_n\sim \mathcal{N}_d(0,\bm{\Sigma})$ for a covariance matrix $\bm{\Sigma}$, an approximate statistical test with significance level $\alpha_0>0$ can be constructed, by choosing a Borel-measurable acceptance region $A_\alpha\subset\mathbb{R}^d$ such that $\mathbb{P}(\mathbf{Z}\in A_\alpha)=\alpha_0$, and accept $H_0$ if $\mathbf{S}_n\in A_\alpha$. Even with a central limit theorem at hand, this raises the question of the \emph{finite-sample error}
\[
\text{Error}_{\text{f.s.}}(A_\alpha)\coloneqq \big|\mathbb{P}_{\xi_0}(\mathbf{S}_n\in A_\alpha)-\alpha_0\big|,
\]
and how to bound it uniformly over different classes of acceptance regions, i.e.\ $\HR$ and $\CD$.

As motivated by the structure of the paper, we have two natural classes in mind, which typically arise in applications:
\begin{itemize}[leftmargin=2.5em]
    \item[\nf{(i)}] \emph{Axis-aligned hyperrectangles.} The structures of many high-dimensional problems are product-like or axis-aligned, and the main interest of such structures lies in marginal and cumulative discrepancies. Hence, such procedures would use acceptance regions given by $A_\alpha=\{\mathbf{x}\in\mathbb{R}^d:\ \mathbf{v}_l\le \mathbf{x}\le \mathbf{v}_u\}$ for thresholds $\mathbf{v}_l,\mathbf{v}_u\in\mathbb{R}^d$ satisfying $\mathbb{P}(\mathbf{Z}\in A_\alpha)=\alpha_0$. Hence, from the definition of $d_\HR$ in~\eqref{eq:d_R_and_d_C_defn}, we have the uniform bound
\begin{equation}\label{eq:fs-hr}
\text{Error}_{\text{f.s.}}(A_\alpha)\le d_{\HR}(\mathbf{S}_n,\mathbf{Z}).
\end{equation}
\item[\nf{(ii)}] \emph{Convex sets.} For geometric procedures, such as quadratic-form and norm-based procedures, we use the acceptance regions $A_\alpha\in \CD$, chosen so that $\mathbb{P}(\mathbf{Z}\in A_\alpha)=\alpha_0$. Classic geometric examples include ellipsoids
$A_\alpha=\{\bm{x} \in \R^d:\ \mathbf{x}^\top \mathbf{W}\,\mathbf{x}\le r_\alpha\}$ (where $\bm{W}$ is a symmetric positive definite matrix) and Euclidean balls $A_\alpha=\{\bm{x} \in \R^d:\ \|\mathbf{x}\|\le r_\alpha\}$. Hence, from the definition of $d_\CD$ in~\eqref{eq:d_R_and_d_C_defn}, we have the uniform bound 
\begin{equation}\label{eq:fs-cvx}
\text{Error}_{\text{f.s.}}(A_\alpha)\le d_{\CD}(\mathbf{S}_n,\mathbf{Z}).
\end{equation}
\end{itemize}

Bounds for $d_{\CD}$ yield finite-sample guarantees for \emph{all} convex acceptance regions (including hyperrectangles), while bounds tailored to $d_{\HR}$ can be sharper and have a better dimensional dependence for box-type procedures. Hence, if your acceptance region or statistic is axis-aligned or CDF-based, $d_{\HR}$ typically yields sharper finite-sample bounds and better rates in $d$. If your region is convex (e.g., balls, ellipsoids, or norms), $d_{\CD}$ is the natural choice, since it aligns with the geometry of the test and is robust to reparametrisations, at the cost of heavier dimensional dependence. 

Another important distinction is in the VC dimension. Axis-aligned rectangles have finite VC dimension, enabling uniform convergence control, standard bootstraps for copulas, and tractable smoothing/anti-concentration arguments~\cite{vdVW1996,Dudley1999}. In contrast, the class of \emph{all} convex sets has infinite VC dimension in $\R^d$ for $d\ge2$, which obstructs tidy empirical-process control and breaks the rectangle-specific smoothing in~\cite{MR4312842}.

As is evident from Theorem~\ref{thm:main_theorem}, the product structure under $d_{\HR}$ allows for sub-polynomial dependence on $d$ even under strong dependence. For $d_{\CD}$ and $1$-Wasserstein $d_\mW$, current tools do not yield sub-polynomial dependence in $d$. For this to be the case, one would have to extend the results in~\cite{MR4312842}, which is infeasible (see further discussion in Section~\ref{sec:methodology}). Nonetheless, we obtain high-dimensional quantitative bounds for $d_{\CD}$ under possible strong dependence, with dimensional growth $d^{65/24}$ for all $\beta \in [1/2,1]$. Finally, recall that regularity/smoothness influences the dimensional dependence under $d_\HR$ via $\psi_{\kappa,\beta}(d)$, where increasing smoothness implies better dimensional dependence, a feature not present for $d_\CD$.

In the following subsections, we apply both~\eqref{eq:fs-hr} and~\eqref{eq:fs-cvx} to key special cases from the literature, to derive explicit bounds on the finite sample error when the time series $G_1,\dots, G_n$ is Gaussian under the null hypothesis. For $d_\HR$, the statistical functionals considered in this section are closely related to fundamental statistical tests, such as the Baringhaus--Henze--Epps--Pulley (BHEP) test~\cite{MR980849,MR725389}, as discussed in Section~\ref{sec:characteristic_func}. For geometric procedures and tests, a classical example is the Method of Moments test, such as the Shenton--Bowman test~\cite{MR381079}, which is a test based on Euclidean balls, with $d_{\CD}$ providing the corresponding guarantees, which is covered in Section~\ref{sec:method_moments} below.

\subsubsection{Method of Moments}\label{sec:method_moments}
Introduced by Karl Pearson in 1894, the \emph{method of moments} has been a fundamental statistical method. Since all moments uniquely characterise the Gaussian distribution, one application of this method is testing for Gaussianity. Moreover, because the first $n$ moments, cumulants, and Hermite polynomials are in one-to-one correspondence, they are often used interchangeably in various statistical tests, such as the Shenton--Bowman test and related methods~\cite{MR381079,AMENGUAL2024,MR2067685,DECLERCQ1998101}. Indeed, if $X$ is a centred random variable with finite moments, then  
\begin{equation}\label{ljsldjflsdj}
    X\sim \mathcal{N}(0,1)\qquad \text{ if and only if }\qquad \E[H_q(X)]=0 \text{ for all }q\geq 2,
\end{equation}
 cf.\ \cite[Thm~2.2.1]{MR2962301} \&~\cite[Prop.~2]{MR1691731}. To test for marginal Gaussianity, we test the first $d$ equations on the right-hand side of~\eqref{ljsldjflsdj}, using the setting of~\eqref{eq:defn_S_n}, with $\bm{S}_n$ given as
 \begin{equation}\label{eq:MOM_defn_S_n} 
     \bm{S}_n = \frac{1}{\sqrt{n}} \sum_{k=1}^n \left(\frac{H_2(G_k)}{\sqrt{2!}}, \dots, \frac{H_{d+1}(G_k)}{\sqrt{(d+1)!}}\right),
 \end{equation} where $G_1\ld G_n$ follows a stationary Gaussian sequence. All coordinates of $\bm{S}_n$ are standardised to have zero mean and unit variance, and $\cov(\bm{S}_n)$ is a diagonal matrix. 

\begin{corollary}\label{cor:method_moments_hermite_var} 
Let $\bm{S}_n$ be given in \eqref{eq:MOM_defn_S_n} where 
 $(G_k)_{k \in \N}$ is stationary Gaussian sequence. Suppose   $\bm{Z}_n\sim \mathcal{N}_d(\bm{0},\bm{\Sigma}_n)$ with 
 $\bm{\Sigma}_n = \cov(\bm{S}_n)$. Then, there exists a universal constant $C$ such that 
   \begin{equation}
d_\HR(\bm{S}_n,\bm{Z}_n) \le d_\CD(\bm{S}_n,\bm{Z}_n)
\le 
C d^{125/24} e^{\log(3) d} n^{-1/2} \|\rho_n\|_{\ell^1(\Z)}^{3/2}, \quad \text{ for all }n,d \in \N.\label{eq:rate_Hermite_variation_Up_d_C}
\end{equation}
\end{corollary}
Note that the underlying Gaussian sequence has a completely general dependence structure in Corollary~\ref{cor:method_moments_hermite_var}, which is more general than the i.i.d.\ setting covered in~\cite[Thm~1.1]{MethodOfMoments}. Compared to~\cite[Thm~1.1]{MethodOfMoments}, it is not surprising that the dimensional dependence is exponential since we even have an exponential lower bound in the i.i.d.\ case, as discussed below in~\eqref{eq:mom_OG}. If we specialize Corollary~\ref{cor:method_moments_hermite_var} to the short-range dependence case, i.e.\ $\|\rho\|_{\ell^1(\Z)}<\infty$, it follows that
\begin{equation}\label{eq:iid_bound_MoM}
    d_\HR(\bm{S}_n,\bm{Z}_n) \le d_\CD(\bm{S}_n,\bm{Z}_n) \le Cd^{125/24} e^{\log(3) d} n^{-1/2},  \quad \text{ for all } n,d\in \N. 
\end{equation} 
Assume that $(G_k)_{k \in \N}$ are i.i.d., $\bm{Z}\sim \mathcal{N}_d(\bm{0},\bm{I}_d)$, and define $C_u \coloneqq 58 e^{3\log(2)/2} /(e^{3\log(2)/2}-1)$ and $C_l \coloneqq (e\pi^32^3)^{-1/4}e^{-3\log(2)/2}$. By~\cite[Thm~1.1]{MethodOfMoments}, it follows for all $d\ge 2$ that there exists a sequence $(N_d)_{d\ge 2}$ in $\N$, such that for all $n \ge N_d$,
\begin{equation}\label{eq:mom_OG}
     C_ld^{-3/4}e^{d \cdot 3\log(2)/2}\, n^{-1/2} \le d_{\HR}(\bm{S}_n,\bm{Z}) \le d_\CD (\bm{S}_n,\bm{Z}) \le C_u  d^{3/4} e^{d\cdot 3\log(2)/2}\, n^{-1/2}.
    \end{equation}
    Note that~\eqref{eq:mom_OG} yields optimal exponential upper and lower bounds, which suggests that~\eqref{eq:iid_bound_MoM} is almost exponentially optimal, since $3\log(2)/2 \approx 1.03972$ and $\log(3) \approx 1.09861$ are close.

\subsubsection{Empirical Characteristic Functions \& Moment-Generating Functions}\label{sec:characteristic_func}

Empirical characteristic functions and moment-generating functions are examples of key statistical functionals used, e.g., in the BHEP test~\cite{MR980849,MR1089501} or~\cite{MR4137748}, respectively. The BHEP test is a characteristic function-based goodness-of-fit test that assesses whether a given sample follows a normal distribution. It works by comparing the empirical characteristic function of the sample with the characteristic function of the hypothesised distribution. A similar approach is used in~\cite{MR4137748} for empirical moment-generating functions. Our next result gives explicit bounds on the Gaussian approximation of empirical characteristic functions ((i) and (ii) below) and moment-generating function ((iii) below), evaluated in the $d$ points $\lambda_1,\dots, \lambda_d$. Let $K\in (0,\infty)$ and $\lambda_i \in (0,K]$ for all $i =1 \ld d$, and define the following functions:
    \begin{itemize}
        \item[\nf{(i)}] $\varphi^{\cos}_i(x)=\cos(\lambda_i x)-e^{-\lambda_i^2/2}$ for all $x \in \R$ and $i=1\ld d$;\label{itemone}
        \item[\nf{(ii)}] $\varphi^{\sin}_i(x)=\sin(\lambda_i x)-\lambda_i xe^{-\lambda_i^2/2}$ for all $x \in \R$ and $i=1\ld d$;
        \item[\nf{(iii)}] $\varphi^e_i(x)=e^{\lambda_i x}-e^{\lambda_i^2/2}(1+\lambda_i x)$ for all $x \in \R$ and $i=1\ld d$.
    \end{itemize}
    The functions (i)-(iii) corresponds to $\cos(\lambda_i x)$, $\sin(\lambda_i x)$ and $e^{\lambda_i x}$, adjusted to have Hermite rank $2$. 

\begin{corollary}
\label{cor:real_part_charac_fun}
Let $\bm{S}_n$ be a $d$-dimensional random vector given by~\eqref{eq:defn_S_n} and assume for $i=1\ld d$ that  $\varphi_i$ is given as in (i)-(iii) above for a $K \in (0,\infty)$. Let $\bm{Z}_n \sim \mathcal{N}_d(\bm{0},\bm{\Sigma}_n)$ with $\bm{\Sigma}_n=\cov(\bm{S}_n)$ being invertible and let $\sigma^2_*=\sigma_*^2(\cor(\bm{S}_n))$. Then, there exist finite constants $C_K,\kappa_K>0$, only depending on $K$, such that
   \begin{align}
        d_\HR(\bm{S}_n,\bm{Z}_n)&\le  C_K \psi_{1,\kappa_K}(d)
        n^{-1/2}\|\rho_n\|_{\ell^1(\Z)}^{3/2} \log_+\left(n^{-1/2}\|\rho_n\|_{\ell^1(\Z)}^{3/2}\right) 
        \frac{\log_+(\sigma^2_*)}{\sigma^2_*}, \quad \text{ and} \nonumber\\
        d_\CD(\bm{S}_n,\bm{Z}_n)&\le  C_K d^{65/24}
        n^{-1/2}\|\rho_n\|_{\ell^1(\Z)}^{3/2} \frac{1}{(\sigma_*^2)^{3/2}}, \quad \text{ for all } n,d \in \N.\label{eq:real_part_charac_fun_d_C}
   \end{align}
   \end{corollary}

Since the function $\psi_{1,\kappa_K}$ is sub-polynomial, Corollary~\ref{cor:real_part_charac_fun} shows that empirical characteristic functions or empirical moment-generating functions have sub-polynomial dimension dependence under $d_\HR$ in case $\sigma_*^2$ is bounded away from zero, and depends on dimension as $d^{65/24}$ under $d_\CD$. Hence, clearly $d_\HR$ yields a better dimensional dependence for the structure imposed by this problem. This result sharply contrasts the exponential dimensional dependence obtained in Corollary~\ref{cor:method_moments_hermite_var} and~\cite[Thm~1.1]{MethodOfMoments} in the case of method of moments, where both $d_\HR$ and $d_\CD$ lead to exponential dimensional dependence. Corollary~\ref{cor:real_part_charac_fun} becomes even more explicit in the short-range dependence case $p=\|\rho\|_{\ell^1(\Z)}<\infty$. Indeed, for all $\epsilon>0$,~\eqref{eq:bounds_psi} yields a constant $C_{K,p,\epsilon}>0$ such that 
\begin{equation}\label{eq:dajshd}
     d_\HR(\bm{S}_n,\bm{Z}_n)\le  C_{K,p,\epsilon}  d^\epsilon\frac{\log_+(n)}{\sqrt{n}} \frac{\log_+(\sigma^2_*)}{\sigma^2_*}, \quad \text{ for all }n,d \in \N,
\end{equation} which is directly comparable to~\eqref{eq:real_part_charac_fun_d_C} whenever $\|\rho_n\|_{\ell^1(\Z)}\le p
<\infty$, which has dimensional dependence as $d^{65/24}$ and depends on $n$ through $1/\sqrt{n}$. Due to the structure of the upper bounds obtained throughout this paper, it is tractable to have sufficient conditions ensuring that $\sigma_*^2=\sigma_*^2(\cor(\bm{S}_n))$ is uniformly bounded away from zero. Such conditions are given in Proposition~\ref{prop:specif_theta_i} below in the setting of Corollary~\ref{cor:real_part_charac_fun}(i).
\begin{proposition}\label{prop:specif_theta_i}
    Assume the setting of Corollary~\ref{cor:real_part_charac_fun} with $\varphi_i=\varphi_i^{\cos}$ given as in~$(i)$. Suppose additionally that 
     $\lambda_i=i^\tau$ for a $\tau \ge 1$ for all $i=1 \ld d$. Then, the following statements hold.
     
\vspace{2mm}
    \noindent{\nf{(i)}} If $\|\rho\|_{\ell^2(\Z)}^2  
        <  \zeta(\tau) \coloneqq \dfrac{2e^{1/2}(1-e^{-1})(1-e^{-2^{2\tau}})}{\frac{e^{2^\tau }}{e^{2^{2\tau}/2}}\left(2+\frac{1}{\tau 2^{\tau-1}(2^\tau-1)}\right)
    +
    \frac{e^{3^\tau}}{e^{3^{2\tau}/2}}\left( 1+\frac{1}{\tau 3^{\tau-1}(3^\tau-1)}\right)}$, then $\inf_{n,d \in \N} \sigma^2_*> 0$. \vspace{2mm}

    \noindent{\nf{(ii)}} Assume $(G_k)_{k \in \N}$ is a fractional Gaussian noise (see Section~\ref{sec:eigenvalues_bound_fGn}) with Hurst parameter $H \in (0,1/2)$. If $\tau\ge 1.302$, then $\inf_{n,d \in \N} \sigma^2_*>0$ for all $H \in (0,1/2)$.
\end{proposition}

\begin{remark}
Assume that $\rho$ satisfies $\sum_{k \in \Z}\rho(k)^2<\infty$. Then we have that $\tau\mapsto \|\rho\|_{\ell^2(\Z)}^2 \zeta(\tau)^{-1}$ is decreasing with $\lim_{\tau \to \infty} \zeta(\tau)^{-1}=0$ (which is shown in the proof of Proposition~\ref{prop:specif_theta_i}(ii)). Hence, we may always choose $\tau$ large enough, so that $\inf_{n,d \in \N}\sigma_*^2(\bm{\Lambda}_n) >0$. 

If $(G_k)_{k \in \N}$ is a fractional Gaussian noise, then $\rho(v) \sim H (2H-1)v^{2H-2}$ as $v \to \infty$ (see~\cite[Prop.~2.8.2]{MR3729426}). Hence, for all $H \in (0,3/4)$, we can pick $\tau$ large enough such that $\inf_{n,d\in \N}\sigma_*^2(\bm{\Lambda}_n) >0$.
\end{remark}

\subsubsection{Finite-Dimensional Convergence in Breuer--Major Theorems}\label{sec:imporve_quanti_CLT_rate}

The univariate Breuer--Major theorem has a natural extension to finite-dimensional convergence. Let $T>0$ and consider a partition $(t_i)_{i=0\ld d}$ of $[0,T]$ where $0=t_0<\cdots<t_d=T$ with $t_i-t_{i-1}\le 1$ for all $i=1\ld d$. If $\varphi \in L^2(\gamma,\R)$ with Hermite rank $m \ge 1$ and $(G_k)_{k \in \N}$ is a stationary Gaussian sequence, then the Breuer--Major theorem~\cite{MR716933} (see also~\cite[Eq.~(16)]{MR4488569}) shows the following. If $\sum_{k \in \Z} |\rho(k)|^m<\infty$, then 
\begin{equation}\label{ljsldfjlsdj}
    \left( \frac{1}{\sqrt{n}}\sum_{k=1}^{\floor{n t_1}}\varphi(G_k)\ld \frac{1}{\sqrt{n}}\sum_{k=1}^{\floor{n t_d}}\varphi(G_k)\right) \cid \mathcal{N}_d(\bm{0},\bm{A}), \quad \text{ as }n \to \infty,
\end{equation} where 
$\bm{A}_{i,j}=\sigma^2 (t_i \wedge t_j)$ for all $i,j \in \{1\ld d\}$ and a $\sigma^2\geq 0$ which depends on the Hermite expansion of $\varphi$. 
A standard transformation shows that~\eqref{ljsldfjlsdj} is equivalent to 
\begin{equation}\label{eq:defn_F_n}
\bm{S}_n=(S_{n,1}\ld S_{n,d}) \underset{n \to \infty}{\cid} \mathcal{N}_d(\bm{0}, \sigma^2 \Diag(t_1-t_0\ld t_d-t_{d-1})),  \, \text{ for }\,  S_{n,i}= \frac{1}{\sqrt{n}} \sum_{ k=\floor{nt_{i-1}}+1}^{\floor{nt_i}}\varphi(G_k). 
\end{equation} The ensuing corollary gives an explicit quantitative bound for the convergence in~\eqref{eq:defn_F_n}. 

\begin{corollary}\label{cor:Breuer--Major_roc}
Let $d,n\in \N$ and $\bm{S}_n$ be given as in~\eqref{eq:defn_F_n}, where $\varphi=\sum_{q \ge 2} a_qH_q$ has Hermite rank greater or equal $2$. Assume there exists a $\bm{\theta}=(\beta,\kappa,c)\in [1/2,1]\times \R \times (0,\infty)$, such that~\eqref{eq:main_assump_a_i,q} is satisfied for $\varphi$. If $\beta=1/2$ suppose additional that $\kappa <  \Upsilon$. Finally, assume that $\bm{Z}_n\sim \mathcal{N}_d(\bm{0},\bm{\Sigma}_n)$, where $\bm{\Sigma}_n=\cov(\bm{S}_n)$ is invertible, and $\sigma_*^2=\sigma_*^2(\cor(\bm{S}_n))$. Then, there exists a finite constant $C_{\bm{\theta}}>0$, 
only depending on $\bm{\theta}$, such that
    \begin{equation*}
    d_\HR(\bm{S}_n,\bm{Z}_n)\le 
    C_{\bm{\theta}} \psi_{\beta,\kappa}(d) n^{-1/2}\|\rho_n\|_{\ell^1(\Z)}^{3/2} \log_+\left(n^{-1/2}\|\rho_n\|_{\ell^1(\Z)}^{3/2}\right)
    \frac{\log_+(\sigma^2_*)}{\sigma^2_*}, \quad \text{ for all }n,d \in \N.
    \end{equation*}
\end{corollary}
 
 The rate of convergence of the elements from~\eqref{eq:defn_F_n} has been studied in~\cite[Cor.~1.4]{MR4488569} (see also Corollary~\ref{cor:conv_dist_bound_wiener_chaos} below), where the authors considered the convex distance $d_\CD$. In~\cite[Cor.~1.4]{MR4488569} the dimensional dependence was of order $d^{65/24}$. This dependence can be improved in the same setting~\eqref{eq:defn_F_n} using the methodology of this paper in the distance $d_\HR$. Indeed, in the setting of Corollary~\ref{cor:Breuer--Major_roc}, the dimensional dependence is shown to be sub-polynomial if $\beta>1/2$ and $\sigma_*^2$ is uniformly bounded away from $0$. For fixed $d\ge 1$, note that $\cor(\bm{S}_n) \to \bm{I}_d$ as $n \to \infty$, implying that $\sigma_*^2 \to 1$ as $n \to \infty$, and hence $\inf_{n \in \N}\sigma_*^2>0$. For completeness, the statement of~\cite[Cor.~1.4]{MR4488569} is given here since it is directly comparable with Corollary~\ref{cor:Breuer--Major_roc} above. 
\begin{corollary}[{\cite[Cor.~1.4(i)]{MR4488569}}]\label{cor:conv_dist_bound_wiener_chaos}
    Let $\bm{S}_n$ be given by~\eqref{eq:defn_F_n}, and $\bm{Z} \sim \mathcal{N}_d(\bm{0},\bm{\Sigma})$, where $\bm{\Sigma}=\big(\sum_{k \ge m}a_k^2k!\sum_{j \in \Z} \rho(j)^k \big)\Diag(t_1-t_0\ld t_d-t_{d-1})$. Assume that $\sum_{k \in \Z} |\rho(k)|^m<\infty$, $\varphi$ is absolutely continuous and $\varphi,\varphi' \in L^4(\R,\gamma)$ with Hermite rank $m \ge 1$. Then, for $C_{\bm{\Sigma}} \coloneqq 402\big(\OPn{\bm{\Sigma}^{-1}}^{3/2}+1\big)$ and $C(\varphi)\coloneqq \E[\varphi'(G_1)^4]^{1/4}\E[\varphi_1(G_1)^4]^{1/4}$, it holds that
    $$d_\CD(\bm{S}_n,\bm{Z}) \le C_{\bm{\Sigma}} d^{41/24} \sum_{i,j=1}^d \left|\bm{\Sigma}_{i,j}-\E\left[S_{n,i}S_{n,j}\right]\right|+ 2C_{\bm{\Sigma}} C(\varphi)d^{65/24}\frac{\|\rho_n\|_{\ell^1(\Z)}^{3/2}}{\sqrt{n}}, \quad \text{ for all }n,d \in \N.$$
\end{corollary}
    Note that $C_{\bm{\Sigma}}$ in Corollary~\ref{cor:conv_dist_bound_wiener_chaos} depends additionally on dimension $d$ through $\OPn{\bm{\Sigma}^{-1}}$. In~\cite[Cor.~1.4~(i)]{MR4488569}, the dimensional dependence is not directly evident; however, using the proof of~\cite[Cor.~1.4~(i)]{MR4488569}, keeping careful track of the dimensional dependence, yields the dimensional dependence present in Corollary~\ref{cor:conv_dist_bound_wiener_chaos}.

\subsection{Generality of Setting under I.I.D}\label{subsec:generality_of_setting}
In Proposition~\ref{prop:Phi_G_recovers_all_dist} we see that our setting~\eqref{eq:defn_S_n} is a natural generalization of sums of i.i.d.\ random vectors $\bm{S}_n = n^{-1/2}\sum_{k=1}^n \bm{X}_k$ to allow for dependence. 

\begin{proposition}\label{prop:Phi_G_recovers_all_dist}
Let $(G_k)_{k \in \N}$ be a sequence of independent standard Gaussian random variables, then the following statements hold:

\smallskip
\noindent\textbf{(i)} Let $(\bm{X}_k)_{k \in \N}$ be an arbitrary i.i.d.\ sequence of random vectors in $\R^d$ for $d \ge 1$. Then, there exists a measurable function $\Phi:\R \to \R^d$, such that $(\bm{X}_k)_{k \in \N} \eqd (\Phi(G_k))_{k \in \N}$.

\smallskip
\noindent\textbf{(ii)} Suppose that $\bm{X}_k=\Phi(G_k)$ with $\Phi:\R\to\R^d$ satisfies~\eqref{eq:main_assump_a_i,q} under the conditions of Theorem~\ref{thm:main_theorem}. 
Then $\Phi$ is $C^\infty$, and in particular, $\bm{X}_k $ is supported on a one-dimensional $C^\infty$ curve in $\mathbb R^d$.

\smallskip
\noindent\textbf{(iii)} Suppose that $\bm{X}_k=\Phi(G_k)$ is centered and square-integralble with invertable covariance matrix. 
Then, $(\bm{S}_n)_{n \in \N}$ has asymptotically full support in 
$\R^d$. In fact, 
$\bm{S}_n \cid \bm{Z}$ where 
$\bm{Z}\sim \mathcal{N}_d(\bm{0},\bm{\Sigma})$ has full support in $\R^d$, which implies that for every nonempty open $U\subset\R^d$,  $\liminf_{n \to \infty}\p(\bm{S}_n\in U) \ge \p(\bm{Z}\in U) >0$.
\end{proposition}

Proposition~\ref{prop:Phi_G_recovers_all_dist}(i) covers all i.i.d.\ cases, and shows that all i.i.d.\ sequences can be generated from one-dimensional Gaussians. However,~\eqref{eq:main_assump_a_i,q} imposes some restrictions on the choices of $\Phi$. Indeed, from Proposition~\ref{prop:Phi_G_recovers_all_dist}(ii), we can see that the Hermite coefficient assumption~\eqref{eq:main_assump_a_i,q} implies that $\Phi$ is $C^\infty$. Consequently, each $\bm{X}_k=\Phi(G_k)$ is supported on a one-dimensional $C^\infty$ curve in $\R^d$. Although each $\bm{X}_k$ lives on the one-dimensional curve $\Phi(\R)$, the partial sums $\bm{S}_n$ usually exhibit genuinely $d$-dimensional behaviour, as is evident form Proposition~\ref{prop:Phi_G_recovers_all_dist}(iii). Specifically, in the setting of Proposition~\ref{prop:Phi_G_recovers_all_dist}(iii), $\bm{S}_n$ will always have asymptotically full support. Indeed, the law of $\sqrt{n}\bm{S}_n$ is the $n$-fold additive convolution $\mu^{*n}$ of the law $\mu$ of $\bm{X}_1$ and convolution can ensure that $\bm{S}_n$ has genuinely $d$-dimensional behaviour already for finite $n$. The ensuing proposition shows that even $n=2$ and $d=2$ can yield genuinely $2$-dimensional behaviour, i.e.\ the support of the law of $\bm{S}_2$ has a non-empty interior in $\R^2$. 
\begin{proposition}\label{prop:example_gen_d_dim}
    Let $\Phi(x)=(\varphi_1(x),\varphi_2(x))$, where $\varphi_1(x)=\cos(x)-e^{-1/2}$ and $\varphi_2(x)=\sin(x)-xe^{-1/2}$. For independent  $G_1,G_2 \sim \mathcal{N}(0,1)$, it follows that the support of the law of $\bm{S}_2=2^{-1/2}(\Phi(G_1)+\Phi(G_2))$ has a non-empty interior in $\R^2$ and that $\Phi$ satisfies~\eqref{eq:main_assump_a_i,q}.
\end{proposition}

The assumption that the covariance matrix is invertible is often satisfied in applications. Indeed, this is the case for all examples covered in Section~\ref{sec:applications_MoM_ChaFct_fdd} and Section~\ref{sec:ext_results_examples}, which include Hermite method of moments, empirical characteristic functions \& moment-generating functions, finite-dimensional convergence in Breuer--Major theorems. Beyond the i.i.d.\ setting, the framework~\eqref{eq:defn_S_n} also encompasses a broad class of partial sums of highly dependent $d$-dimensional vectors, as is the main topic of this paper.

\subsection{Methodology and Structure}\label{sec:methodology}

To understand the underlying main ideas, some insight into the methodology of proof of the main theorem, i.e.\ Theorem~\ref{thm:main_theorem}, is in order. To prove Theorem~\ref{thm:main_theorem}, the powerful methods from~\cite{MR4312842} and~\cite{MR4488569} are applied, which rely heavily on the applications of Stein's kernels and Malliavin calculus. Working with Stein's kernels in the setting of Malliavin calculus is very tractable, especially when working with multiple integrals (see~\cite[Prop.~3.7]{MR3158721}), since in most cases the construction yields directly a Stein's kernel, see~\eqref{eq:SteinKernel_gen}. The proof of Theorem~\ref{thm:main_theorem} builds upon ideas from the proofs of~\cite[Cor.~1.4]{MR4488569},~\cite[Cor.~1.2]{MR4312842} and \cite[Lem.~2.2]{MR3911126} (as well as~\cite[Lem.~A.7 \& Props~A.1 \&~A.2]{MR3911126}), where we apply sub-Gaussian chaos properties for multiple integrals, the general toolbox of Malliavin calculus, and being able to find tractable bounds in terms of dimension of very involved sums, see Lemma~\ref{lem:Theta_d_growth}. 

With the current known methods, it is unfeasible to have sub-polynomial dimensional dependence under $d_\CD$ and $d_\mW$. Indeed, this would require an extension of the ideas from~\cite{MR4312842}, which heavily depends on the hyper-rectangle structure. Indeed, the proof of our main result, Theorem~\ref{thm:main_theorem}, relies vitally on the choice of the hyper rectangle distance $d_{\HR}$, via~\cite[Thm~1.1]{MR4312842}. The argument in~\cite{MR4312842} combines Stein’s method with a Gaussian anti-concentration inequality~\cite[Lem.~2.1]{MR4312842}, together with delicate smoothing bounds~\cite[Lem.~2.2 \&~2.3]{MR4312842} that are specifically tailored to hyper rectangles, and are not easily generalised to general convex or other sets. These results crucially exploit the product structure of sets of the form $A=\bigtimes_{j=1}^d(a_j,b_j)$, which permits coordinate-wise control.

The remainder of the paper is structured as follows. Section~\ref{sec:connec_lit} discusses the connected literature. In Section~\ref{sec:ext_results_examples}, we look at further specific applications of Theorem~\ref{thm:main_theorem}, such as smoothness assumptions, the limiting case $\beta=1/2$, and finite Hermite expansions. A general introduction to Malliavin calculus and Stein's kernels is given in Section~\ref{sec:bascis_malliavin_proofs}. Subsequently, a more general version of Theorem~\ref{thm:main_theorem} is proved in Section~\ref{subsec:proofs_sec_technical}. Section~\ref{sec:proofs_sec_1} contains the proofs of the results from Section~\ref{sec:main_results}, where Section~\ref{sec:proofs_sec_2} gives the proofs of the results from Section~\ref{sec:ext_results_examples}.

\section{Related Literature}\label{sec:connec_lit}
Multivariate Breuer--Major theorems can in the literature be ambiguous~\cite{MR4488569,MR1331224}, and differ from the setting considered in this paper. Indeed, contrary to the setting of this paper, another approach, generalising the one-dimensional setting, is to consider a multivariate Gaussian process and a function $\Phi:\R^d \to \R$, making the observations univariate random variables~\cite{MR2770907,MR1331224}. Several distances, closely related to $d_\HR$, $d_\CD$ and $d_\mW$, are considered across the literature, e.g.\ the Kolmogorov--Smirnov distance $d_\KS$ where the supremum is taken over $\{\bigtimes_{i=1}^d (-\infty,x_{i}):\bm{x}=(x_{1}\ld x_{d})^\tra \in \R^d\}$ and the $2$-Wasserstein distance. Directly from the definitions of the sets, we can note that $d_\KS \le d_\HR \le d_\CD$. 

In the ensuing subsections, we discuss state-of-the-art results from the literature and compare them to our results. First, we consider the multivariate i.i.d.\ literature, and subsequently, we discuss the few existing results which consider the multivariate setting with dependence and explicit bounds in $d$. As discussed earlier, a $\log$ error in optimality is often expected and can be seen in some of the results mentioned here. Some versions of this trade-off are more extreme than others, with the most common one being $n^{-1/2}\log(n)$, as seen here. In~\cite{cammarota2023quantitative}, the authors show a rate of convergence where the rate is independent of $d$ but where the dependence on $n$ is reduced to a power of logarithm. 

\subsection{The Multivariate I.I.D.\ Case}
In this section, the focus is on a thorough literature review for quantitative bounds in the multivariate i.i.d.\ setting. Polynomial dependence in $d$ is common in the literature and shows up in e.g.~\cite{MR2144310}, where the author considered a Berry--Esseen theorem with dimensional dependence in the classical multivariate CLT setting. In~\cite[Thm~1.1]{MR2144310}, it is shown that the dimensional dependence is $d^{1/4}$, with the optimal rate in $n$ when there is a finite third moment of the random vectors. It is unknown, but conjectured, that the rate $d^{1/4}$ in dimension is optimal for $d_\CD$, see~\cite[p.~401]{MR2144310}. In~\cite{MR4003566}, the author finds explicit constants for this bound.

In~\cite[Cor.~1.1]{MR4312842}, the bound $d_\HR(\bm{S}_n,\bm{Z})\le c(\sigma^2_*)^{-1}\log(d)^{3/2}\log(n)/\sqrt{n}$ is proved when the observations have $\log$-concave densities and  $\bm{Z} \sim \mathcal{N}_d(\bm{0},\bm{\Sigma})$. If $\sigma^2_*$ is bounded away from $0$ by an absolute constant and the observations have finite exponential moments, then~\cite[Prop.~1.1]{MR4312842} shows that the optimal rate for $d_\HR$ generally is $\log(d)^{3/2}/\sqrt{n}$.

Rate of convergence results in $d_\HR$ for independent high-dimensional centred random vectors has also been studied extensively with applications to Bootstrap and one hidden layer neural networks~\cite{MR4583674,MR4505371,MR4500619,MR3161448,MR3693963}. These bounds give a strong dimensional dependence, but they all require strict assumptions on the i.i.d.\ observation, such as being bounded, sub-Gaussian or sub-exponential. Allowing for a degenerated covariance matrix in the limit has also been studied in this setting~\cite{MR4505371}. In~\cite[Thm~1 \& Thm~2]{MR4152649}, the authors study the classical multivariate CLT setting with the additional assumption of bounded observations or observations having a log-concave and isotropic distribution. The rates in~\cite{MR4152649} are studied in the $2$-Wasserstein distance, and found to be polynomial in $d$. 

Note, in the i.i.d.\ setting outside of assumptions like bounded, sub-Gaussian or sub-exponential observations, the sub-polynomial rate from Theorem~\ref{thm:main_theorem} cannot be recovered from the state-of-the-art literature. 

\subsection{The Multivariate Dependent Case}\label{sec:con_lit_depend_case}

There are limited results in the literature on quantitative bounds for central limit theorems in the setting of sums of non-linear functionals of Gaussian variables, allowing for dependent vectors while attaining an exact dimensional dependence in the bound. In~\cite[Cor.~1.4]{MR4488569}, the authors show a rate of convergence for the multivariate Breuer--Major from~\eqref{eq:defn_F_n}, as stated in Corollary~\ref{cor:conv_dist_bound_wiener_chaos}, with a dimensional dependence of $d^{65/24}$. In~\cite[Thm~3.1]{bong2022highdimensional}, the authors study the dimensional dependence in the multivariate CLT when the random vectors are $m$-dependent. We say that $(\bm{X}_k)_{k \in \N}$ is $m$-dependent if $\bm{X}_k$ and $\bm{X}_j$ are independent when $|k-j|>m$ and otherwise dependent. The distance considered in~\cite{bong2022highdimensional}, is the Kolmogorov--Smirnov distance $d_{\KS}$, which is weaker than $d_\HR$ i.e.\ $d_\KS \le d_\HR$. Under some regularity assumptions, a non-degenerate covariance matrix and finite third moment, the dimensional dependence is shown to be a power of $\log$, where the bound dependence on $n$ through $\log(n/m)^{3/2}(n/m)^{-1/2}$.

\section{Further Applications and Examples}\label{sec:ext_results_examples}
Further applications are given in this section to demonstrate the utility and strength of the results from Section~\ref{sec:main_results}. A smoothness condition implying~\eqref{eq:main_assump_a_i,q} is given, implying that the size of the derivatives of $\varphi_i$ heavily influences $\bm{\theta}$ and, therefore, the rate of convergence in~\eqref{eq:main_result_inequality}. Moreover, specific examples of functions $\varphi_i$ that satisfy assumption~\eqref{eq:main_assump_a_i,q} for $\beta=1/2$ are given. Lastly, the main bounds~\eqref{eq:main_result_inequality} \&~\eqref{eq:main_result_inequality_d_C} are evaluated and improved when $\varphi_i$ has a finite Hermite expansion for all $i=1\ld d$. 

\subsection{Dimension dependent on sample size}
In the literature, it is often of interest to let the dimension depend on the number of samples, i.e.\ $d=d(n) \to \infty$ as $n \to \infty$. Theorem~\ref{thm:main_theorem} allows us to study the case $d=d(n)$ for all polynomials, as presented in the ensuing corollary.

\begin{corollary}\label{cor:main_theorem_n_dep_d}
    Assume the setting of Theorem~\ref{thm:fixed_d}. Let $d=d(n)\le n^\lambda$ for $\lambda>0$ and all $n \in \N$. 
    
    \noindent{\nf{(i)}} \underline{Short-range dependence:} Assume that $|\rho(k)| \le \wt c \,|k|^{-\mu}L(|k|)$ for all $k\in \Z\setminus \{0\}$ for $\mu \ge 1$ and $L \in \mathrm{SV}_\infty$ (if $\mu=1$ assume that $L\in \mathrm{SV}_\infty$ is such that $\|\rho\|_{\ell^1(\Z)}<\infty$). 
    If $\beta=1/2$ assume additionally that $\kappa < -\log(4e^{1/(2e)}\lambda)/2-\log(24)/2-5/(4e)$ under $d_{\HR}$ and $\kappa<-\log(3)/2$ under $d_\CD$. Under $d_\CD$, assume additionally that $\lambda<12/65$. Then, there exists a $\zeta_\lambda > 0$ and a constant $C_{\bm{\theta},\rho,\lambda}\in (0,\infty)$ independent of $n$, such that,  for all $n\in \N$,
    \begin{equation*}
    d_\HR(\bm{S}_n,\bm{Z}_n)
    \le 
    C_{\bm{\theta},\rho,\lambda}
    n^{-\zeta_\lambda} \frac{\log_+(\overline{\overline{\sigma}}\, \wt \sigma^2_*/\underline{\underline{\sigma}})}{\wt \sigma^2_*} , \quad \text{ and }\quad d_\CD(\bm{S}_n,\bm{Z}_n)\le 
    C_{\bm{\theta},\rho,\lambda}
    n^{-\zeta_\lambda} \frac{1}{(\sigma_*^2)^{3/2}}.
    \end{equation*}
    
    \noindent{\nf{(ii)}} \underline{Long-range dependence:} Assume that $|\rho(k)| \le \wt c \,|k|^{-\mu}L(|k|)$ for all $k \in \Z \setminus\{0\}$ for $\mu \in (2/3,1)$ and $L \in \mathrm{SV}_\infty$. 
    If $\beta=1/2$, assume additionally under $d_\HR$ that $\kappa< 
    \log((1-3\alpha)/(4e^{1/(2e)}\lambda))
    /2-\log(24)/2-5/(4e)$, and under $d_\CD$, assume that $\lambda<(1-3\alpha)\cdot 12/65$. Then, there exists a $\zeta_\lambda \in (0,(1-3\alpha)/2)$ and a constant $C_{\bm{\theta},\alpha,R,\lambda}\in (0,\infty)$ independent of $n$, such that, for all $n \in \N$,
    \begin{equation*}
    d_\HR(\bm{S}_n,\bm{Z}_n)\le 
    C_{\bm{\theta},\alpha,R,\lambda} 
    n^{-\zeta_\lambda}\frac{\log_+(\overline{\overline{\sigma}}\, \wt \sigma^2_*/\underline{\underline{\sigma}})}{\wt \sigma^2_*}, \quad \text{ and }\quad d_\CD(\bm{S}_n,\bm{Z}_n)\le 
    C_{\bm{\theta},\alpha,R,\lambda}
    n^{-\zeta_\lambda} ((\sigma_*^2)^{-3/2}+1).
    \end{equation*}
\end{corollary}
Note by Corollary~\ref{cor:main_theorem_n_dep_d}, that under $d_\HR$, the dimension can grow as any polynomial as a function of $n$, and a Gaussian approximation is still valid if $\sigma_*^2$ is uniformly bounded away from $0$. As for $d_\CD$, we have a strict upper bound on $\lambda$, i.e.\ an upper bound on the polynomial growth of $d(n)$, such that a Gaussian approximation is still valid. 

\subsection{$C^\infty$ version of~(\ref{eq:main_assump_a_i,q})}\label{sec:C^infty_version_a_i,q_ass}
Recall assumption~\eqref{eq:main_assump_a_i,q} from Theorem~\ref{thm:main_theorem}, which states that there exists a $\bm{\theta}=(\beta,\kappa,c) \in [1/2,1]\times \R \times (0,\infty)$, such that $|a_{i,q}| \le c e^{\kappa q} (q!)^{-\beta}$ for all $i=1\ld d$ and $q \ge 2$. Such an assumption is easily verified whenever the Hermite expansion of a function is available. If not directly available in the literature, the Hermite coefficients can often be calculated from the function $\varphi_i$ under some smoothness assumptions. In the ensuing lemma, we give smoothness conditions such that~\eqref{eq:main_assump_a_i,q} is satisfied. We use the notation $\varphi_i^{(q)}(x)=\tfrac{\D^q}{\D y^q}\varphi_i(y)\big|_{y=x}$ for all $q \ge 1$. It is well known in the literature that the smoothness of the function directly corresponds to the size of the Hermite coefficients, see e.g.~\cite{Hermite_Davis} \&~\cite[Sec.~1.4]{MR2962301}.

\begin{lemma}\label{lem:c_infty_main_result}
    Assume the following two properties for each coordinate of $\Phi(x)$: \nf{I)} $\varphi_i\in C^\infty$ and $\varphi_i^{(q)} \in L^2(\gamma,\R)$ for all $q \ge 2$; \nf{II)} there exists $c \in (0,\infty)$, such that for all $q \ge 2$ and $i=1 \ld d$,  
    $$\left| \int_{-\infty}^\infty \varphi_i^{(q)}(x)
    \gamma(\D x)\right|\le c\begin{dcases}
         e^{\kappa q}\sqrt{q!}, & \text{ for }\beta=1/2 \text{ and }\kappa \le 0, \text{ or }\\
         e^{\kappa q} (q!)^{1-\beta}, &\text{ for }\beta \in (1/2,1] \text{ and }\kappa \in \R.
    \end{dcases}$$
Then,~\eqref{eq:main_assump_a_i,q} in Theorem~\ref{thm:main_theorem} is satisfied for all $q \ge m_i$ and $i=1\ld d$, for the choice of $\bm{\theta}$ from II).
\end{lemma}

The smoothness of the coordinates of $\Phi$, in terms of the size of the derivatives of $\varphi_i$, significantly influences how good dimensional rates we can expect under $d_\HR$. Indeed, under $d_\HR$, as evident from Lemma~\ref{lem:c_infty_main_result} and Theorem~\ref{thm:main_theorem}, the case $\beta=1$ has a much better dimensional dependence (sub-polynomial) than in the case where $\beta=1/2$ (polynomial). In Section~\ref{subsec:examples_beta=1/2}, we will see specific examples of functions in the case $\beta=1/2$.

\begin{remark}\label{rem:dim_dep_constant}
    Assume that~\eqref{eq:main_assump_a_i,q} holds for a $\bm{\theta}$, which can be verified by the assumptions in Lemma~\ref{lem:c_infty_main_result} or directly if possible. It is of interest how the bound of $d_{\HR}(\bm{S}_n,\bm{Z}_n)$ from~\eqref{eq:main_result_inequality} in Theorem~\ref{thm:main_theorem} depends on $c$, i.e.\ how the constant $C_{\bm{\theta}}$ depends on $c$. If one follows the dependence of $c$ through the proofs of Theorem~\ref{thm:main_mult_clt_techncial_thm} and subsequently the proof of Theorem~\ref{thm:main_theorem}, it becomes clear that the dependence is a multiplication with $c^2\log_+(c)$. This means that there exists a universal finite constant $\wt C_{\beta,\kappa}>0$, independent of $c$, $n$ and $d$, such that $C_{\bm{\theta}} \le c^2\log_+(c)\wt C_{\beta,\kappa}$. Similarly, for~\eqref{eq:main_result_inequality_d_C} under $d_\CD$, there exists a universal finite constant $\wt C_{\beta,\kappa}>0$, independent of $c$, $n$ and $d$, such that $C_{\bm{\theta}} \le c^2\wt C_{\beta,\kappa}$. 
\end{remark}

\subsection{Examples for $\beta=1/2$}\label{subsec:examples_beta=1/2} 
In Lemma~\ref{lem:example_beta=1/2} below, we explore examples of functions $\varphi_i$ for which assumption~\eqref{eq:main_assump_a_i,q} is satisfied for $\beta=1/2$ and $\kappa<\Upsilon$. These functions are closely related to the normal distribution since they are modifications of the density and the cumulative distribution function of a $\mathcal{N}(0,\sigma_i^2)$-distribution, for a large enough $\sigma_i^2>0$ for $i= 1 \ld d$. 

\begin{lemma}\label{lem:example_beta=1/2} Let $M=\Upsilon$ or $M=-\log(3)/2$, and let $\sigma_i^2 \in (e^{-2M}-1,\infty)$ for all $i=1\ld  d$, and define $\varphi_i$ as one of the following functions:
\begin{align*}
    \textup{(i)}&\quad  \varphi_i(x)=\frac{1}{\sqrt{2\pi \sigma_i^2}}e^{-\frac{x^2}{2\sigma_i^2}}-\frac{1}{\sqrt{2\pi(\sigma_i^2+1)}}, \quad \text{ for all }x\in \R \text{ and }i=1\ld d,\\
    \textup{(ii)}&\quad  \varphi_i(x)=\int_{-\infty}^x \frac{1}{\sqrt{2\pi \sigma_i^2}}e^{-\frac{y^2}{2\sigma_i^2}}\D y-\frac{1}{2}-\frac{x}{\sqrt{2\pi(\sigma^2+1)}}, \quad \text{ for all }x\in \R, \text{ and }i=1\ld d.
\end{align*}
     Both choices of $\varphi_i$ have Hermite rank at least $2$ for all $i=1\ld d$. For $\kappa = -\log(\min_{1 \le i \le d}\sigma_i^2+1)/2<M$ there exists a $c \in (0,\infty)$ such that~\eqref{eq:main_assump_a_i,q} holds for $\beta=1/2$, i.e.\ $|a_{i,q}|\le ce^{\kappa q}/\sqrt{q!}$ for all $q \ge 2$ and $i=1\ld d$.
\end{lemma} Note that~\eqref{eq:main_assump_a_i,q} cannot hold for $\beta>1/2$, for the functions in Lemma~\ref{lem:example_beta=1/2}. Moreover, note that we under $d_\HR$ choose $M=\Upsilon$ and under $d_\CD$ choose $M=-\log(3)/2$. Moreover, we note that from the proof of Lemma~\ref{lem:example_beta=1/2}, it is clear that~\eqref{eq:main_assump_a_i,q} cannot hold for any $\beta>1/2$. An example of a function where $\beta =1/2$ but $\kappa>\Upsilon$, is the modified indicator function $f(x)=\1_{(0,\infty)}(x)+(1/2)\1_{\{0\}}(x)-1/2-x/\sqrt{2\pi}$ for $x \in \R$, which by~\cite[Thm~2.11]{Hermite_Davis} has Hermite coefficients given by $a_{i,q}=(-1)^{(q-1)/2}(q((q-1)/2)!2^{(q-1)/2}\sqrt{2\pi})^{-1}$ for odd $q$ and $a_{i,q}=0$ for even $q$.

\subsection{Finite Hermite Expansions}\label{subsec:finite_hermite_expansion}
 In this section, functions in $L^2(\gamma,\R)$ which have finite Hermite expansions are considered, i.e.\ functions $\varphi_i$ for which $N_i=\sup\{q \ge 0:a_{i,q}\ne 0\}<\infty $ such that $\varphi_i(x)=\sum_{q=m_i}^{N_i} a_{i,q} H_q(x)$. With a finite Hermite expansion, sharper bounds than in the proof of Theorem~\ref{thm:main_theorem} are available, and hence, a better bound is possible in this case. These bounds exploit the fact that the Hermite expansions are finite, and construct quantitative bounds, which are explicit in both $n$, $d$ and $N$. Another advantage of having a finite Hermite expansion is that we, in the case $\beta=1/2$, can consider all $\kappa \in (-\infty,0]$, where Theorem~\ref{thm:main_theorem} was restricted to the case $\kappa<\Upsilon$ for $d_\HR$ or $\kappa<-\log(3)/2$ for $d_\CD$.

\begin{corollary}\label{cor:finite_hermite_series}
For $n,d\in \N$ let $\bm{S}_n$ be a $d$-dimensional random vector given by~\eqref{eq:defn_S_n}, and suppose that $\varphi_i\in L^2(\gamma,\R)$ has Hermite rank at least 2 and Hermite coefficients $a_{i,q}$ given as in~\eqref{eq:defn_Hermite_expansion} for all $q\ge 2$ and $i=1,\dots, d$. Assume furthermore, that $N=\max_{i \in \{1\ld d\}}\sup\{q \ge 0:a_{i,q}\ne 0\}<\infty$, and that there exists $\bm{\theta}=(\beta,\kappa,c)\in [1/2,1]\times \R \times (0,\infty)$ such that~\eqref{eq:main_assump_a_i,q} is satisfied. If $\beta=1/2$, assume additionally that $\kappa \le 0$. Let $\bm{Z}_n \sim \mathcal{N}_d(\bm{0},\bm{\Sigma}_n)$ and assume $\bm{\Sigma}_n\coloneqq  \cov(\bm{S}_n)$ is invertible, and let $\sigma^2_*=\sigma^2_*(\cor(\bm{S}_n))$.  
    Then, there exists a constant $C_{\bm{\theta}}>0$, which only depends on $\bm{\theta}$, such that 
    \begin{align}\label{eq:d_R_bound_finite expansion}
    \begin{aligned}
        d_\HR(\bm{S}_n,\bm{Z}_n) &\le C_{\bm{\theta}}\log_+(d)\Delta(d,n,N)\log_+(\Delta(d,n,N))\frac{\log_+(\sigma^2_*)}{\sigma^2_*}, \text{ for all }n,d \in \N, \text{ where} \\
    \Delta(d,n,N)&=
        \frac{\|\rho_n\|_{\ell^1(\Z)}^{3/2}}{\sqrt{n}} 
    \left(\frac{2e \log(2d^2-1+e^{N-2})}{N-1}\right)^{N-1} \\
    &\qquad \big(\1_{\{\beta=1/2,\kappa \in [-\log(3)/2,0]\}}N^{7/2}e^{N(2\kappa+\log(3))}
    +\1_{\{\beta>1/2\}\cup\{\beta=1/2,\kappa<-\log(3)/2\}}\big),\text{ and }
    \end{aligned}
    \end{align}
    \begin{align}\label{eq:d_c_bound_finite expansion}
    \begin{aligned}
        d_\CD(\bm{S}_n,\bm{Z}_n)& \le C_{\bm{\theta}} d^{65/24} \Psi(N,\beta) \frac{ \|\rho_n\|_{\ell^1(\Z)}^{3/2}}{\sqrt{n}} \frac{1}{(\sigma_*^2)^{3/2}}, \quad \text{ for all }n,d \in \N, \text{ where }\\
        \Psi(N,\beta) &\coloneqq \1_{\{\beta>1/2\}\cup\{\beta=1/2,\kappa<-\log(3)/2\}}+\1_{\{\beta=1/2\}}\times \begin{dcases}
            \left( \sum_{q=2}^N q^{1/2}\right)\left( \sum_{\ell=2}^N \ell^{-1/2}\right) \1_{\{\kappa=-\log(3)/2\}}, \\
        N^{5/2} e^{(2\kappa+\log(3))N} \1_{\{\kappa \in (-\log(3)/2,0]\}}.
        \end{dcases}
    \end{aligned}
    \end{align} 
\end{corollary}

\section{Proof of Main Result \& Malliavin Calculus}\label{sec:tech_res_malliavin}
In this section, we introduce the general main result, i.e.\ Theorem~\ref{thm:main_mult_clt_techncial_thm}, which is a generalisation of Theorem~\ref{thm:main_theorem} and the basics of Malliavin calculus and Stein's kernels. 

\subsection{Malliavin Calculus and Stein's Kernels}\label{sec:bascis_malliavin_proofs}

This section introduces the basics of Malliavin calculus and Stein's kernels. For a thorough background on Malliavin calculus, we refer to the monographs~\cite{MR2200233,MR2962301}. Let $\mH$ be a real separable Hilbert space with inner product $\langle \cdot,\cdot\rangle_\mH$ and norm $\|\cdot\|_\mH=\langle \cdot,\cdot\rangle_\mH^{1/2}$. We say that $X=\{X(h):h \in \mH\}$ is an \emph{isormormal Gaussian process} over $\mH$, where $X$ is defined on some probability space $(\Omega,\mathcal{F},\mathds{P})$, if $X$ is a centred Gaussian family indexed by $\mH$ with $\E[X(h)X(g)]=\langle h,g\rangle_{\mH}$. Moreover, let $\mathcal{F}$ be the $\sigma$-algebra generated by $X$, i.e.\ $\mathcal{F}=\sigma\{X\}$, and write $L^2(\Omega,\mathcal{F},\mathds{P})=L^2(\Omega)$ where $F \in L^2(\Omega)$ if $\E[|F|^2]<\infty$. Let $\mathcal{H}_q $ be the $q$th Wiener chaos of $X$, defined to be the closure of the span of $\{H_p(W(h)) : h\in \mH, \|h\|_\mH = 1\}$, i.e.\ $\mathcal{H}_q \coloneqq \overline{\text{span}}\{H_q(X(h)) : h\in \mH, \|h\|_\mH = 1\}$ for $q \in \N_0$, and define $\mathcal{P}_{q} \coloneqq \oplus_{i=0}^q \mathcal{H}_i$ for $q \in \N_0$. Recall the \emph{Wiener--It\^o chaos expansion} (see~\cite[Cor.~2.7.8]{MR2962301}), which, for any $F \in L^2(\Omega)$, states that $F=\E[F]+\sum_{q=1}^\infty I_q(f_q)$ where $I_0(f_0)=\E[F]$, $I_q(f_q)$ is the $q$th multiple Wiener--It\^o integral and the kernels $f_q$ are in the symmetric $q$th tensor product in $ \mH^{\odot q}$. For such a $F$ we say that $F \in \mD$ if $\sum_{q=1}^\infty q!q\|f_q\|_{\mH^{\otimes q}}^2<\infty$.

For $F \in \mD$, we let the random element $D F$ with values in $\mH$ be the \emph{Malliavin derivative}, which satisfies the \emph{chain rule}. Indeed, let $\bm{F}=(F_1\ld F_d)$ where $F_i \in \mD$ for $i=1\ld d$ and let $\phi:\R^d \to \R$ be a continuous differentiable function with bounded partial derivatives. Then, by~\cite[Prop.~1.2.3]{MR2200233}, $\phi(\bm{F}) \in \mD$ and $D\phi(\bm{F})=\sum_{i=1}^d \partial_i \phi(\bm{F}) D F_i$. Throughout, use the notation: $\partial_i \phi(\bm{x})=\tfrac{\partial}{\partial x_i} \phi(\bm{x})$ as the $i$th partial derivative and $\partial_{ij} \phi(\bm{x})=\tfrac{\partial^2}{\partial x_i \partial x_j} \phi(\bm{x})$ as the second order $i,j$th partial derivative. 

The adjoint of the derivative operator $D$ is the \emph{divergence operator} $\delta$~\cite[Def.~1.3.1]{MR2200233}. The divergence operator is defined through the following two conditions: i) the domain of $\delta$, denoted $\Dom \, \delta$, is given by the set of $\mH$-valued square integrable random variables $u$ where $|\E[\langle DF,u\rangle_{\mH}]| \le c \sqrt{\E[F^2]}$ for all $F \in \mD$ and a constant $c$ depending on $u$; ii) if $u \in \Dom \, \delta$ then $\delta(u) \in L^2(\Omega)$ and 
\begin{equation}\label{eq:Dom_delta_2.5.2}
    \E[F \delta(u)]=E[\langle DF,u\rangle_\mH ], \quad \text{for all }F \in \mD.
\end{equation} The $\delta$ operator is also called the \emph{Skorohod integral}. If $\bm{F}=(F_1\ld F_d)$ and $F_i \in \mD$ for all $i=1\ld d$, we can without loss of generality assume that $F_i=\delta(u_i)$ for some $u_i \in \Dom(\delta)$. 

In the following, the operator $L$ which is the \emph{generator of the Ornstein-Uhlenbeck semigroup} and its \emph{pseudo inverse} $L^{-1}$ (see~\cite[Defs~2.8.7 \& ~2.8.10]{MR2962301}) are introduced. For $F \in L^2(\Omega)$, it holds that $F \in \Dom \, L$ if $\sum_{q=1}^\infty q^2 \E[I_q(f_q)^2]<\infty$, and hence the operator $L$ is defined for $F \in \Dom \, L$ and $L^{-1}$ for general $F \in L^2(\Omega)$ via the following relations
\begin{equation}\label{eq:defn_L_L^(-1)}
    L F = -\sum_{q=1}^\infty q I_q(f_q), \quad \text{and }\quad L^{-1} F=-\sum_{q =1}^\infty \frac{1}{q}I_q(f_q).
\end{equation} By the \emph{isometry property} of multiple integrals (see~\cite[Prop.~2.7.5]{MR2962301}), i.e.\ for $f \in \mH^{\odot q}$ and $g \in \mH^{\odot p}$ it holds that $\E[I_q(f)I_p(g)]=p!\langle f,g \rangle_{\mH^{\otimes p}}$ if $p=q$ and $\E[I_q(f)I_p(g)]=0$ otherwise, the summability criteria in $\Dom \, L$ can be rewritten as $\sum_{q=1}^\infty q^2 q! \|f_q\|_{\mH^{\otimes q}}^2<\infty$. A useful connection exists between the operators $\delta, L$ and $L^{-1}$. Indeed, by~\cite[Prop.~2.8.8]{MR2962301}, we have, for any $F \in L^2(\Omega)$, that $F \in \Dom \, L$ if and only if $F \in \mD$ and $DF \in \Dom \, \delta$, in which case $\delta D F=-LF$. Hence, as shown in~\cite[Prop.~2.8.11]{MR2962301},
\begin{equation}\label{eq:defining_relation_L_delta_D}
    F-\E [F]= L L^{-1}F=- \delta D L^{-1} F, \quad \text{ for any }F \in L^2(\Omega).
\end{equation}

As mentioned in Section~\ref{sec:main_results}, one of the main tools used to prove Theorem~\ref{thm:main_mult_clt_techncial_thm}, is Stein's kernels~\cite[Def.~1.1]{MR4312842}. We say that the $d \times d$ matrix valued function $\bm{x} \mapsto \bm{\tau^{F}}(\bm{x})=(\bm{\tau}_{i,j}^{\bm{F}}(\bm{x}))_{i,j =1\ld d}$ on $\R^d$ is the \emph{Stein's kernel} for (the law of) $\bm{F}$ if $\E[|\bm{\tau}_{i,j}^{\bm{F}}(\bm{F})|]<\infty$ for any $i,j \in \{1\ld d\}$ and 
\begin{equation}\label{eq:defn_steins_kernel}
    \sum_{j=1}^d \E[\partial_j f(\bm{F})F_j]=\sum_{i,j=1}^d \E[\partial_{ij}f(\bm{F})\bm{\tau}_{i,j}^{\bm{F}}(\bm{F})],
\end{equation} for all functions $f: \R^d \to \R$ where $f \in C^\infty$ with bounded partial derivatives of all orders. Stein's kernels often exist in the Malliavin calculus setting as suggested by~\cite[Prop.~3.7]{MR3158721}. Indeed, if $\bm{F}$ is a centred random vector with $F_i \in \mD$ for $i=1\ld d$, we define the $d \times d$ matrix $(\bm{M_F}(i,j))_{1 \le i,j \le d} = (\langle -DL^{-1} F_i, D F_j\rangle_{\mH})_{1 \le i,j \le d}$. Hence,~\cite[Prop.~3.7]{MR3158721} implies that $\bm{F}$ has a Stein's kernel given by
\begin{equation}\label{eq:SteinKernel_gen}
    \bm{\tau}^{\bm{F}}_{i,j}(\bm{x})=\E[\bm{M}_{\bm{F}}(i,j)|\bm{F}=\bm{x}], \qquad \text{for all }i,j=1\ld d \text{ and }\bm{x} \in \R^d.
\end{equation} 

The main inequality used in the proof of Theorem~\ref{thm:main_mult_clt_techncial_thm}, is~\eqref{eq:log_dim_upperbound} (\cite[Thm~1.1]{MR4312842}) stated below. If $\bm{F}$ has a Stein's kernel $\bm{\tau^{F}}$ and $\bm{Z} \sim \mathcal{N}_d(0,\bm{\Sigma})$ with $\sigma_*^2=\sigma_*^2(\bm{\Sigma})>0$, then
\begin{equation}\label{eq:log_dim_upperbound}
    d_\HR(\bm{F},\bm{Z})\le 
    \frac{C\Delta_{\bm{F}}\log_+(d)}{\sigma_*^2
    } \log_+\left(\frac{\underline{\sigma}
    \Delta_{\bm{F}}}{\overline{\sigma}
    \sigma_{*}^2
    }\right), \, \text{with } \Delta_{\bm{F}}\coloneqq \E\bigg[\max_{1 \le i,j \le d}|\bm{\Sigma}_{i,j}-\bm{\tau^{F}}_{i,j}(\bm{F})|\bigg],
\end{equation} for all $d \ge 3$. In~\eqref{eq:log_dim_upperbound}, $\underline{\sigma}$ and $\overline{\sigma}$ are both associated to $\bm{\Sigma}$. 

\subsection{General Main Result}\label{subsec:proofs_sec_technical}
Throughout, we will denote by $S_{n,i}$ the $i$th coordinate of $\bm{S}_n$ for any $i=1\ld d$. Recall $r=r_{\beta,\kappa}$ from~\eqref{defn:r_constant-2} by $r = 2e^{1/(2e)}\beta e^{(\kappa+\log(24)/2+5/(4e))/\beta}2^{1/(2\beta)}$ for all $\beta \in [1/2,1]$ and $\kappa \in \R$. We will now state and prove a slightly more general result, which implies Theorem~\ref{thm:main_theorem} under $d_\HR$.

\begin{theorem}\label{thm:main_mult_clt_techncial_thm}
Let $\bm{S}_n$ be as in~\eqref{eq:defn_S_n}, and assume that $\bm{Z}_n \sim \mathcal{N}_d(\bm{0},\bm{\Sigma}_n)$ for an invertible covariance matrix $\bm{\Sigma}_n$ with associated correlation matrix $\bm{\Lambda}_n$ and $\sigma_*^2=\sigma_*^2(\bm{\Lambda}_n)$. Assume there exists a $\bm{\theta}=(\beta,\kappa,c)\in [1/2,1]\times \R\times (0,\infty)$ such that~\eqref{eq:main_assump_a_i,q} holds. If $\beta=1/2$ assume that $\kappa <  \Upsilon$. Then, there exists a $C_{\bm{\theta}} >0$, depending only on $\bm{\theta}$, such that 
    \begin{align*}
    d_\HR(\bm{S}_n,\bm{Z}_n)&\le 
    C_{\bm{\theta}} \log_+(d) \Delta(\bm{S}_n,\bm{\Sigma}_n) \log_+\left(\frac{\Delta(\bm{S}_n,\bm{\Sigma}_n)}{\sigma_{*}^2}\right)\frac{1}{\sigma_*^2}, \quad \text{for all } n,d\ge 1, \text{ where}\\ 
    \Delta(\bm{S}_n,\bm{\Sigma}_n)&= \frac{\|\rho_n\|_{\ell^1(\Z)}^{3/2}e^{r\log_+^{1/(2\beta)}(d)}}{\sqrt{n}\log_+(d)}+\max_{1 \le i,j \le d}\left|(\bm{\Lambda}_n)_{i,j}-\frac{\E\left[S_{n,i}S_{n,j}\right]}{\sqrt{(\bm{\Sigma}_n)_{i,i}(\bm{\Sigma}_n)_{j,j}}}\right|.
    \end{align*}
\end{theorem}
Note in the setting of Theorem~\ref{thm:main_mult_clt_techncial_thm}, that $\bm{\Sigma}_n= \cov(\bm{S}_n)$ implies 
\begin{equation*}
    \max_{1 \le i,j \le d}\bigg|(\bm{\Lambda}_n)_{i,j}-\E\left[S_{n,i}S_{n,j}\right]/\sqrt{(\bm{\Sigma}_n)_{i,i}(\bm{\Sigma}_n)_{j,j}}\bigg|=0.
\end{equation*}

Moreover, we will also state and prove a slightly more general version of Theorem~\ref{thm:main_theorem} under $d_\CD$.
\begin{theorem}\label{thm:Ext_to_convex_dist}
    Let $\bm{S}_n$ be as in~\eqref{eq:defn_S_n}, and assume that $\bm{Z}_n \sim \mathcal{N}_d(\bm{0},\bm{\Sigma}_n)$ for an invertible covariance matrix $\bm{\Sigma}_n$, with associated correlation matrix $\bm{\Lambda}_n$ and $\sigma_*^2=\sigma_*^2(\bm{\Lambda}_n)$. Assume there exists a $\bm{\theta}=(\beta,\kappa,c)\in [1/2,1]\times \R\times (0,\infty)$ such that~\eqref{eq:main_assump_a_i,q} holds. If $\beta=1/2$ assume that $\kappa < -\log(3)/2$. Then, there exists a $C_{\bm{\theta}} >0$, depending only on $\bm{\theta}$, such that it for all $n,d \in \N$, holds that
     \begin{equation*}
        d_\CD(\bm{S}_n,\bm{Z}_n) \le C_{\bm{\theta}} d^{41/24} \left(d \frac{ \|\rho_n\|_{\ell^1(\Z)}^{3/2}}{\sqrt{n}} + \sqrt{\sum_{1 \le i,j \le d} \!\! \! \left(\frac{\E[S_{n,i}S_{n,j}]}{\sqrt{(\bm{\Sigma}_n)_{i,i}(\bm{\Sigma}_n)_{j,j}}}-(\bm{\Lambda}_n)_{i,j} \right)^2}\right)\frac{1}{(\sigma_*^2)^{3/2}}.
    \end{equation*}
\end{theorem}

To prove Theorems~\ref{thm:main_mult_clt_techncial_thm} \&~\ref{thm:Ext_to_convex_dist}, the ensuing lemmas and estimates are needed. The proofs of the lemmas are pushed to the end of this section, except for the proof of Lemma~\ref{lem:extend_fang_koike}, which is in Appendix~\ref{app:proof_lemma}.
\begin{lemma}\label{lem:mom_bound_Hermite}
    Let $\ell\ge 2$, $H_\ell$ be the $\ell$th Hermite polynomial and $G \sim \mathcal{N}(0,1)$. Then, there exists a uniform constant $c>0$ such that
    \begin{equation}\label{eq_rough_momentbound}
        \E[H_\ell(G)^4]^{1/4} \le  c \ell^{\ell/2}e^{\ell(\log(3)/2-1/2)}\ell^{1/4}, \quad \text{ for all } \ell \ge 2.
    \end{equation}
\end{lemma}

Recall that the product log function $W(x)$ is defined to be the inverse function of $f(w)=we^w$, i.e.\ $W(x)$ is the principal solution to $W(x) e^{W(x)}=x$ for $x \ge 0$, where e.g.\ $W(e^{-1+1/(2e)})=0.3208\ldots$. Recall the definition of $r=r_{\beta,\kappa}$ from~\eqref{defn:r_constant-2}.

\begin{lemma}\label{lem:Theta_d_growth} Assume that $d \ge 2$, let $(f_q)_{q \ge 2}\subset [0,\infty)$ be a sequence where there exists $\bm{\theta} \in [1/2,1]\times \R\times (0,\infty)$
such that $f_q \le c e^{\kappa q}(q!)^{-\beta}$ for $q \ge 2$. Let
 \begin{align}\label{eq:theta_d_denf}
     \Theta(d)\coloneq \sum_{\ell,q =2}^\infty&\bigg\{ f_{\ell}f_q\left(\frac{4e\log(2d^2-1+e^{(\ell+q)/2})}{\ell+q-2}\right)^{(\ell+q-2)/2}\\
         &\quad \times \ell^{(\ell-1)/2}e^{(\ell-1)(\log(3)/2-1/2)}\ell^{1/4}q^{(q-1)/2}e^{(q-1)(\log(3)/2-1/2)}q^{5/4}\bigg\}.
 \end{align}
    \noindent(i) Let $\alpha \in (0,\sqrt{W(e^{-1 + 1/(2 e)})/e^{1/(2e)}})$, $\kappa = \log(\alpha)-\log(24)/2-5/(4e)$ and $\beta =1/2$. Then there exists a $K_{\bm{\theta}}$, only dependent on $\bm{\theta}$, such that $\Theta(d) \le 
    K_{\bm{\theta}}d^{r}/\log_+(d)$, for all $d \ge 1$.

    \noindent(ii) Let $c>0$, $\kappa \in \R$ and $\beta \in (1/2,1]$. Then there exists a $K_{\bm{\theta}}$, only dependent on $\bm{\theta}$, such that $\Theta(d) \le  
    K_{\bm{\theta}} \exp\{r \log_+^{1/(2\beta)}(d)\}/\log_+(d)$, for all $d \ge 1$.
\end{lemma}

\begin{lemma}\label{lem:extend_fang_koike}
    If $\bm{F}$ has a Stein's kernel $\bm{\tau^{F}}$ and $\bm{Z} \sim \mathcal{N}_d(0,\bm{\Sigma})$ with $\sigma_*^2=\sigma_*^2(\bm{\Sigma})>0$, then
\begin{equation}\label{eq:log_dim_upperbound_2}
    d_\HR(\bm{F},\bm{Z})\le 
    \frac{C\Delta_{\bm{F}}\log_+(d)}{\sigma_*^2\wedge 1
    } \log_+\left(\frac{(\underline{\sigma}\wedge 1)
    \Delta_{\bm{F}}}{(\overline{\sigma}\vee 1)
    (\sigma_{*}^2\wedge 1)
    }\right), \, \text{with } \Delta_{\bm{F}}\coloneqq \E\bigg[\max_{1 \le i,j \le d}|\bm{\Sigma}_{i,j}-\bm{\tau^{F}}_{i,j}(\bm{F})|\bigg],
\end{equation} for $d =1$ and $d=2$. In~\eqref{eq:log_dim_upperbound}, $\underline{\sigma}$ and $\overline{\sigma}$ are both associated to $\bm{\Sigma}$. 
\end{lemma}

Recall that $(G_k)_{k \in \N}$ is a centred stationary Gaussian sequence with $\rho(j-k)=\E[G_jG_k]$ and $\rho(0)=1$. Hence, by~\cite[Rem.~2.1.9]{MR2962301} we may choose an isonormal Gaussian process $\{X_h:h \in \mH\}$, such that $(G_k)_{k \in \N} \eqd \{X_{e_k}:k \in \N\}$ where $(e_k)_{k \in \N}\subset \mH$ verifying $\langle e_k,e_j\rangle_\mH = \rho(j-k)$ for all $j,k \in \N$. 
The proof of Theorem~\ref{thm:main_mult_clt_techncial_thm} relies on Lemma~\ref{lem:Theta_d_growth}, with the proof postponed to the end of this section. 
\begin{proof}[Proof of Theorem~\ref{thm:main_mult_clt_techncial_thm}]
By~\cite[Prop.~3.7]{MR3158721}, $\bm{S}_n$ has a Stein's kernel given by~\eqref{eq:SteinKernel_gen}, and hence, by~\eqref{eq:log_dim_upperbound}, it follows that
    \begin{align}
        d_\HR(\bm{S}_n,\bm{Z}_n)&\le 
    C \frac{\wt\Delta(\bm{S}_n,\bm{\Sigma}_n)}{\sigma_*^2}\log_+(d) \log_+\left(\frac{\underline{\sigma}\wt\Delta(\bm{S}_n,\bm{\Sigma}_n)}{\overline{\sigma} \sigma_{*}^2}\right), \quad \text{ for all }n \in \N, \, d \ge 3 \text{ with } \label{eq:bound_d_ge_3}\\
    \wt\Delta(\bm{S}_n,\bm{\Sigma}_n) &\coloneqq  \E\bigg[\max_{1 \le i,j \le d}|(\bm{\Sigma}_n)_{i,j}-\bm{\tau}^{\bm{S}_n}_{i,j}(\bm{S}_n)|\bigg], \label{eq:delta_tilde_Sigma}
    \end{align} where $\sigma_*^2=\sigma_*^2(\bm{\Sigma}_n)$, $\underline{\sigma}^2=\underline{\sigma}^2(\bm{\Sigma}_n)$ and $\overline{\sigma}^2=\overline{\sigma}^2(\bm{\Sigma}_n)$. If $d=1$ or $d=2$, then Lemma~\ref{lem:extend_fang_koike} implies that 
    \begin{equation}\label{eq:bound_d_=1,=2}
        d_\HR(\bm{S}_n,\bm{Z}_n)\le 
    C \frac{\wt\Delta(\bm{S}_n,\bm{\Sigma}_n)}{\sigma_*^2\wedge 1}\log_+(d) \log_+\left(\frac{(\underline{\sigma}\wedge 1)\wt\Delta(\bm{S}_n,\bm{\Sigma}_n)}{(\overline{\sigma} \vee 1) (\sigma_{*}^2\wedge 1)}\right), \quad \text{ for all }n \in \N,
    \end{equation}
    with $\wt\Delta(\bm{S}_n,\bm{\Sigma}_n)$ as in~\eqref{eq:delta_tilde_Sigma}. Hence, the main aim is now to find a bound for $\wt\Delta(\bm{S}_n,\bm{\Sigma}_n)$ for all $n,d \in \N$.

    Applying~\cite[Prop.~2.8.8]{MR2962301}, it follows that $\delta D S_{n,i}=-L S_{n,i}$. Hence, by~\eqref{eq:Dom_delta_2.5.2} and~\eqref{eq:defining_relation_L_delta_D}, it follows that $\E[\langle -DL^{-1}S_{n,i}, DS_{n,j} \rangle_{\mH}]
    =\E[S_{n,i}S_{n,j}]$ for all $n \ge 1$ and $i,j=1\ld d$. Thus, by Jensen's inequality and the triangle inequality, we get that
    \begin{align*}
        \wt\Delta &(\bm{S}_n,\bm{\Sigma}_n) \le \E\bigg[\max_{1 \le i,j \le d}|(\bm{\Sigma}_n)_{i,j}-\langle-DL^{-1}S_{n,i}, DS_{n,j} \rangle_{\mH}|\bigg] \\
        &\le \max_{1 \le i,j \le d}|(\bm{\Sigma}_n)_{i,j}-\E[S_{n,i}S_{n,j}]| + \E\bigg[\max_{1 \le i,j \le d}|\E[S_{n,i}S_{n,j}]-\langle -DL^{-1}S_{n,i}, DS_{n,j} \rangle_{\mH}|\bigg]\\
        &= \max_{1 \le i,j \le d}|(\bm{\Sigma}_n)_{i,j}-\E[S_{n,i}S_{n,j}]| + \E\left[\max_{1 \le i,j \le d}|\Delta_n(i,j)|\right],
    \end{align*} where $\Delta_n(i,j)\coloneqq \E[\langle -DL^{-1}S_{n,i}, DS_{n,j} \rangle_{\mH}]-\langle -DL^{-1}S_{n,i}, DS_{n,j} \rangle_{\mH}$ for all $n \in \N$ and $i,j=1\ld d$. The purpose of the remainder of the proof is to bound $\E\left[\max_{1 \le i,j \le d}|\Delta_n(i,j)|\right]$. 

    The next step of the proof is to reduce the problem to working with sums of multiple integrals. By the definition of $S_{n,i}$, and since $\varphi_i$ has Hermite rank $m_i\ge 2$ for $i=1\ld d$, it follows that
    \begin{equation*}
        S_{n,i}= 
        \frac{1}{\sqrt{n}} \sum_{k=1}^n \sum_{ \ell=2}^\infty a_{i,\ell} H_\ell(G_k) \eqd \frac{1}{\sqrt{n}} \sum_{k=1}^n \sum_{ \ell=2}^\infty a_{i,\ell} H_\ell(X_{e_k}) =  \sum_{ \ell=2}^\infty \frac{a_{i,\ell}}{\sqrt{n}} \sum_{k=1}^n H_\ell(X_{e_k}),
    \end{equation*} for all $n \in \N$ and $i=1\ld d$.
    Due to~\cite[Thm~2.7.7]{MR2962301}, it follows that
    \begin{equation}\label{eq:f_n,i_kernels}
        S_{n,i} \eqd \sum_{ \ell=2}^\infty \frac{a_{i,\ell}}{\sqrt{n}} \sum_{k=1}^n I_\ell (e_k^{\otimes \ell})= \sum_{ \ell=2}^\infty I_\ell \left(\frac{a_{i,\ell}}{\sqrt{n}} \sum_{k=1}^n e_k^{\otimes \ell}\right)=\sum_{ \ell=2}^\infty I_\ell \left(f_n(i,\ell)\right),
    \end{equation} where $f_n(i,\ell)=a_{i,\ell} n^{-1/2} \sum_{k=1}^n e_k^{\otimes \ell} \in \mH^{\odot \ell} $ for all $n \in \N$ and $i \in \{1\ld d\}$. Next, for any $i \in \{1\ld d\}$, we note that $D S_{n,j} \eqd\sum_{ q =2}^\infty D I_{q}(f_n(j,q))$ and $-DL^{-1}S_{n,i}\eqd \sum_{\ell=2}^\infty \ell^{-1} D I_{\ell}(f_n(i,\ell))$, which follows by~\eqref{eq:defn_L_L^(-1)} and~\cite[Prop.~2.7.4]{MR2962301}. Thus, by~\eqref{eq:f_n,i_kernels}, we have that $\Delta_n(i,j)\eqd \sum_{\ell,q=2}^\infty \delta(i,j,q,\ell,n)$, where
    \begin{align}
        \delta(i,j,q,\ell,n) 
        &\coloneqq 
        \frac{1}{\ell}\big( \E\big[\big\langle D I_{\ell}(f_n(i,\ell)), DI_{q}(f_n(j,q)) \big\rangle_{\mH}\big]-\big\langle DI_{\ell}(f_n(i,\ell)), DI_{q}(f_n(j,q)) \big\rangle_{\mH}\big).\label{eq_defn_delta_i,j,q,l}
    \end{align} Hence, the following inequality follows directly from the triangle inequality: 
    \begin{equation}\label{eq:triangle_inequ_sec5}
        \E\big[\max_{1 \le i,j \le d}| \Delta_n(i,j)|\big] \le \sum_{\ell,q =2}^\infty \E\big[\max_{1 \le i,j \le d}|\delta(i,j,q,\ell,n)|\big].
    \end{equation} 
     Note that $D(G_k)=e_k$ by~\cite[Prop.~2.3.7]{MR2962301} and $H_\ell'(x)=\ell H_{\ell-1}(x)$ by~\cite[Prop.~1.4.2(i)]{MR2962301}. Hence, by definition of the kernels $f_n(i,\ell)$ together with the chain rule, it follows that
    \begin{align}\label{eq:Product_hermite_form_delta}
    \begin{aligned}
        &\frac{1}{\ell}\langle DI_\ell(f_n(i,\ell)),DI_q(f_n(j,q))\rangle_\mH= 
        \frac{1}{\ell} \bigg\langle D\bigg(\frac{a_{i,\ell}}{\sqrt{n}}\sum_{k=1}^nI_\ell\big(e_k^{\otimes \ell}\big)\bigg),D\bigg(\frac{a_{j,q}}{\sqrt{n}}\sum_{r=1}^nI_q\big(e_r^{\otimes q}\big)\bigg)\bigg\rangle_\mH\\
       &\qquad  =
        \frac{a_{i,\ell} a_{j,q}}{\ell n} \sum_{k,r=1}^n \langle D I_\ell\big(e_k^{\otimes \ell}\big),D I_q\big(e_r^{\otimes q}\big)\rangle_\mH 
        \eqd 
        \frac{a_{i,\ell} a_{j,q}}{\ell n} \sum_{k,r=1}^n  \langle D H_\ell(G_k),D H_q(G_r)\rangle_\mH\\
        &\qquad =
        \frac{a_{i,\ell} a_{j,q} q}{n} \sum_{k,r=1}^n H_{\ell-1}(G_k)H_{q-1}(G_r) \rho(k-r).
    \end{aligned}
    \end{align} Denote $\|\cdot \|_2=\sqrt{\E[|\cdot|^2]}$. By~\eqref{eq:Product_hermite_form_delta}, it follows that $\delta(i,j,q,\ell,n) \in \mathcal{P}_{\ell+q-2}$ for all $q,\ell \ge 2$ (see Section~\ref{sec:bascis_malliavin_proofs} for the definition of $\mathcal{P}_q$). By~\cite[Def.~A.1]{MR3911126}, we say that a variable $X$ is a sub-$r$th chaos random variable relative to scale $R\ge 0$, if for a positive integer $r$ we have $\E[\exp((|X|/R)^{r/2})]\le 2$. Hence,~\cite[Prop.~A.1]{MR3911126} implies the existence of a constant $M_{\ell+q-2}>0$, only dependent on $q+\ell$, such that $\delta(i,j,q,\ell,n)$ is a sub-$(\ell+q-2)$th chaos random variable relative to scale $M_{\ell+q-2}\|\delta(i,j,q,\ell,n)\|_2$. By~\cite[Prop.~A.2]{MR3911126}, it thus follows for all $\ell, q \ge 2$ that
    \begin{equation*}
        \E\bigg[\max_{1 \le i,j \le d}|\delta(i,j,q,\ell,n)|\bigg] \le M_{\ell+q-2} \log^{(q+\ell-2)/2}(2d^2-1+e^{(\ell+q-2)/2-1})\max_{1\le i,j \le d} \|\delta(i,j,q,\ell,n)\|_2.
    \end{equation*} 
    
    Next, we show that $M_{\varpi}=(4e/\varpi)^{\varpi/2}$ for $\varpi>0$. Indeed, by the proof of~\cite[Lem.~A.7]{MR3911126}, it suffices to find a $M_{\varpi}>0$ only dependent on $\varpi$, which satisfies $\sum_{k=1}^\infty (2k/\varpi)^kM_\varpi^{-2k/\varpi}(k!)^{-1} \le 1$. Using that $(k/e)^k \le k!$, and $2e\varpi^{-1} M_\varpi^{-2/\varpi}=1/2 $ for our choice of $M_\varpi$, it follows that
\begin{equation*}
        \sum_{k=1}^\infty \frac{(2k/\varpi)^k}{k!M_\varpi^{2k/\varpi}} 
        \le \sum_{k=1}^\infty \frac{(2k/\varpi)^k}{(k/e)^k M_\varpi^{2k/\varpi}}
        = \sum_{k=1}^\infty \left(\frac{2e}{\varpi M_\varpi^{2/\varpi}} \right)^k=1.
    \end{equation*} For the remainder of the proof, we thus choose $M_{\ell+q-2}\coloneqq(4e/(\ell+q-2))^{(\ell+q-2)/2}$ for all $\ell,q \ge 2$.

    Next, we will find an explicit bound for $\max_{1\le i,j \le d} \|\delta(i,j,q,\ell,n)\|_2$. Note that
    \begin{equation}\label{eq:L^2_norm_delta}
        \|\delta(i,j,q,\ell,n)\|_2^2=\var(\langle -DL^{-1}I_\ell(f_n(i,\ell)),DI_q(f_n(j,q))\rangle_\mH).
    \end{equation} Hence, by~\eqref{eq:Product_hermite_form_delta} and~\eqref{eq:L^2_norm_delta}, it holds that
    \begin{align*}
         \|\delta(i,j,q,\ell,n)\|_2^2\le\frac{ a_{i,\ell}^2 a_{j,q}^2 q^2}{ n^2}\!\!\sum_{\substack{k,k',r,r' \\ \in \{1\ld n\}}} |\cov(H_{\ell-1}(G_k)H_{q-1}(G_r),H_{\ell-1}(G_{k'})H_{q-1}(G_{r'})) \rho(k-r)\rho(k'-r')|.
    \end{align*} Next, by Gebelin's inequality~\cite[Thm~2.3]{MR3978683}, it follows that
    \begin{align}
    \begin{aligned}\label{eq_Gebelin'_inequality}
         &\cov(H_{\ell-1}(G_k)H_{q-1}(G_r),H_{\ell-1}(G_{k'})H_{q-1}(G_{r'})) \\
        & \qquad \le \theta \sqrt{\var(H_{\ell-1}(G_k)H_{q-1}(G_r))}\sqrt{\var(H_{\ell-1}(G_{k'})H_{q-1}(G_{r'}))}, \quad  \text{ for all }q,\ell \ge 2,
    \end{aligned}
    \end{align} where $\theta=\max\{|\rho(k-k')|,|\rho(k-r')|,|\rho(r-r')|,|\rho(r-k')|\}$. Moreover, by Cauchy-Schwarz,
    \begin{align}
    \begin{aligned}\label{eq:Cauchy_schwarz}
        \sqrt{\var(H_{\ell-1}(G_k)H_{q-1}(G_r))} &\le \sqrt{\E[H_{\ell-1}(G_k)^2H_{q-1}(G_r)^2]} \\
        &\le \E[H_{\ell-1}(G_1)^4]^{1/4} \E[H_{q-1}(G_1)^4]^{1/4} \coloneqq \wt C_{q,\ell}<\infty.
    \end{aligned}
    \end{align} Thus, using $\theta \le |\rho(k-k')|+|\rho(k-r')|+|\rho(r-r')|+|\rho(r-k')|$ and that $\sum_{k,k',r,r'=0}^{n-1} |\rho(k-k')\rho(k-r)\rho(k'-r')|=\sum_{k,k',r,r'=0}^{n-1} |\rho(a-b)\rho(k-r)\rho(k'-r')|$ for $(a,b) \in \{(k,r'),(r,r'),(r,k')\}$, it holds that
    \begin{equation}\label{eq:proof_bound_indep_d}
        \|\delta(i,j,q,\ell,n)\|_2^2 \le \frac{4 a_{i,\ell}^2 a_{j,q}^2 q^2}{ n^2} \wt C_{q,\ell}^2 \sum_{k,k',r,r'=0}^{n-1} |\rho(k-k')\rho(k-r)\rho(k'-r')|. 
    \end{equation} For a sequence $s=(s_k)_{k \in \Z}$, recall that $\|s\|_{\ell^1(\Z)}=\sum_{k \in \Z}|s_k|$ and $\rho_n(k)= |\rho(k)|\1_{\{|k|<n\}}$ for all $k \in \Z$ and $n \in \N$. As in the proof of~\cite[Cor.~1.4~i)]{MR4488569}, applying Young's inequality for convolutions twice, implies
    \begin{align}\label{eq:ineq_Youngs_applic}
    \begin{aligned}
        \sum_{k,k',r,r'=0}^{n-1} |\rho(k-k')\rho(k-r)\rho(k'-r')|
        &\le 
        \sum_{k,r=0}^{n-1} (\rho_n * \rho_n * \rho_n)(k-r) \\
        &\le n\|\rho_n * \rho_n * \rho_n\|_{\ell^1(\Z)} \le n\|\rho_n\|_{\ell^1(\Z)}^3.
    \end{aligned}
    \end{align}
Define $K_n\coloneqq 2\|\rho_n\|_{\ell^1(\Z)}^{3/2}n^{-1/2}$ for all $n \in \N$. Using~\eqref{eq:proof_bound_indep_d} and~\eqref{eq:ineq_Youngs_applic}, it follows for all $n \in \N$ that
    \begin{equation}\label{eq:L1norm_max_Delta_n}
        \E\bigg[\max_{1 \le i,j \le d}|\Delta_n(i,j)|\bigg] \le K_n\sum_{q,\ell=2}^\infty  M_{\ell+q-2} \wt C_{q,\ell}q \log^{(\ell+q-2)/2}(2d^2-1+e^{(\ell+q)/2})\max_{1 \le i,j \le d}|a_{i,\ell} a_{j,q}|.
    \end{equation} 
    
    It remains to show that~\eqref{eq:L1norm_max_Delta_n} is finite and has the correct dimensional dependence. To show this, a tractable bound for $\wt C_{q,\ell}$ for any $q,\ell \ge 2$ is needed. By the fourth moment bound of Hermite polynomials from Lemma~\ref{lem:mom_bound_Hermite}, there exists a constant $c>0$, such that 
    \begin{equation*}
        \E[H_{\ell-1}(G_1)^4]^{1/4} \le c (\ell-1)^{(\ell-1)/2}e^{(\ell-1)(\log(3)/2-1/2)}(\ell-1)^{1/4}
        \le c \ell^{(\ell-1)/2}e^{(\ell-1)(\log(3)/2-1/2)}\ell^{1/4}, 
    \end{equation*}
    for all $\ell \ge 2$. Applying this bound 
    for $\wt C_{q,\ell}$, yields  
    \begin{equation}\label{eq:C_l_q_bound}
        \wt C_{q,\ell} \le c^2 \ell^{(\ell-1)/2}e^{(\ell-1)(\log(3)/2-1/2)}\ell^{1/4} q^{(q-1)/2}e^{(q-1)(\log(3)/2-1/2)}q^{1/4},
        \quad \text{ for all }\ell,q \ge 2.
    \end{equation} 
    Inserting~\eqref{eq:C_l_q_bound} and the specific choice of $M_{\ell+q-2}$, implies
    \begin{align}\label{eq:rate_bound_general}
    \begin{aligned}
    \E\big[\max_{1 \le i,j \le d}|\Delta_n(i,j)|\big]&\le 
         c^2 K_n\sum_{\ell,q=2}^\infty \Bigg\{\max_{1 \le i,j \le d}|a_{i,\ell} a_{j,q}|\left(\frac{4e\log(2d^2-1+e^{(\ell+q)/2})}{\ell+q-2}\right)^{(\ell+q-2)/2} \\
          &\qquad \times \ell^{(\ell-1)/2}e^{(\ell-1)(\log(3)/2-1/2)}\ell^{1/4}q^{(q-1)/2}e^{(q-1)(\log(3)/2-1/2)}q^{5/4}\Bigg\},
         \end{aligned}
    \end{align} for all $n,d \in \N$. Applying Lemma~\ref{lem:Theta_d_growth}(i) if $\beta=1/2$ and Lemma~\ref{lem:Theta_d_growth}(ii) if $\beta\in (1/2,1]$, with $f_q=\max_{1 \le i \le d}|a_{i,q}|$ together with the assumption on $a_{i,q}$, we can conclude from~\eqref{eq:bound_d_ge_3} that there exist constants $C_{\bm{\theta}}, K_{\bm{\theta}} >0$, only dependent on $\bm{\theta}$, such that
    \begin{align}\label{eq:bound_before_corre_1}
    \begin{aligned}
        d_\HR(\bm{S}_n,\bm{Z}_n)&\le 
    C_{\bm{\theta}} \log_+(d)\frac{\wh\Delta(\bm{S}_n,\bm{\Sigma}_n)}{\sigma_*^2}
    \log_+\left(\frac{\underline{\sigma}\wh\Delta(\bm{S}_n,\bm{\Sigma}_n)}{\overline{\sigma} \sigma_{*}^2}\right), \,  \text{ for all }n \in \N,\, d\ge 3 \text{ where}\\
    \wh\Delta(\bm{S}_n,\bm{\Sigma}_n)&= K_{\bm{\theta}} \frac{\|\rho_n\|_{\ell^1(\Z)}^{3/2}}{\sqrt{n}}\frac{e^{r\log_+^{1/(2\beta)}(d)}}{\log_+(d)}+\max_{1 \le i,j \le d}\left|(\bm{\Sigma}_n)_{i,j}-\E\left[S_{n,i}S_{n,j}\right]\right|.
    \end{aligned}
    \end{align} A similar bound holds for $d=1$ and $d=2$ by~\eqref{eq:bound_d_=1,=2}: 
    \begin{equation}\label{eq:bound_before_corre_2}
        d_\HR(\bm{S}_n,\bm{Z}_n)\le 
    C_{\bm{\theta}} \log_+(d)\frac{\wh\Delta(\bm{S}_n,\bm{\Sigma}_n)}{\sigma_*^2\wedge 1}
    \log_+\left(\frac{(\underline{\sigma}\wedge 1) \wh\Delta(\bm{S}_n,\bm{\Sigma}_n)}{(\overline{\sigma}\vee 1) (\sigma_{*}^2\wedge 1)}\right), \,  \text{ for all }n \in \N.
    \end{equation}

    Next, we note that the distance $d_\HR$ is invariant under multiplication with a diagonal matrix where the values are strictly positive. This means that $d_\HR(\bm{S}_n,\bm{Z}_n) =d_\HR(\bm{D} \bm{S}_n,\bm{D}\bm{Z}_n)$ where $\bm{D}=\Diag(p_1\ld p_d)$ and $p_i>0$ for $i=1\ld d$. Since taking infimum preserves soft inequalities, showing $d_\HR(\bm{S}_n,\bm{Z}_n) \le f(\bm{S}_n,\bm{Z}_n)$, implies for all $n,d \in \N$ that 
    \begin{equation}\label{eq:inf_d_HR_distance}
        d_\HR(\bm{S}_n,\bm{Z}_n)=\inf_{\bm{D} \in \mathcal{D}}d_\HR(\bm{D} \bm{S}_n,\bm{D}\bm{Z}_n) \le \inf_{\bm{D} \in \mathcal{D}}f(\bm{D}\bm{S}_n,\bm{D}\bm{Z}_n),
    \end{equation}
    where $\mathcal{D}\coloneqq \{\text{Diagonal matrices }\bm{D}\in \R^{d \times d}:\bm{D}_{i,i}>0 \text{ for }i=1\ld d\}$. We now choose $\bm{D}\in \mathcal{D}$ such that $\bm{D}^{1/2}\bm{\Sigma}_n\bm{D}^{1/2}=\bm{\Lambda}_n$. Then, $\underline{\sigma}(\bm{\Lambda}_n), \overline{\sigma}(\bm{\Lambda}_n)=1$ and $\bm{\Lambda}_n$ has an eigenvalue less than or equal to one, since $\bm{\Lambda}_n$ is a correlation matrix. Hence, by~\eqref{eq:inf_d_HR_distance},
    \begin{equation*}
        d_\HR(\bm{S}_n,\bm{Z}_n)\le 
    C \log_+(d)\frac{\wh\Delta(\bm{DS}_n,\bm{\Lambda}_n)}{\sigma_*^2(\bm{\Lambda}_n)}
    \log_+\left(\frac{\wh\Delta(\bm{DS}_n,\bm{\Lambda}_n)}{\sigma_{*}^2(\bm{\Lambda}_n)}\right), \quad \text{ for all }n,d \in \N.
    \end{equation*} Using the bound $x+y \le 2xy$ for $x,y\ge 1$, it follows that $\log_+(ab) \le 2\log_+(a) \log_+(b)$. Using this together with the fact that we can find some large enough $\wt K>1$, such that $\wh \Delta(\bm{S}_n,\bm{\Sigma}_n)\le  \wt K \Delta(\bm{S}_n,\bm{\Sigma}_n)$ (where $\Delta(\bm{S}_n,\bm{\Sigma}_n)$ is as defined in the statement of Theorem~\ref{thm:main_mult_clt_techncial_thm}), we can conclude the proof.
\end{proof}

The proof of Theorem~\ref{thm:Ext_to_convex_dist}, is mainly based on~\cite[Thm~1.2]{MR4488569} and the ideas from the proof of Theorem~\ref{thm:main_mult_clt_techncial_thm}. 
\begin{proof}[Proof of Theorem~\ref{thm:Ext_to_convex_dist}]
    By~\cite[Thm~1.2]{MR4488569}, it follows for all $d,n\ge 1$, that 
    \begin{equation*}
        d_\CD(\bm{S}_n,\bm{Z}_n) \le 402 d^{41/24} \sqrt{\E[\|\bm{M}_{\bm{S}_n}-\bm{\Sigma}_n\|_{\mathrm{H.S.}}^2]}(\|\bm{\Sigma}_n^{-1}\|_{\mathrm{op}}^{3/2}+1),
    \end{equation*} where $(\bm{M}_{\bm{S}_n}(i,j))_{1 \le i,j \le d} = (\langle -DL^{-1} S_{n,i}, D S_{n,j} \rangle_{\mH})_{1 \le i,j \le d}$. Hence, using the sub-additivity of $x \mapsto \sqrt{x}$ and of $\|\cdot\|_{\mathrm{H.S.}}$, it follows that
    \begin{equation*}
        d_\CD(\bm{S}_n,\bm{Z}_n) \le 402 d^{41/24} \left(\sqrt{\E[\|\bm{M}_{\bm{S}_n}-\cov(\bm{S}_n)\|_{\mathrm{H.S.}}^2]} + \|\cov(\bm{S}_n)-\bm{\Sigma}_n\|_{\mathrm{H.S.}}\right)(\|\bm{\Sigma}_n^{-1}\|_{\mathrm{op}}^{3/2}+1).
    \end{equation*} For the remainder of the proof, we will focus on bounding $\sqrt{\E[\|\bm{M}_{\bm{S}_n}-\cov(\bm{S}_n)\|_{\mathrm{H.S.}}^2]} $ for all $n,d \ge 1$. Note, by definition of $\bm{M}_{\bm{S}_n}$ together with~\eqref{eq:f_n,i_kernels}, that
    \begin{equation}\label{m_s_n_equality_1}
        \E[\|\bm{M}_{\bm{S}_n}-\cov(\bm{S}_n)\|_{\mathrm{H.S.}}^2]
        = \sum_{1 \le i,j \le d } \E\bigg[ \bigg(\sum_{q,\ell \ge 2} \delta(i,j,q,\ell,n) \bigg)^2 \bigg] \le d^2 \!\!\max_{1 \le i,j \le d} \bigg\|\sum_{q,\ell \ge 2} \delta(i,j,q,\ell,n)\bigg\|_2^2,
    \end{equation} where where $f_n(i,\ell)=a_{i,\ell} n^{-1/2} \sum_{k=1}^n e_k^{\otimes \ell} \in \mH^{\odot \ell} $ for all $n \in \N$ and $i \in \{1\ld d\}$, and $\delta(i,j,q,\ell,n) \coloneqq \ell^{-1}(\langle D I_\ell(f_n(i,\ell)),DI_q(f_n(j,q))\rangle_\mH-\E[\langle D I_\ell(f_n(i,\ell)),DI_q(f_n(j,q))\rangle_\mH])$, exactly as in~\eqref{eq_defn_delta_i,j,q,l}. Applying the triangle inequality in $L^2$
together with~\eqref{m_s_n_equality_1}, implies that
    \begin{align*}
        \sqrt{\E[\|\bm{M}_{\bm{S}_n}-\cov(\bm{S}_n)\|_{\mathrm{H.S.}}^2]} 
        \le d \sum_{q, \ell \ge 2} \max_{1 \le i,j \le d}\|\delta(i,j,q,\ell,n)\|_2 ,
    \end{align*} where we show below that the upper bound is finite. We can now use specific inequalities from the proof of Theorem~\ref{thm:main_mult_clt_techncial_thm}, namely~\eqref{eq:proof_bound_indep_d},~\eqref{eq:ineq_Youngs_applic} and~\eqref{eq:C_l_q_bound}, which implies that 
    \begin{multline*}
        \sqrt{\E[\|\bm{M}_{\bm{S}_n}-\cov(\bm{S}_n)\|_{\mathrm{H.S.}}^2]} 
        \le
        2d \frac{\|\rho_n\|_{\ell^1(\Z)}^{3/2}}{\sqrt{n}} \sum_{q,\ell \ge 2} q \max_{1 \le i,j \le d}|a_{i,\ell}a_{j,q}|\wt{C}_{q,\ell}\\
        \le 2d c^2 \frac{\|\rho_n\|_{\ell^1(\Z)}^{3/2}}{\sqrt{n}} \sum_{q,\ell \ge 2} \max_{1 \le i,j \le d} |a_{i,\ell}a_{j,q}|\ell^{(\ell-1)/2}e^{(\ell-1)(\log(3)/2-1/2)}\ell^{1/4} q^{(q-1)/2}e^{(q-1)(\log(3)/2-1/2)}q^{5/4},
    \end{multline*} where $\wt{C}_{q,\ell} = \E[H_{\ell-1}(G_1)^4]^{1/4}\E[H_{q-1}(G_1)^4]^{1/4}$. Applying assumption~\eqref{eq:main_assump_a_i,q}, implies the existence of some uniform constant $\wh C$, only dependent on $\bm{\theta}$ and $c$ from~\eqref{eq:C_l_q_bound}, such that
    \begin{align}
        &\sqrt{\E[\|\bm{M}_{\bm{S}_n}-\cov(\bm{S}_n)\|_{\mathrm{H.S.}}^2]}\le \wh C d \frac{\|\rho_n\|_{\ell^1(\Z)}^{3/2}}{\sqrt{n}} \Theta \nonumber \\
        &\qquad \coloneqq \wh C d \frac{\|\rho_n\|_{\ell^1(\Z)}^{3/2}}{\sqrt{n}} \sum_{q,\ell \ge 2} \frac{e^{\kappa (q+\ell)}}{(q!\ell!)^\beta} \ell^{(\ell-1)/2}e^{(\ell-1)(\log(3)/2-1/2)}\ell^{1/4} q^{(q-1)/2}e^{(q-1)(\log(3)/2-1/2)}q^{5/4}.\label{eq:finite_herm_exp_proof_d_C}
    \end{align} Hence, since the sum is independent of $d$, it only remains to show that the sum is finite. Let $w=\log(3)/2-1/2$, and note the multiplicative structure, given by
    \begin{align*}
        \Theta \coloneqq \left(\sum_{q = 2}^\infty \frac{e^{\kappa q}}{(q!)^\beta} q^{(q-1)/2}e^{w(q-1)}q^{5/4}  \right) \left(\sum_{\ell=2}^\infty \frac{e^{\kappa \ell}}{(\ell!)^\beta} \ell^{(\ell-1)/2}e^{w(\ell-1)}\ell^{1/4} \right).
    \end{align*} Due to almost matching sums (up to a polynomial factor), it suffices to show that the slowest converging sum is finite, i.e.\ $\wt{\Theta} \coloneqq \sum_{q = 2}^\infty e^{\kappa q}(q!)^{-\beta} q^{(q-1)/2}e^{w(q-1)}q^{5/4} <\infty$, to show that $\Theta<\infty$. Applying Stirling's inequality~\eqref{eq:Stirling's_formula}, we see that
    \begin{equation}\label{eq:bound_M_s_n_stirling}
        \wt{\Theta} \le e^{-w}(2\pi)^{-\beta/2}\sum_{q=2}^\infty e^{(\kappa+\beta+w) q} q^{(1/2-\beta) q}  q^{3/4-\beta/2}.
    \end{equation}
    
    Next, we consider the cases $\beta=1/2$ and $\beta>1/2$ separately. Assume that $\beta=1/2$, then~\eqref{eq:bound_M_s_n_stirling}, implies that $\wt{\Theta} \le e^{-w}(2\pi)^{-\beta/2}\sum_{q=2}^\infty e^{(\kappa+\beta+w) q} q^{1/2}$, where we note that $\sum_{q \ge 2} q^{1/2}e^{y q}<\infty$ if and only if $y<0$. Hence, it follows that $\wt\Theta$ (and hence $\Theta$) is finite if $\kappa<-1/2-w=-\log(3)/2$. 
    
    Assume instead that $\beta>1/2$. Then,~\eqref{eq:bound_M_s_n_stirling}, implies that
    \begin{align*}
         e^w (2\pi)^{\beta/2} \wt{\Theta} \le \sum_{q=2}^\infty e^{(\kappa+\beta+w) q} q^{(1/2-\beta) q}  q^{3/4-\beta/2}
         =
         \sum_{q=2}^\infty e^{(\kappa+\beta+w) q-(\beta-1/2)q\log(q)}  q^{3/4-\beta/2}.
    \end{align*} Recall that $\kappa \in \R$ is fixed and $\beta \in (1/2,1]$, thus there exists a $q_0$ which depends on $\kappa$ and $\beta$ and a constant $c_1>0$ which only depends on $\beta$, such that $(\kappa+\beta+w) q-(\beta-1/2)q\log(q) \le -c_1 q \log(q)$ for all $q \ge q_0$. To see why this holds, consider two cases: 1) if $\kappa+\beta+w<0$, the statement obviously holds, by choosing $q_0=2$ and $c_1=(\beta-1/2)>0$, 2) if $\kappa+\beta+w>0$, choose some $0<\nu < (\beta-1/2)$ and a $q_0\ge 2$, for which it holds that $(\kappa+\beta+w)q \le \nu q \log(q)$ for all $q \ge q_0$, concluding the statement with $c_1=(\beta-1/2)-\nu>0$. Thus, in all cases, it follows that 
    \begin{equation*}
         e^w (2\pi)^{\beta/2} \wt{\Theta} \le 
         \sum_{q=2}^{q_0} e^{(\kappa+\beta+w) q-(\beta-1/2)q\log(q)}  q^{3/4-\beta/2}
         + \sum_{q=q_0+1}^\infty e^{-c_1 q \log(q)}  q^{3/4-\beta/2}.
    \end{equation*} Since $(e^{-c_1 q \log(q)}  q^{3/4-\beta/2})_{q \ge 2}$ is summable (use e.g. the ratio-test), we can conclude that $\wt\Theta$ (and hence $\Theta$) is finite for all $\kappa \in \R$.

    Since $\bm{\Sigma}_n$ is invertible and a covariance matrix, it follows that $\|\bm{\Sigma}_n^{-1}\|_{\mathrm{op}}=1/\sigma_*^2(\bm{\Sigma}_n)$. Thus, we have so far that there exists a $C_{\bm{\theta}} >0$, depending only on $\bm{\theta}$, such that 
    \begin{equation}\label{eq:d_c_before_correlation}
        d_\CD(\bm{S}_n,\bm{Z}_n) \le C_{\bm{\theta}} d^{41/24} \Big(d \, n^{-1/2} \|\rho_n\|_{\ell^1(\Z)}^{3/2} + \|\cov(\bm{S}_n)-\bm{\Sigma}_n\|_{\mathrm{H.S.}}\Big)\big(\sigma_*^2(\bm{\Sigma}_n)^{-3/2}+1\big).
    \end{equation} Note that $d_\CD$ is invariant under multiplication with an invertible matrix (specifically a diagonal matrix where the values are strictly positive). This means that $d_\CD(\bm{S}_n,\bm{Z}_n) =d_\CD(\bm{D} \bm{S}_n,\bm{D}\bm{Z}_n)$ where $\bm{D}=\Diag(p_1\ld p_d)$ and $p_i>0$ for $i=1\ld d$. Hence, as in~\eqref{eq:inf_d_HR_distance}, we get that, showing $d_\CD(\bm{S}_n,\bm{Z}_n) \le f(\bm{S}_n,\bm{Z}_n)$, implies for all $n,d \in \N$ that $d_\HR(\bm{S}_n,\bm{Z}_n) \le \inf_{\bm{D} \in \mathcal{D}}f(\bm{D}\bm{S}_n,\bm{D}\bm{Z}_n)$,
    where $\mathcal{D}\coloneqq \{\text{Diagonal matrices }\bm{D}\in \R^{d \times d}:\bm{D}_{i,i}>0 \text{ for }i=1\ld d\}$. Choose $\bm{D}\in \mathcal{D}$ such that $\bm{D}^{1/2}\bm{\Sigma}_n\bm{D}^{1/2}=\bm{\Lambda}_n$, where $\bm{\Lambda}_n$ is the correlation matrix associated with the covariance matrix $\bm{\Sigma}_n$. This then implies that there exists a $C_{\bm{\theta}} >0$, depending only on $\bm{\theta}$, such that 
    \begin{equation}\label{eq:d_c_after_correlation}
        d_\CD(\bm{S}_n,\bm{Z}_n) \le C_{\bm{\theta}} d^{41/24} \left(d \frac{ \|\rho_n\|_{\ell^1(\Z)}^{3/2}}{\sqrt{n}} + \sqrt{\sum_{1 \le i,j \le d} \!\! \! \left(\frac{\E[S_{n,i}S_{n,j}]}{\sqrt{(\bm{\Sigma}_n)_{i,i}(\bm{\Sigma}_n)_{j,j}}}-(\bm{\Lambda}_n)_{i,j} \right)^2}\right)\big(\sigma_*^2(\bm{\Lambda}_n)^{-3/2}+1\big),
    \end{equation} for all $n,d \in \N$. Since $\bm{\Lambda}_n$ is a correlation matrix, it follows that $\sigma_*^2(\bm{\Lambda}_n) \le 1$, and hence $\sigma_*^2(\bm{\Lambda}_n)^{-3/2}+1 \le 2\sigma_*^2(\bm{\Lambda}_n)^{-3/2}$, concluding the proof. 
\end{proof}

\subsubsection{Proofs of Technical Lemmas}
The main estimate used in the proofs is Stirling's inequality~\cite[Eqs.~(1) \&~(2)]{MR69328}, given by
\begin{equation}\label{eq:Stirling's_formula}
    \sqrt{2\pi n}\left( \frac{n}{e}\right)^n 
    < n! <
    e^{\frac{1}{12}}\sqrt{2\pi n}\left( \frac{n}{e}\right)^n, \quad \text{ for all }n \ge 1.
\end{equation} Note that a similar asymptotic behaviour holds for the Gamma function. Indeed, by~\cite[Eq.~(SL)]{MR3026014}, it follows that $\Gamma(1+x) \sim \sqrt{2\pi x}\left( x/e\right)^x$ as $x \to \infty$. 

\begin{proof}[Proof of Lemma~\ref{lem:mom_bound_Hermite}]
By~\cite[Prop.~2.2.1 \&~Cor.~2.8.14]{MR2962301}, it follows that $$\E[H_\ell(G_1)^4]^{1/4} \le 3^{\ell/2}\E[H_\ell(G_1)^2]^{1/2} = 3^{\ell/2} \sqrt{\ell!}, \quad \text{ for all }\ell\ge 2.$$ Hence, applying Stirling's inequality~\eqref{eq:Stirling's_formula}, it follows that
\begin{equation*}
    \E[H_\ell(G_1)^4]^{1/4} \le 3^{\ell/2} (2\pi \ell)^{1/4} (\ell/e)^{\ell/2} = (2\pi)^{1/4} \ell^{1/4} e^{\ell(\log(3)/2-1/2)} \ell^{\ell/2}, \quad \text{ for all }\ell \ge 2.\qedhere
\end{equation*}
\end{proof}

\begin{proof}[Proof of Lemma~\ref{lem:Theta_d_growth}]
We start by proving general estimates, followed by separate proofs of cases (i) and (ii). For all $\ell,q \ge 2$, we have $(\ell+q-2)^{(\ell+q-2)/2}\ge \ell^{(\ell-1)/2} q^{(q-1)/2}$, hence
    \begin{equation*}
         \Theta(d) \le \sum_{\ell=2}^\infty   \left(2e^{\log(3)/2}\right)^{\ell-1}\ell^{1/4}f_{\ell}  \sum_{q=2}^\infty  \log^{(\ell+q-2)/2}(2d^2-1+e^{(\ell+q)/2})\left(2e^{\log(3)/2}\right)^{q-1}q^{5/4}f_q.
    \end{equation*}
 Next, define the sub-additive function $g(x)=\log(x+1)$, and note that
    \begin{equation*}
        \log(2d^2-1+e^{(\ell+q)/2})=g(2d^2-2+e^{(\ell+q)/2})
        \le\log(2d^2)+\frac{\ell+q}{2},
    \end{equation*} by using the sub-additivity $g(x+y)\le g(x)+g(y)$ and choosing $x=2d^2-1$ and $y=e^{(\ell+q)/2}-1$. Thus, applying the inequality $(x+y)^p \le 2^{p}(x^p+y^p)$ for $x,y \in [0,\infty)$ and $p\ge 1$, it follows that
    \begin{equation}\label{eq:log_ineq}
        \log^{(\ell+q-2)/2}(2d^2-1+e^{(\ell+q)/2})\le 2^{(\ell+q-2)/2}\left(\log^{(\ell+q-2)/2}(2d^2)+\left(\frac{\ell+q}{2}\right)^{(\ell+q-2)/2}\right).
    \end{equation} 
    For all $q \ge 2$, define 
    \begin{equation}\label{eq:defn_c_q}
        \upsilon_q \coloneqq \left(2^{3/2}e^{\log(3)/2}\right)^{q-1}q^{5/4}f_q.
    \end{equation}
    Then, using~\eqref{eq:log_ineq}, it follows that
    \begin{equation} \label{eq:up_bound_theta_d}
    \Theta(d) 
    \le 
    \sum_{\ell =2}^\infty \frac{\upsilon_\ell }{\ell} \sum_{q=2}^\infty \upsilon_q\log^{(\ell+q-2)/2}(2d^2)
        + \sum_{\ell =2}^\infty  \frac{\upsilon_\ell }{\ell} \sum_{q=2}^\infty \upsilon_q\left(\frac{\ell+q}{2}\right)^{(\ell+q-2)/2} \eqqcolon \Theta_1(d)+\Theta_2. 
    \end{equation} Note that only $\Theta_1$ depends on $d$, hence $\Theta_1(d)$ determines the dimensional dependence of $\Theta(d)$. Let $u(d)\coloneqq \log^{1/2}(2d^2)$ for $d \ge 2$, and note that 
    \begin{equation}\label{eq:Theta_1(d)_bound}
        \Theta_1(d) 
        \le 
        \bigg( u(d)^{-1}\sum_{q=2}^\infty u(d)^{q}\upsilon_q\bigg)^2.
    \end{equation}

Part (i). Fix $\alpha \in (0,\sqrt{W(e^{-1 + 1/(2 e)})/e^{1/(2e)}})$, and recall from the assumptions that there exists a $c>0$ such that $f_q \le c\alpha^qe^{-(\log(24)/2+5/(4e))q}/\sqrt{q!}$ for all $q \ge 2$. By definition of $\upsilon_q$ from~\eqref{eq:defn_c_q}, it holds that there exists a universal constant $\wt c>0$, such that
    \begin{equation}\label{eq:c_q_bound}
        \upsilon_q \le\frac{\wt c \alpha^q\left(2\sqrt{2}e^{\log(3)/2}\right)^{q}q^{5/4}}{e^{(\log(24)/2+5/(4e))q}\sqrt{q!}}= \wt c \frac{ \alpha^q (e^{\log(24)/2+5\log(q)/(4q)})^q}{e^{(\log(24)/2+5/(4e))q}\sqrt{q!}}
        \le 
        \frac{\wt c \alpha^q }{\sqrt{q!}} , \quad \text{ for all }q\ge 2.
    \end{equation}
    To show~\eqref{eq:c_q_bound}, it was also used that $\log(3)/2+3\log(2)/2=\log(24)/2$, and that $q \mapsto \log(24)/2+5\log(q)/(4q)$ attains its maximal value at $q=e$ with maximal value attained being $\log(24)/2+5/(4e)$. This holds, since $\tfrac{\D }{\D q} (\log(24)/2+5\log(q)/(4q))=5(1-\log(q))/(4 q^2)=0$ has one solution at $q=e$ and since the derivative is positive for $q<e$ and negative for $q>e$. 
    Using~\eqref{eq:Theta_1(d)_bound} and~\eqref{eq:c_q_bound}, it follows that
    \begin{equation*}
        \Theta_1(d) \le \frac{\wt c^2}{u(d)^2}\left(\sum_{q=2}^\infty \frac{\left(u(d)\alpha\right)^{q}}{\sqrt{q!}} \right)^2 
        \le \frac{\wt c^2}{u(d)^2}\left(\sum_{q =0}^\infty  \frac{\left(u(d)\alpha\right)^{q}}{\sqrt{q!}} \right)^2. 
    \end{equation*} Next, let $\wt u(d)\coloneqq (u(d)\alpha)^2$, and use Stirling's inequality~\eqref{eq:Stirling's_formula} for $q!$ and Stirling's inequality for the gamma function $\Gamma(1+q/2)$, it follows that there exists a uniform constant $C>0$, such that
    \begin{align}
    \begin{aligned}\label{eq:sum_sqrt(q!)}
        \sum_{q =0}^\infty  \frac{\wt u(d)^{q/2}}{\sqrt{q!}} 
        &\le 
        \left( \frac{\pi e^{1/3}}{2}\right)^{1/4}\sum_{q=0}^\infty \frac{q^{1/4}(\wt u(d)/2)^{q/2}}{\sqrt{2\pi(q/2)}(q/(2e))^{q/2}e^{1/12}}
        \le
        C\sum_{q=0}^\infty \frac{(e^{1/(2e)}\wt u(d)/2)^{q/2}}{\Gamma(1+q/2)} \\
        &=
        Ce^{e^{1/(2e)}\wt u(d)/2}\left(\erf\left(e^{\tfrac{1}{4e}}\frac{\sqrt{\wt u(d)}}{\sqrt{2}}\right)+1\right)
        \le 
        2C(2d^2)^{e^{1/(2e)}\alpha^2/2}.
    \end{aligned}
    \end{align} In the second inequality of~\eqref{eq:sum_sqrt(q!)}, that $q^{1/4}=(e^{\log(q)/(2q)})^{q/2}$, and that $q \mapsto \log(q)/(2q)$ is maximised at $q=e$ with value $1/(2e)$. In third inequality of~\eqref{eq:sum_sqrt(q!)}, we used properties of the Mittag--Leffler function, yielding that $
    \sum_{q=0}^\infty x^{q/2}/\Gamma(1+q/2)=e^x(\erf(\sqrt{x})+1)$, where $\erf$ is the error function $\erf(x) \coloneqq \tfrac{2}{\sqrt{\pi}}\int_0^x e^{-t^2}\D t$.
    Since $\log(2d^2) \ge \log(2)\log_+(d)$ for all $d \ge 1$, we can conclude that there exists a $K_{\bm{\theta}}$, only dependent on $\bm{\theta}$, such that $\Theta_1(d) \le K_{\bm{\theta}}d^{r}/\log_+(d) $ for all $d \ge 1$.

    It remains to check if $\Theta_2<\infty$. By assumption on $f_q$ together with~\eqref{eq:c_q_bound}, we have $\Theta_2<\infty$ if  
    \begin{equation*}
        K \coloneqq \sum_{\ell =2}^\infty  \frac{\alpha^\ell}{\sqrt{\ell!}\ell } \sum_{q=2}^\infty \frac{\alpha^q}{\sqrt{q!}} \left(\frac{\ell+q}{2}\right)^{(\ell+q)/2}<\infty.
    \end{equation*} Applying $(x+y)^p \le 2^{p}(x^p+y^p)$ for $x,y \ge 0$ and $p \ge 1$ and $1/\ell \le 1$ for all $\ell \ge 2$, it follows that
    \begin{align}\label{eq:inequalities_K_bound}
        K 
        &\le
        \sum_{\ell=2}^\infty  \frac{\alpha^\ell}{\sqrt{\ell!}\ell } \sum_{q=2}^\infty \frac{\alpha^q}{\sqrt{q!}} \ell^{(\ell+q)/2}+ \sum_{\ell =2}^\infty  \frac{\alpha^\ell}{\sqrt{\ell!}\ell } \sum_{q=2}^\infty \frac{\alpha^q}{\sqrt{q!}}q^{(\ell+q)/2}
         \le 2 \sum_{\ell =2}^\infty  \frac{\alpha^\ell \ell^{\ell/2}}{\sqrt{\ell!} } \sum_{q=2}^\infty \frac{\alpha^q \ell^{q/2}}{\sqrt{q!}}.
         \end{align} Using~\eqref{eq:sum_sqrt(q!)}, there exists a uniform constant $C$, such that
         $$\sum_{q =0}^\infty \frac{\alpha^q \ell^{q/2}}{\sqrt{q!}}=\sum_{q =0}^\infty \frac{(\ell \alpha^2)^{q/2}}{\sqrt{q!}}\le 
         C e^{\ell e^{1/(2e)}\alpha^2/2}.$$ Together with the inequality $(q/e)^q \le q!$ and~\eqref{eq:inequalities_K_bound}, this implies that
        \begin{align}\label{eq_Up_bound_K_1}
        K\le 
        2C \sum_{\ell =2}^\infty  \frac{ \ell^{\ell/2}}{\sqrt{\ell!}}\left(e^{ e^{1/(2e)}\alpha^2/2}\alpha\right)^\ell
        \le 2 C \sum_{\ell =2}^\infty \left(e^{ e^{1/(2e)}\alpha^2/2+1/2}\alpha\right)^\ell.
    \end{align} To see that the upper bound in~\eqref{eq_Up_bound_K_1} is finite, note that $\sum_{k \ge 2} \beta^k<\infty$ if and only if $|\beta|<1$, and that $\beta=e^{ e^{1/(2e)} \alpha^2/2+1/2}\alpha<1$ for all $\alpha<\sqrt{W(e^{-1 + 1/(2 e)})/e^{1/(2e)}}$. Since this is the case for our choice of $\alpha$, we can conclude that $\Theta_2<\infty$.
        
    Part (ii). Recall by assumption that $f_q \le ce^{\kappa q}/(q!)^\beta$ for all $q \ge 2$ and let $\varpi\coloneqq e^{\kappa+\log(24)/2+5/(4e)}$. Hence, by~\eqref{eq:c_q_bound}, there exists a universal constant $\wt c>0$, such that
    \begin{align}\label{eq:c_q_inequality_2}
        \upsilon_q \le
        \wt c \frac{  e^{\kappa q}e^{\log(24)q/2+5\log(q)/4}}{(q!)^\beta}\le \frac{\wt c  e^{\kappa q}(e^{\log(24)/2+5\log(q)/(4q)})^q}{(q!)^\beta}
        \le 
        \frac{\wt c \varpi 
        ^q }{(q!)^\beta} , \quad \text{ for all }q \ge 2.
    \end{align} By~\eqref{eq:Theta_1(d)_bound} and~\eqref{eq:c_q_inequality_2}, it follows that
    \begin{equation*}
        \Theta_1(d) \le \frac{\wt c^2}{u(d)^2}\left(\sum_{q=2}^\infty \frac{\left(u(d)\varpi 
        \right)^{q}}{(q!)^\beta} \right)^2 
        \le \frac{\wt c^2}{u(d)^2}\left(\sum_{q =0}^\infty  \frac{\left(u(d)\varpi 
        \right)^{q}}{(q!)^\beta} \right)^2. 
    \end{equation*} Next, let $\wt u(d)\coloneqq (u(d)\varpi)^{1/\beta}$. The inequality $(\sum_{i=0}^\infty |x_i|^b)^{1/b} \le (\sum_{i=0}^\infty |x_i|^a)^{1/a}$ for $0<a \le b$, with $b=2\beta \in (1,2]$ and $a=1$, yields that
    \begin{equation}\label{eq:sum_(q!)^beta}
        \sum_{q =0}^\infty  \left( \frac{\wt u(d)^{ q}}{q!}\right)^\beta =  \left(\sum_{q =0}^\infty  \left( \sqrt{\frac{\wt u(d)^{ q}}{q!}}\right)^{2\beta}\right)^{\frac{2\beta}{2\beta}} \le \left(\sum_{q =0}^\infty  \sqrt{\frac{\wt u(d)^{ q}}{q!}}\right)^{2\beta}.
    \end{equation} Note that $\log(2d^2) \ge \log(2)\log_+(d)$ and $\log^{1/(2\beta)}(2d^2) \le \log(2)^{1/(2\beta)}+2^{1/(2\beta)}\log_+(d)^{1/(2\beta)}$ for all $d \ge 1$. Hence, applying~\eqref{eq:sum_sqrt(q!)} and~\eqref{eq:sum_(q!)^beta}, it follows that there exists a $K_{\bm{\theta}}$, only dependent on $\bm{\theta}$, such that 
    \begin{equation*}
     \Theta_1(d) 
     \le
     \frac{\wt c^2 2^{4\beta}\exp\left\{2e^{1/(2e)}\beta \varpi^{1/\beta}
     \log^{1/(2\beta)}(2d^2)\right\}}{\log(2d^2)}
     \le 
     K_{\bm{\theta}} \frac{e^{r\log_+^{1/(2\beta)}(d)}}{\log_+(d)}, \, \text{ for all }d \ge 1.
    \end{equation*}

    Finally, we need to check that $\Theta_2<\infty$. By the assumption on $f_q$ together with~\eqref{eq:c_q_inequality_2}, it suffices to show that
    \begin{equation*}
        K \coloneqq \sum_{\ell =2}^\infty  \frac{\varpi^\ell}{(\ell!)^\beta\ell } \sum_{q=2}^\infty \frac{\varpi^q}{(q!)^\beta} \left(\frac{\ell+q}{2}\right)^{(\ell+q)/2}<\infty.
    \end{equation*} Note that $(x+y)^p \le 2^{p}(x^p+y^p)$ for $x,y \ge 0$ and $p \ge 1$ and $1/\ell \le 1$ for all $\ell \ge 2$. Hence, by~\eqref{eq:inequalities_K_bound}, it follows that $K \le 
        2 \sum_{\ell =2}^\infty  \varpi^\ell \ell^{\ell/2}(\ell!)^{-\beta } \sum_{q =2}^\infty \varpi^q \ell^{q/2}(q!)^{-\beta}$. Applying~\eqref{eq:sum_sqrt(q!)} and~\eqref{eq:sum_(q!)^beta}, yields that
         $$\sum_{q =0}^\infty \frac{\varpi^q \ell^{q/2}}{(q!)^\beta}=\sum_{q =0}^\infty \left(\frac{\varpi^{q/(2\beta)}\ell^{q/(4\beta)}}{\sqrt{q!}}\right)^{2\beta}\le 
         2e^{e\beta\varpi^{1/\beta}\ell^{1/(2\beta)}},$$ which, together with the inequality $q^q \le q!e^q$, implies that
        \begin{equation*}
        K\le 4 \sum_{\ell =2}^\infty  \frac{ \varpi^\ell \ell^{\ell/2}}{(\ell!)^\beta} e^{e\beta\varpi^{1/\beta}\ell^{1/(2\beta)}}
        \le 
        4\sum_{\ell =2}^\infty \frac{\varpi^\ell e^{\ell/2}}{(\ell!)^{\beta-1/2}}e^{e\beta\varpi^{1/\beta}\ell^{1/(2\beta)}}
        \eqqcolon 4\sum_{\ell =2}^\infty \theta_\ell.
    \end{equation*} We now have that
    \begin{equation*}
        \frac{\theta_{\ell+1}}{\theta_\ell} =\frac{\varpi e^{1/2}\exp\left\{e\beta\varpi^{1/\beta}((\ell+1)^{1/(2\beta)}-\ell^{1/(2\beta)})\right\}}{\ell^{\beta-1/2}} \to 0, \quad \text{ as } \ell \to \infty,
    \end{equation*} where we used that $\beta\in (1/2,1]$ (i.e.\ $1/(2\beta) \in [1/2,1)$), $e^{a((\ell+1)^b-\ell^b)} \to 1 $ as $\ell \to \infty$ for all $a\in \R$ and $b \in (0,1)$. The proof can now be concluded by applying the ratio test to see that $\sum_{\ell =2}^\infty \theta_\ell<\infty$, and hence $\Theta_2< \infty$.
\end{proof}

\section{Proofs of the Results from Section~\ref{sec:main_results}}\label{sec:proofs_sec_1}

\begin{proof}[Proof of Theorem~\ref{thm:main_theorem}]
\textbf{Proof of~\eqref{eq:main_result_inequality}.} Under the assumptions of Theorem~\ref{thm:main_theorem} we are in the setting of Theorem~\ref{thm:main_mult_clt_techncial_thm} with $\bm{\Sigma}_n=\cov(\bm{S}_n)$. Hence, by Theorem~\ref{thm:main_mult_clt_techncial_thm} there exists a $\wt{C}_{\bm{\theta}}>0$, only dependent on $\bm{\theta}$, such that
\begin{align}\label{eq:inequality_first_step}
    d_\HR(\bm{S}_n,\bm{Z})&\le 
    \wt{C}_{\bm{\theta}} \frac{e^{r\log_+^{1/(2\beta)}(d)}\|\rho_n\|_{\ell^1(\Z)}^{3/2}}{\sqrt{n}\sigma_*^2}
    \log_+\left(\frac{\|\rho_n\|_{\ell^1(\Z)}^{3/2}e^{r\log_+^{1/(2\beta)}(d)}}{\sigma_*^2\sqrt{n}\log_+(d)}\right), \quad \text{ for all }n,d \in \N.
    \end{align} By the triangle inequality, the inequality $x+y \le 2xy$ for $x,y\ge 1$ and $\log_+(ab) \le 2\log_+(a)\log_+(b)$ for all $a,b \in (0,\infty)$, it follows that
\begin{equation}\label{eq:inequ_split_constant}
    \log_+\left(\frac{\|\rho_n\|_{\ell^1(\Z)}^{3/2} e^{r\log_+^{1/(2\beta)}(d)}}{ \sigma_*^2\sqrt{n}\log_+(d)}\right)  \le 4\log_+(\sigma_*^2) \log_+\left(\sqrt{n^{-1}\|\rho_n\|_{\ell^1(\Z)}^{3}}\right) \log_+\left(\frac{e^{r\log_+^{1/(2\beta)}(d)}}{\log_+(d)}\right).
    \end{equation} 
    Moreover, there exists a finite constant $c_{\bm{\theta}}>0$, dependent only on $\bm{\theta}$, for which
    \begin{equation}\label{eq:inequality_psi(d)_proof}
        e^{r\log_+^{1/(2\beta)}(d)} \log_+\left(\frac{e^{r\log_+^{1/(2\beta)}(d)}}{\log_+(d)}\right)
        \le c_{\bm{\theta}} \log_+^{1/(2\beta)}(d)e^{r\log_+^{1/(2\beta)}(d)}=c_{\bm{\theta}}\psi_{\beta,\kappa}(d), 
    \end{equation} 
    for all $d \ge 1$. Thus, by~\eqref{eq:inequality_first_step} and~\eqref{eq:inequality_psi(d)_proof}, there exists a $C_{\bm{\theta}}>0$, depending only on $\bm{\theta}$, such that
    \begin{equation}
    d_\HR(\bm{S}_n,\bm{Z}_n)\le 
    C_{\bm{\theta}} \psi_{\beta,\kappa}(d) 
    n^{-1/2}\|\rho_n\|_{\ell^1(\Z)}^{3/2} \log_+\left(n^{-1/2}\|\rho_n\|_{\ell^1(\Z)}^{3/2}\right)\frac{\log_+(\sigma^2_*(\bm{\Lambda}_n))}{\sigma^2_*(\bm{\Lambda}_n)}, \quad \text{ for all }n,d \in \N.
\end{equation} 

\textbf{Proof of~\eqref{eq:main_result_inequality_d_C}.} Under the assumptions of Theorem~\ref{thm:main_theorem} we are in the setting of Theorem~\ref{thm:Ext_to_convex_dist} with $\bm{\Sigma}_n=\cov(\bm{S}_n)$. Hence,the bound for $d_\CD$ in~\eqref{eq:main_result_inequality_d_C} follows directly by Theorem~\ref{thm:Ext_to_convex_dist}. 

Using~\eqref{m_s_n_equality_1} and~\eqref{eq:finite_herm_exp_proof_d_C} from the proof of Theorem~\ref{thm:Ext_to_convex_dist}, similar bounds to the ones found in Theorems~\ref{thm:main_theorem} \&~\ref{thm:fixed_d} can be constructed for $d_\mW$. Indeed, for $\bm{Z} \sim \mathcal{N}_d(0,\bm{\Sigma}_n)$,~\cite[Eq.~(12)]{MR4488569} implies that $d_\mW(\bm{S}_n,\bm{Z}_n) \le \sqrt{d} \, \|\bm{\Sigma}^{-1}_n\|_{\mathrm{op}}\|\bm{\Sigma}_n\|_{\mathrm{op}}^{1/2} \sqrt{\E[\|\bm{M}_{\bm{S}_n}-\bm{\Sigma}_n\|_{\mathrm{H.S.}}^2]}$, where $\bm{M}_{\bm{S}_n}$ is given in Section~\ref{sec:bascis_malliavin_proofs} above, and the a bound on $\E[\|\bm{M}_{\bm{S}_n}-\bm{\Sigma}_n\|_{\mathrm{H.S.}}^2]$ can be found via~\eqref{m_s_n_equality_1} and~\eqref{eq:finite_herm_exp_proof_d_C}, as is used to prove Theorem~\ref{thm:Ext_to_convex_dist}. Hence, it follows that
\begin{equation*}
    d_\mW(\bm{S}_n,\bm{Z}_n) \le C_{\bm{\theta}} d^{3/2} \, \|\bm{\Sigma}^{-1}_n\|_{\mathrm{op}}\|\bm{\Sigma}_n\|_{\mathrm{op}}^{1/2} \, \| \rho_n \|_{\ell^1(\Z)}^{3/2}n^{-1/2},
\end{equation*} where $(\sigma_*^2)^{-1}(\bm{\Sigma}_n)= \|\bm{\Sigma}^{-1}_n\|_{\mathrm{op}}$ and $\sigma_\dag^2(\bm{\Sigma}_n)=\|\bm{\Sigma}_n\|_{\mathrm{op}}$. 
\end{proof}

\begin{proof}[Proof of Theorem~\ref{thm:fixed_d}] To prove Theorem~\ref{thm:fixed_d}, we apply inequalities from the proof of Theorems~\ref{thm:main_mult_clt_techncial_thm} \&~\ref{thm:Ext_to_convex_dist}. Indeed, using that $\log_+(ab) \le 2\log_+(a)\log_+(b)$, together with~\eqref{eq:bound_before_corre_1} and~\eqref{eq:bound_before_corre_2}, it follows that there exists a $C_{\bm{\theta}}>0$, only dependent on $\bm{\theta}$, such that for all $n,d \in \N$,
 \begin{align*}
    d_\HR(\bm{S}_n,\bm{Z})&\le 
    C_{\bm{\theta}} \log_+(d)\Delta(\bm{S}_n,\bm{\Sigma})
    \log_+\left(\Delta(\bm{S}_n,\bm{\Sigma})\right)\frac{\log_+\left(\overline{\overline{\sigma}} \,\wt\sigma_{*}^2/\,\underline{\underline{\sigma}}\right)}{\wt\sigma_*^2}, \, \text{ where}\\
    \Delta(\bm{S}_n,\bm{\Sigma})&=  \frac{\|\rho_n\|_{\ell^1(\Z)}^{3/2}}{\sqrt{n}}\frac{e^{r\log_+^{1/(2\beta)}(d)}}{\log_+(d)}+\max_{1 \le i,j \le d}\left|\bm{\Sigma}_{i,j}-\E\left[S_{n,i}S_{n,j}\right]\right|.
    \end{align*} Similarly, by~\eqref{eq:d_c_before_correlation}, it follows that
    \begin{equation*}
    d_\CD(\bm{S}_n,\bm{Z}) \le C_{\bm{\theta}} d^{41/24} \left(d \frac{ \|\rho_n\|_{\ell^1(\Z)}^{3/2}}{\sqrt{n}} + \sqrt{\sum_{1 \le i,j \le d} \!\! \! \left(\E[S_{n,i}S_{n,j}]-\bm{\Sigma}_{i,j} \right)^2}\right)\big((\sigma_*^2)
    ^{-3/2}+1\big).
    \end{equation*}

    Using~\eqref{m_s_n_equality_1} and~\eqref{eq:finite_herm_exp_proof_d_C} from the proof of Theorem~\ref{thm:Ext_to_convex_dist}, together with~\cite[Eq.~(12)]{MR4488569}, it follows that
    $$d_\mW(\bm{S}_n,\bm{Z}) \le C_{\bm{\theta}} \sqrt{d}  \, \left(d \frac{ \|\rho_n\|_{\ell^1(\Z)}^{3/2}}{\sqrt{n}} + \sqrt{\sum_{1 \le i,j \le d} \!\! \! \left(\E[S_{n,i}S_{n,j}]-\bm{\Sigma}_{i,j} \right)^2}\right)\frac{\sigma_\dag}{\sigma_*^2}.$$
    
    Hence, it remains only to consider $\bm{\Sigma}_{i,j}-\E[S_{n,i}S_{n,j}]$. To bound this quantity, note by~\cite[Prop.~2.2.1]{MR2962301}, that
    \begin{align}
        \E[S_{n,i}S_{n,j}]&=\frac{1}{n} \sum_{k,r=1}^n \E[\varphi_i(G_k)\varphi_j(G_r)] = \sum_{q,\ell =m}^\infty \frac{a_{i,q}a_{j,\ell}}{n} \sum_{k,r=1}^n \E[H_q(G_k)H_\ell (G_r)] \\
        &= 
        \sum_{\ell =m}^\infty \frac{\ell! a_{i,\ell}a_{j,\ell}}{n} \sum_{k,r=1}^n \rho(k-r)^\ell =
        \sum_{\ell =m}^\infty \ell! a_{i,\ell}a_{j,\ell} \sum_{k \in \Z: |k|<n } \bigg(1-\frac{|k|}{n}\bigg)\rho(k)^\ell.
    \end{align} Altogether, it follows that
    \begin{equation*}
        |\bm{\Sigma}_{i,j}-\E[S_{n,i}S_{n,j}]| \le \left(  \sum_{|k| \ge n} |\rho(k)|^m + \frac{1}{n} \sum_{|k|<n} |k| |\rho(k)|^m \right) \sum_{\ell=m}^\infty \ell! \max_{1 \le i,j \le d} |a_{i,\ell}a_{j,\ell}|, 
    \end{equation*} where we note that $\sum_{\ell=m}^\infty \ell! \max_{1 \le i,j \le d} |a_{i,\ell}a_{j,\ell}|<\infty$ is a finite constant which depends only on $m$ and $\bm{\theta}$, due to~\eqref{eq:main_assump_a_i,q} and Stirling's inequality~\eqref{eq:Stirling's_formula}. Thus, since $x\mapsto x\log_+(x)$ is increasing, we may conclude the proof of the theorem.
\end{proof}

\begin{proof}[Proof of Corollary~\ref{cor:LRD_SRD_fixed cov:mat}]
For the ensuing proof, we only consider $d_\HR$ and $d_\CD$, since the proof of $d_\mW$ follows \emph{exactly} in the same way as $d_\CD$.

\textbf{Part~(i).} We are in the setting of Theorem~\ref{thm:fixed_d}, and will apply this theorem. Assume that $|\rho(k)| \le \wt c \,|k|^{-\mu}L(|k|)$ for $\mu \ge 1$ and $L \in \mathrm{SV}_\infty$ (if $\mu=1$ assume that $L\in \mathrm{SV}_\infty$ is such that $\|\rho\|_{\ell^1(\Z)}<\infty$), implying that $\|\rho\|_{\ell^1(\Z)}<\infty$. Hence, since $\|\rho_n\|_{\ell^1(\Z)}\le \|\rho\|_{\ell^1(\Z)}$ for all $n \in \N$, it remains to bound $\sum_{|j| \ge n} |\rho(j)|^m $ and $ n^{-1}\sum_{|j|<n} |j| \cdot|\rho(j)|^m$. First, using the assumption on $\rho$, together with Karamata's theorem~\cite[Prop.~1.5.10]{MR898871}, implies the existence of a constant $K>0$ such that, 
\begin{equation*}
    \sum_{|j| \ge n} |\rho(j)|^m \le 2 \wt c^m \sum_{j=n}^\infty j^{-\mu m}L(j)^m \le \frac{2 K \wt c^m}{\mu m-1} n^{1-\mu m}L(n)^m, \quad \text{ for all }n \in \N.  
\end{equation*} Since $m \ge 2$ and $\mu \ge 1$, it follows that $\mu m-1 \ge 1$, and hence $\sum_{|j| \ge n} |\rho(j)|^m$ converges to $0$ faster than $n^{-1/2}$ as $n \to \infty$. Next, by assumption on $\rho$, it follows directly, that
\begin{equation*}
    \frac{1}{n}\sum_{|j|<n} |j| \cdot|\rho(j)|^m \le \frac{2 \wt c^m}{n} \sum_{j=1}^n j^{1-\mu m} L(j)^m \le \frac{2 \wt c^m}{n} \sum_{j=1}^\infty j^{1-\mu m} L(j)^m\eqqcolon \frac{K_m}{n}.
\end{equation*} Note that $K_m <\infty$, since if $\mu>1$, then $\mu m >2$, and hence $j^{1-\mu m}L(j)^m$ is obviously summabe, since $L(j)^m \in \mathrm{SV}_\infty$. If instead $\mu=1$, then $L \in \mathrm{SV}_\infty$ such that $\sum_{j=1}^\infty j^{-1 }L(j) <\infty$, hence $L$ is eventually decreasing and tends to $0$ as $j \to \infty$. Assume therefore without loss of generality, that $L(j)<1$ for all $j >k_0$ for some $k_0 \in \N$.  Then 
\begin{equation*}
    \sum_{j=1}^\infty j^{1-m} L(j)^m = \sum_{j=1}^{k_0} j^{1- m} L(j)^m + \sum_{j=k_0+1}^\infty j^{1- m} L(j)^m < \sum_{j=1}^{k_0} j^{1- m} L(j)^m + \sum_{j=k_0+1}^\infty j^{-1} L(j),
\end{equation*} where the upper bound is finite, since $\sum_{j=1}^\infty j^{-1 }L(j) <\infty$. Altogether, it therefore follows that there exists a constant $K$, such that 
\begin{equation*}
n^{-1/2}\|\rho_n\|_{\ell^1(\Z)}^{3/2}
    +\sum_{|j| \ge n} |\rho(j)|^m + n^{-1}\sum_{|j|<n} |j| \cdot |\rho(j)|^m \le K n^{-1/2}
    , \quad \text{ for all }n \in \N.
\end{equation*} Hence, using $\log_+(xy) \le 2 \log_+(x)\log_+(y)$ and that $x \mapsto x\log_+(x)$ is increasing, together with Theorem~\ref{thm:fixed_d}, it follows that there exists a $C_{\bm{\theta},\rho,m}$, which depends only on $\bm{\theta}$, $\mu$, $\wt c$ and $m$, such that
\begin{align*}
    d_\HR(\bm{S}_n,\bm{Z})&\le 
    C_{\bm{\theta},\rho,m} \log_+(d)\wt\Delta(\bm{S}_n,\bm{\Sigma})\log_+\left( \wt\Delta(\bm{S}_n,\bm{\Sigma})\right)
\frac{\log_+\left(\,\overline{\overline{\sigma}} \, \wt\sigma_{*}^2/\,\underline{\underline{\sigma}}\,)\right)}{\wt\sigma_*^2}, \, \text{ where}\\
    \wt\Delta(\bm{S}_n,\bm{\Sigma})&= n^{-1/2}e^{r\log_+^{1/(2\beta)}(d)}/\log_+(d),
    \end{align*} and $d_\CD(\bm{S}_n,\bm{Z}) \le C_{\bm{\theta},\rho,m} d^{65/24} n^{-1/2}\big((\sigma_*^2)
    ^{-3/2}+1\big)$. Hence, the proof of the bound under $d_\CD$ follows directly. The bound under $d_\HR$ follows now by using the exact steps from the proof of~\eqref{eq:main_result_inequality} from Theorem~\ref{thm:main_theorem}.

\textbf{Part~(ii).} Assume that $|\rho(k)| \le \wt c \,|k|^{-\mu}L(|k|)$ for $\mu \in (2/3,1)$ and $L \in \mathrm{SV}_\infty$. We start by showing, that $\|\rho_n\|_{\ell^1(\Z)}\le R n^\alpha S(n)$ for all $n \in \N$, for some $R>0$, $\alpha \in (0,1/3)$ and $S \in \mathrm{SV}_\infty$. Since, $\mu \in (2/3,1)$, it follows by~\cite[Prop.~1.5.8]{MR898871} that there exists some $K>0$, such that $\sum_{j=1}^{n} j^{-\mu} L(j) \le K n^{1-\mu}L(n)$ for all $n \in \N$. Hence, altogether, it follows that $\|\rho_n\|_{\ell^1(\Z)} \le R n^{1-\mu}L(n)$ for all $n \in \N$ where $1-\mu \in (0,1/3)$.

Next, we will bound $\sum_{|j| \ge n} |\rho(j)|^m $ and $ n^{-1}\sum_{|j|<n} |j| \cdot|\rho(j)|^m$. By~\cite[Prop.~1.5.10]{MR898871}, it follows that $\sum_{|j| \ge n} |\rho(j)|^m \le K n^{1-\mu m } L(n)^m$, for all $n \in \N$, since $-\mu m <-1$. To bound $n^{-1}\sum_{|j|<n} |j| \cdot|\rho(j)|^m $, we split into the two cases where $m=2$ and $m\ge 3$. If $m \ge 3$, then $\mu m > 2$, and hence $\sum_{k \in \Z} |j|\cdot |\rho(j)|^m<\infty$, as shown in the proof of Part~(i) above. Hence, $n^{-1}\sum_{|j|<n} |j| \cdot|\rho(j)|^m \le K n^{-1}$, and it remains to consider the case $m=2$. In this case, we have that $1-2\mu \in (-1,-1/3)$, and hence,~\cite[1.5.8]{MR898871} yields that $n^{-1}\sum_{|j|<n} |j| \cdot|\rho(j)|^2 \le K n^{1-2\mu}L(n)^2$.

Altogether, there exists some $\epsilon>0$ small, such that, for all $n \in \N$, it follows that
\begin{equation}\label{eq:gen_bound_n_dependence}
    \frac{\|\rho_n\|_{\ell^1(\Z)}^{3/2}}{\sqrt{n}}
    +\sum_{|j| \ge n} |\rho(j)|^m + \sum_{|j|<n} \frac{|j|}{n} |\rho(j)|^m \le 
        R n^{(2-3\mu)/2}L(n)^{3/2}. 
\end{equation}

The proof of~\eqref{eq:cor_main_lrd_d_C_fix_sigma} now follows directly from~\eqref{eq:gen_bound_n_dependence} and Theorem~\ref{thm:fixed_d}. Hence, it only remains to prove~\eqref{eq:cor_main_lrd_d_R_fix_sigma}. For the remainder of the proof, we handle both cases of~\eqref{eq:gen_bound_n_dependence} at the same time, by writing $Rn^\alpha S(n)$ for the upper bound, where $S$ and $\alpha$ can be found in~\eqref{eq:gen_bound_n_dependence}. By Theorem~\ref{thm:fixed_d}, together with~\eqref{eq:gen_bound_n_dependence}, the proof of Theorem~\ref{thm:main_theorem} and since $x \mapsto x\log_+(x)$ is increasing on $(0,\infty)$, there exists a $\wh{C}_{\bm{\theta}}>0$, dependent only on $\bm{\theta}$, such that
\begin{equation}\label{eq:bound_before_LRD_step}
    d_\HR(\bm{S}_n,\bm{Z}_n)\le 
    \wh{C}_{\bm{\theta}} \psi_{\beta,\kappa}(d) Rn^\alpha S(n) \log_+\left(Rn^\alpha S(n)\right)\frac{\log_+\left(\,\overline{\overline{\sigma}} \, \wt\sigma_{*}^2/\,\underline{\underline{\sigma}}\,)\right)}{\wt\sigma_*^2},
\end{equation} for all $n,d \in \N$. Thus, by the inequality $\log_+(ab) \le 2\log_+(a)\log_+(b)$ for all $a,b \in (0,\infty)$, it follows that $Rn^\alpha S(n) \log_+(R n^\alpha S(n)) \le 2R \log_+(R)n^\alpha S(n)\log_+(n^\alpha S(n))$. Hence, together with~\eqref{eq:bound_before_LRD_step}, this implies the existence of a $C_{\bm{\theta},R,\alpha}>0$, which depends only on $\bm{\theta}$, $R$ and $\alpha$, such that
    \begin{equation*}
        d_\HR(\bm{S}_n,\bm{Z}_n)\le C_{\bm{\theta},R,\alpha}\frac{}{} 
        \psi_{\beta,\kappa}(d)n^{\alpha}S(n)\log_+(n^\alpha S(n)) \frac{\log_+\left(\,\overline{\overline{\sigma}} \, \wt\sigma_{*}^2/\,\underline{\underline{\sigma}}\,)\right)}{\wt\sigma_*^2}, \quad \text{ for all } n,d \in \N.\qedhere
    \end{equation*}
\end{proof}

\subsection{Proof of Corollary~\ref{cor:method_moments_hermite_var}}\label{sec:proofs_thm_2.4&2.5}

\begin{proof}[Proof of Corollary~\ref{cor:method_moments_hermite_var}]
    The proof is based on Corollary~\ref{cor:finite_hermite_series}, which is proved in Section~\ref{sec:proofs_sec_2} below. Note that we are in the setting of Corollary~\ref{cor:finite_hermite_series} since the $\varphi_i$'s are the scaled Hermite polynomials and therefore have a finite Hermite series, i.e.\ $N=\max_{i \in \{1\ld d\}}\sup\{q \ge 0:a_{i,q}\ne 0\}=d+1<\infty$. By definition of $\varphi_i$, it follows that $a_{i,q}=0$ for all $q \ne i+1$ and $a_{i,i+1}=((i+1)!)^{-1/2}$, implying that~\eqref{eq:main_assump_a_i,q} holds for $c=1$, $\kappa=0$ and $\beta=1/2$. Moreover, recall that $\sigma_*^2=1$ by construction of $\bm{S}_n$. Applying Corollary~\ref{cor:finite_hermite_series} with $N=d+1$ and $\bm{\theta}=(1/2,0,1)$, together with the ineqluality $(d+1)^{5/2} \le 2^{5/2}d^{5/2}$ for all $d \ge 1$, yields a uniform constant $C$, such that
    \begin{equation*}
        d_\CD(\bm{S}_n,\bm{Z}_n)\le C d^{125/24} e^{\log(3) d} n^{-1/2} \|\rho_n\|_{\ell^1(\Z)}^{3/2} 
        , \quad \text{ for all }n,d \in \N. \qedhere
    \end{equation*} 
\end{proof}

\subsection{Proof of Corollary~\ref{cor:real_part_charac_fun} and Proposition~\ref{prop:specif_theta_i}}\label{sec:eigenvalues_bound_fGn}
In this section, we will show Corollary~\ref{cor:real_part_charac_fun} and show that certain choices of $\varphi_i$ yield that the smallest eigenvalues of $\bm{\Lambda}_n=\cor(\bm{S}_n)$ are uniformly bounded away from $0$. We start by proving Corollary~\ref{cor:real_part_charac_fun}.
\begin{proof}[Proof of Corollary~\ref{cor:real_part_charac_fun}]
For any choice of $\varphi_i$, i.e.\ $\varphi^{\cos}_i$, $\varphi^{\sin}_i$ or $\varphi^{e}_i$, it follows that $\varphi_i \in L^2(\gamma,\R)$ and $\E[\varphi_i(G_1)]=0=\E[\varphi_i(G_1)G_1]$, implying that $m_i \ge 2$. 
To prove that $\E[\varphi_i(G_1)]=0=\E[\varphi_i(G_1)G_1]$, we start with the case $x\mapsto \varphi^{\cos}_i(x)$, and note that
\begin{equation*}
    \E[\cos(\lambda_i G_1)]=\int_\R \cos(\lambda_i x) \gamma(\D x)=\frac{1}{\sqrt{2\pi}}\int_\R \cos(\lambda_i x) e^{-x^2/2}\D x=e^{-\lambda_i^2/2},
\end{equation*} and similarly $\E[(\cos(\lambda_i G_1)-e^{-\lambda_i^2/2})G_1]=0$. Likewise, it follows directly that $\E[\sin(\lambda_i G_1)]=0$ and $\E[\sin(\lambda_i G_1)G_1]=e^{-\lambda_i^2/2}$, and lastly that $\E[e^{\lambda_i G_1}]=e^{\lambda_i^2/2}$ and $\E[e^{\lambda_i G_1} G_1]=\lambda_i e^{\lambda_i^2/2}$.

Fix a $K\in (0,\infty)$ such that $\lambda_i \in (0,K]$ for all $i=1 \ld d$. The proof is split into three parts, one for each choice of function $\varphi_i$. For any choice of $\varphi_i$, we note that $\E[\varphi_i(G_1)]=0=\E[\varphi(G_1)G_1]$ as explained in the paragraph above. For each choice of $\varphi_i$, we show that Assumptions I) and II) from Lemma~\ref{lem:c_infty_main_result} are satisfied. Note that all choices of $\varphi_i$ are in $C^\infty$, and $\varphi_i^{(q)}\in L^2(\gamma,\R)$, since $\varphi_i^{(q)}(x)^2 e^{-x^2/2} \to 0$ as $x \to \pm \infty$ and $\varphi_i^{(q)}(x)^2 e^{-x^2/2}$ is bounded on any compact interval. Hence, Assumption~I) of Lemma~\ref{lem:c_infty_main_result} is satisfied for all choices of $\varphi_i$. 

Part~(i). Assume that $\varphi_i=\varphi_i^{\cos}$. To show Assumption~II) from Lemma~\ref{lem:c_infty_main_result} is satisfied, it suffices to find $\kappa \in \R$ and $c \in (0,\infty)$, such that $\big| \int_{-\infty}^\infty \varphi_i^{(q)}(x)\gamma(\D x)\big|\le c e^{\kappa q} $ for all $q \ge 2$. To show that such a $\kappa$ exists, note that $\varphi_i^{(q)}(x)
=\lambda_i^{q}\cos\left(q\pi/2+\lambda_i x\right)$ for all $q \ge 2$ and $x \in \R$, implying that
\begin{align*}
    \left| \int_{-\infty}^\infty \varphi_i^{(q)}(x)
    \gamma(\D x)\right| 
    = 
    \lambda_i^{q} \left| \int_{-\infty}^\infty \cos\left(q\pi/2+\lambda_i x\right)
    \gamma(\D x)\right|
    \le 
    \lambda_i^{q} e^{-\lambda_i^2/2} 
    \le e^{q \log(K)}, 
\end{align*} since $|\int_{-\infty}^\infty \cos(q\pi/2+\lambda_i x) \D x|=e^{-\lambda_i^2/2}$ if $q$ is even and equal to $0$ if $q$ is odd. Hence, Assumption~II) is satisfied for $c=1$ and $\kappa=\log(K)$. Lemma~\ref{lem:c_infty_main_result} now implies that~\eqref{eq:main_assump_a_i,q} is satisfied with $\bm{\theta}=(1,\log(K),1)$. Since $\sigma_*^2=\sigma_*^2(\bm{\Lambda}_n)>0$ by assumption, Theorem~\ref{thm:main_theorem} yields a constant $C_{K}>0$, which depends only on $K$, but is independent of $n$ and $d$, which concludes the proof of part~(i).

Part~(ii). Assume that $\varphi_i=\varphi_i^{\sin}$. To show that Assumption~II) from Lemma~\ref{lem:c_infty_main_result} holds, we find $\kappa \in \R$ and $c \in (0,\infty)$, such that $\big| \int_{-\infty}^\infty \varphi_i^{(q)}(x)\gamma(\D x)\big|\le c e^{\kappa q} $ for all $q \ge 2$. To show that such a $\kappa$ exists, note that $\varphi_i^{(q)}(x)=\lambda_i^{q}\sin\left(q\pi/2+\lambda_i x\right)$, implying that $\big| \int_{-\infty}^\infty \varphi_i^{(q)}(x)
    \gamma(\D x)\big| \le e^{q\log(K)}$, exactly as in the case of cosine. Hence, II) is satisfied for $c=1$ and $\kappa=\log(K)$. Lemma~\ref{lem:c_infty_main_result} now implies that~\eqref{eq:main_assump_a_i,q} is satisfied with $\bm{\theta}=(1,\log(K),1)$. Since $\sigma_*^2=\sigma_*^2(\bm{\Lambda}_n)>0$ by assumption, Theorem~\ref{thm:main_theorem} yields a constant $C_{K}>0$, which depends only on $K$, but is independent of $n$, $d$, concluding the proof of part~(ii).

Part~(iii). Assume that $\varphi_i=\varphi_i^{e}$. To show Assumption~II) from Lemma~\ref{lem:c_infty_main_result}, we find $\kappa \in \R$ and $c \in (0,\infty)$, such that $\big| \int_{-\infty}^\infty \varphi_i^{(q)}(x)\gamma(\D x)\big|\le c e^{\kappa q} $ for all $q \ge 2$. To show that such a $\kappa$ exists, note that $\varphi_i^{(q)}(x)=\lambda_i^q e^{\lambda_i x}$ for all $q \ge 2$ and $x \in \R$, implying that 
\begin{equation*}
    \left| \int_{-\infty}^\infty \varphi_i^{(q)}(x) \gamma(\D x)\right|  = \lambda_i^q  \int_{-\infty}^\infty e^{\lambda_ix} \gamma(\D x)=\lambda_i^q e^{\lambda_i^2/2}\le e^{K^2/2}e^{q\log(K)}, \quad \text{ for all }q \ge 2.
\end{equation*} Hence, assumption II) is satisfied for $c=e^{K^2/2}$ and $\kappa=\log(K)$. Finally, Lemma~\ref{lem:c_infty_main_result} implies that~\eqref{eq:main_assump_a_i,q} is satisfied with $\bm{\theta}=(1,\log(K),e^{K^2/2})$. Since $\sigma_*^2=\sigma_*^2(\bm{\Lambda}_n)>0$, Theorem~\ref{thm:main_theorem} yields a constant $C_{K}>0$, which depends only on $K$, but is independent of $n$ and $d$, concluding the proof.
\end{proof}

    Assume that $B^H=(B_t^H)_{t \in \R}$ is a \emph{fractional Brownian motion} (fBm)~\cite[Def.~2.6.2]{MR3729426}, which means that $B^H$ is a Gaussian process on $\R$ where $B_0^H=0$, $\E[B_t^H]=0$ and $B^H$ has covariance $\E[B_t^H B_s^H]=\frac{1}{2}(|t|^{2H}+|s|^{2H}-|t-s|^{2H})$ for $t,s \in \R$, where $H \in (0,1)$ is the \emph{Hurst parameter}. Moreover, we define the \emph{fractional Gaussian noise} (fGn) as the increments of $B^H$, i.e.\ $G_k\coloneqq B^H_{k+1}-B^H_k$ for $k \in \N$, with $\rho(k)
    =\frac{1}{2}(|k+1|^{2H}+|k-1|^{2H}-2|k|^{2H})$ for $k \in \Z$. Note that $G_k \sim \mathcal{N}(0,1)$ for any $k \in \N$. In the ensuing proof, the notation $\sum_{|v| \le n-1}=\sum_{v \in \Z:|v|\le n-1}$ is used.
\begin{proof}[Proof of Proposition~\ref{prop:specif_theta_i}]
    \textbf{Part~(i)}. To show part~(i) of the lemma, we note that it suffices to show that $\inf_{n,d \in \N}\sigma_*^2(\bm{\Lambda}_n) \ge 1-\xi(\tau)>0$, if $\xi(\tau)\coloneqq  
        \|\rho\|_{\ell^2(\Z)}^2
        \zeta(\tau)^{-1}<1$.
We need the following technical result, which is a direct consequence of Gershgorin circle theorem~\cite{zbMATH02560682}. Let $\bm{A}$ be a covariance matrix. Then, the following holds:
\begin{equation}\label{eq:lower_bound_eigenvalues}
    \text{if $a\ge 0$ exists such that $\|\bm{A}-\bm{I}_d\|_\infty \le a$, then $\sigma_*^2(\bm{A})\ge 1-a$.}
\end{equation}
By~\eqref{eq:lower_bound_eigenvalues}, it thus suffices to show that $\|\bm{\Lambda}_n-\bm{I}_d\|_\infty \le \xi(\tau)$ for all $n,d \in \N$, since this will imply that $\inf_{n,d \in \N}\sigma_*^2(\bm{\Lambda}_n)\ge 1-\xi(\tau)>0$ whenever $\xi(\tau)<1$. Since $\sum_{v=1}^{n-1}\rho(v)^2$ is increasing in $n$, the ensuing equation~\eqref{eq:bound_infty_norm_corr_1} will imply $\|\bm{\Lambda}_n-\bm{I}_d\|_\infty \le \xi(\tau)$ for all $n,d \in \N$. Hence, it suffices to show
\begin{align}
    \begin{aligned}\label{eq:bound_infty_norm_corr_1}
        \|\bm{\Lambda}_n-\bm{I}_d\|_\infty
         &\le   
        \frac{1+2\sum_{v=1}^{n-1}\rho(v)^2}{2e^{1/2}(1-e^{-1})(1-e^{-2^{2\tau}})}\\
        &\qquad \times 
        \Bigg( \frac{e^{2^\tau }}{e^{2^{2\tau}/2}}\left(2+\frac{1}{\tau 2^{\tau-1}(2^\tau-1)}\right)
    +
    \frac{e^{3^\tau}}{e^{3^{2\tau}/2}}\left( 1+\frac{1}{\tau 3^{\tau-1}(3^\tau-1)}\right)\Bigg).
    \end{aligned}
    \end{align}
    
    In the remainder of part (i), we show~\eqref{eq:bound_infty_norm_corr_1}. By definition of $\varphi_i$ and~\cite[Eq.~(1.2)]{MR3306374}, it follows that
   \begin{align*}
    \cos(\lambda_i x)-e^{-\lambda_i^2/2}
    =\sum_{q=1}^\infty \frac{(-1)^q \lambda_i^{2q} e^{-\lambda_i^2/2}}{(2q)!}H_{2q}(x)=\sum_{q=2}^\infty 
    \1_{\{\text{even}\}}(q)
    \frac{(-1)^{q/2} \lambda_i^{q} e^{-\lambda_i^2/2}}{q!}H_{q}(x),
\end{align*} which implies that $a_{i,q}=0$ if $q$ is odd, and $a_{i,q}=(-1)^{q/2} \lambda_i^{q} e^{-\lambda_i^2/2}(q!)^{-1}$ if $q$ is even. Recall the notation $S_{n,i}= n^{-1/2}\sum_{k=1}^n \varphi_i(G_k)$. Applying the specific form of $a_{i,q}$ and~\cite[Prop.~2.2.1]{MR2962301}, it follows that
\begin{align*}
    \cov(S_{n,i},S_{n,j})
    &=
    \frac{1}{n} \sum_{k,r=1}^n \sum_{q,\ell =2}^\infty a_{i,q}a_{j,\ell} \E[H_q(G_k)H_\ell(G_r)]
    =
    \sum_{|v|\le n-1}\sum_{q=2}^\infty a_{i,q}a_{j,q} q!\rho(v)^q\\
    &=
    \sum_{|v|\le n-1}\sum_{q=1}^\infty \frac{(\lambda_i \lambda_j )^{2q} e^{-\lambda_i^2/2-\lambda_j^2/2}}{(2q)!}\rho(v)^{2q}.
\end{align*} Considering the case $i=j$ and using the equality $\sum_{n=1}^\infty x^{2n}/(2n)!=\cosh(x)-1$, it follows that
\begin{equation*}
    \var(S_{n,i})
    =
    e^{-\lambda_i^2}\sum_{|v|\le n-1}\sum_{q=1}^\infty \frac{(\lambda_i^2\rho(v))^{2q}}{(2q)!}
    =
    e^{-\lambda_i^2}\sum_{|v|\le n-1}(\cosh(\lambda_i^2 \rho(v))-1).
\end{equation*}Altogether, it can thus be deduced that $\bm{\Lambda}_n$ has the form
    \begin{equation*}
        (\bm{\Lambda}_n)_{i,j}
        =
        \frac{\cov(S_{n,i},S_{n,j})}{\sqrt{\var(S_{n,i})}\sqrt{\var(S_{n,j})}}
        =
        \frac{e^{-\lambda_i^2/2-\lambda_j^2/2}
    \sum_{q=1}^\infty 
    (\lambda_i\lambda_j)^{2q}((2q)!)^{-1}  \sum_{|v|\le n-1} \rho(v)^{2q} }{\prod_{m\in \{i,j\}}\left( \sum_{|v| \le n-1}e^{-\lambda_m^2}(\cosh(\lambda_m^2|\rho(v)|)-1)\right)^{1/2}}, 
    \end{equation*} for all $i,j =1 \ld d$. Since $\bm{\Lambda}_n$ is a correlation matrix, it holds that $(\bm{\Lambda}_n)_{i,i}=1$ for all $i=1\ld d$. Hence, since $\rho(v) \le 1$ and by definition of $\|\cdot\|_{\infty}$, it follows that
\begin{align*}
    \|\bm{\Lambda}_n-\bm{I}_d\|_\infty
    &=
    \max_{j=1\ld d}
    \sum_{i\in \{1\ld d\}\setminus\{j\}}
    \frac{e^{-\lambda_i^2/2-\lambda_j^2/2}
    \sum_{q=1}^\infty 
    (\lambda_i\lambda_j)^{2q}((2q)!)^{-1}  \sum_{|v|\le n-1} \rho(v)^{2q} }{\prod_{m\in \{i,j\}}\left( \sum_{|v| \le n-1}e^{-\lambda_m^2}(\cosh(\lambda_m^2|\rho(v)|)-1)\right)^{1/2}}\\
    &\le 
    \max_{j=1\ld d}
    \sum_{i\in \{1\ld d\}\setminus\{j\}}
    \frac{e^{-\lambda_i^2/2-\lambda_j^2/2}
    (\cosh(\lambda_i\lambda_j)-1) \sum_{|v|\le n-1} \rho(v)^{2} }{\prod_{m\in \{i,j\}}\left( \sum_{|v| \le n-1}e^{-\lambda_m^2}(\cosh(\lambda_m^2|\rho(v)|)-1)\right)^{1/2}}.
\end{align*}

Fix now $j \in \{1\ld  d\}$, and note for all $i \in \{1\ld d\}\setminus\{j\}$ that
\begin{equation}\label{eq:general_assumption_rho_dep}
    \frac{\sum_{|v| \le n-1}\rho(v)^2}{\prod_{m\in \{i,j\}}\left( \sum_{|v| \le n-1}e^{-\lambda_m^2}(\cosh(\lambda_m^2|\rho(v)|)-1)\right)^{1/2}}
    \le 
    \frac{1+2\sum_{v=1}^{n-1}\rho(v)^2}{(1-e^{-1})(1-e^{-(2\vee j)^{2\tau}})},
\end{equation} where we used that 
\begin{align*}
    \prod_{m \in \{i,j\}} \left(\sum_{|v| \le n-1}e^{-\lambda_m^2}(\cosh(\lambda_m^2|\rho(v)|)-1)\right)^{1/2} & = 
    \prod_{m \in \{i,j\}}\left(\sum_{|v| \le n-1}e^{-\lambda_m^2(1-|\rho(v)|)}(1-e^{-\lambda_m^2 |\rho(v)|})^2\right)^{1/2}\\
    &\ge 
    (1-e^{-1})(1-e^{-(2 \vee j)^{2\tau}}).
\end{align*}

For fixed $j \in \{1\ld d\}$ and using that $\lambda_i=i^\tau$, it now remains to consider
\begin{align}\label{eq:overall-bound_toshow}
\begin{aligned}
     &\sum_{i\in \{1\ld d\}\setminus\{j\}} e^{-\lambda_i^2/2-\lambda_j^2/2}(\cosh(\lambda_i\lambda_j)-1) \\
    &\qquad = 
    e^{-j^{2\tau}/2}\left( \sum_{i=1}^{j-1}e^{-i^{2\tau}/2}(\cosh(i^\tau j^\tau)-1)
    +
    \sum_{i=j+1}^d e^{-i^{2\tau}/2}(\cosh(i^\tau j^\tau)-1)\right).
\end{aligned}
\end{align} Note, if $j=1$ then $\sum_{j=1}^0 \equiv 0$ and similarly, if $j=d$, then $\sum_{j=d+1}^d \equiv 0$. Let now $j \ge 2$, and use the inequality $\cosh(x)-1 \le e^{x}/2$ for $x \ge 0$, such that
\begin{equation*}
e^{-j^{2\tau}/2}\sum_{i=1}^{j-1}e^{-i^{2\tau}/2}(\cosh(i^\tau j^\tau)-1) \le \frac{1}{2}\sum_{i=1}^{j-1}e^{-i^{2\tau}/2-j^{2\tau}/2+i^\tau j^\tau}
= 
\frac{1}{2}\sum_{i=1}^{j-1} f(i,j),
\end{equation*} where $f(i,j)\coloneqq e^{-i^{2\tau}/2-j^{2\tau}/2+i^\tau j^\tau}$. The next step is to prove that
\begin{equation}\label{eq:up_bound_alpha_j_terms}
    \sum_{i=1}^{j-1} f(i,j) \le e^{-1/2}\left( e^{2^\tau-2^{2\tau}/2}+e^{3^\tau-3^{2\tau}/2}\left( 1+\frac{1}{\tau 3^{\tau-1}(3^\tau-1)}\right)\right).
\end{equation}
To show~\eqref{eq:up_bound_alpha_j_terms}, we start by proving that 
\begin{equation}\label{eq:bound_f_j(i)}
    f(i,j) \le h(i,j)\coloneqq e^{-1/2-(j-i+1)^{2\tau}/2+(j-i+1)^\tau}, \qquad \text{ for }i=1\ld j-1.
\end{equation} To prove this statement, we note that $f(x,y)$ is a continuous function, and it therefore suffices to show that $f(x,y) \le h(x,y)$ for all $x,y \in \R$ where $1 \le x \le y-1$. Hence, if we fix an arbitrary $x_0\ge 1$, then it suffices to show that
\begin{equation}\label{eq:cont_vers_inequality}
    f(x_0,y)=e^{-x_0^{2\tau}/2-y^{2\tau}/2+x_0^\tau y^\tau} \le e^{-1/2-(y-x_0+1)^{2\tau}/2+(y-x_0+1)^{\tau}}=h(x_0,y), \text{ for all }y \ge x_0.
\end{equation} Next, by dividing through by the right-hand side and taking $\log$, we see that~\eqref{eq:cont_vers_inequality} is equivalent to 
\begin{equation*}
    g(y)\coloneqq 
    -x_0^{2\tau}/2-y^{2\tau}/2+x_0^\tau y^\tau +1/2+(y-x_0+1)^{2\tau}/2-(y-x_0+1)^{\tau} \le 0, \text{ for all }y \ge x_0.
\end{equation*} By rewriting, we see that $g(y)= -(y^\tau-x_0^\tau)^2/2+((y-x_0+1)^\tau-1)^2/2$ for all $y \ge x_0$. Hence, since $y^\tau \ge x_0^\tau$ and $y-x_0+1 \ge 0$ for all $y \ge x_0$, it suffices to show that 
\begin{equation}\label{eq:reduced_eq_tau_func}
    h(y)\coloneqq (y-x_0+1)^\tau-1 -y^\tau + x_0^\tau \le 0, \quad \text{ for all } y \ge x_0.
\end{equation} Since~\eqref{eq:reduced_eq_tau_func} holds for $y=x_0$, it suffices to show that $y \mapsto h(y)$ is a non-increasing function for $y \ge x_0$, i.e.\ 
\begin{equation*}
    \frac{\D}{\D y} h(y) = \tau(y-x_0+1)^{\tau-1}-\tau y^{\tau-1}\le 0, \quad \text{ for all }y \ge x_0.
\end{equation*} However, this follows directly from the fact that $y \mapsto \tau y^{\tau-1}$ is a non-increasing function for all $y \ge 1$, since $\tau \ge 1$, and since $x_0 \ge 1$. Altogether,~\eqref{eq:cont_vers_inequality} thus holds.

Applying~\eqref{eq:cont_vers_inequality}, it follows that
\begin{equation}\label{eq:part1_reduct_ffun}
    e^{1/2} \sum_{i=1}^{j-1}f(i,j) 
    \le 
    \sum_{i=1}^{j-1} e^{-(j-i+1)^{2\tau}/2+(j-i+1)^\tau }
    = 
    \sum_{i=2}^{j} e^{-i^{2\tau}/2+i^\tau }
    \le 
     e^{2^\tau-2^{2\tau}/2 }+\sum_{i=3}^{\infty } e^{-i^{2\tau}/2+i^\tau } .
\end{equation} Hence, since $x \mapsto e^{-x^{2\tau}/2+x^\tau}$ is decreasing on $[3,\infty)$, it follows by change of variable that
\begin{equation*}
    \sum_{i=3}^{\infty } e^{-i^{2\tau}/2+i^\tau } 
    \le 
    e^{-3^{2\tau}/2+3^\tau }+\int_3^\infty e^{-x^{2\tau}/2+x^\tau } \D x
    = 
    e^{-3^{2\tau}/2+3^\tau }+\frac{\sqrt{2e}}{\tau}\int_{3^\tau/\sqrt{2}-1/\sqrt{2}}^\infty e^{-y^2 } (\sqrt{2}y+1)^{-1+1/\tau} \D y.
\end{equation*} Next, using that $(\sqrt{2}y+1)^{-1+1/\tau} \le 3^{1-\tau}$ for all $y \ge 3^\tau/\sqrt{2}-1/\sqrt{2}$ and that $\int_x^\infty e^{-y^2}\D y \le e^{-x^2/(2x)} $, it follows that
\begin{equation}\label{eq:part2_reduct_ffun}
    \sum_{i=3}^{\infty } e^{-i^{2\tau}/2+i^\tau } 
    \le 
    e^{-3^{2\tau}/2+3^\tau }\left(1+\frac{1}{\tau 3^{\tau-1}(3^\tau -1)}\right).
\end{equation}
Hence, putting together~\eqref{eq:part1_reduct_ffun} and~\eqref{eq:part2_reduct_ffun}, we can conclude that~\eqref{eq:up_bound_alpha_j_terms} holds.

Assume now that $j \ne d$. Applying the inequality $\cosh(x)-1 \le e^{x}/2$ for $x \ge 0$, together with the fact that $x \mapsto e^{-x^{2\tau}/2+x^\tau j^\tau}$ is decreasing on $(j,\infty)$, yields
\begin{equation}\label{eq:part1_largeJ}
    2\sum_{i=j+1}^d e^{-i^{2\tau}/2}(\cosh(i^\tau j^\tau)-1) \le \sum_{i=j+1}^\infty e^{-i^{2\tau}/2+i^\tau j^\tau}\le \frac{e^{(j+1)^\tau j^\tau}}{e^{(j+1)^{2\tau}/2}}+\int_{j+1}^\infty e^{-x^{2\tau}/2+x^\tau j^\tau} \D x.
\end{equation}
Next, using the change of variable $y=x^\tau/\sqrt{2}-j^\tau/\sqrt{2}$, and $\int_x^\infty e^{-y^2}\D y \le e^{-x^2}/(2x)$, we see that
\begin{equation}\label{eq:part2_largeJ}
    \int_{j+1}^\infty e^{-x^{2\tau}/2+x^\tau j^\tau} \D x \le \frac{e^{j^{2\tau}/2}}{\tau (j+1)^{\tau-1}} \int_{(j+1)^\tau/\sqrt{2}-j^\tau/\sqrt{2}}^\infty e^{-y^2} \D y\le \frac{e^{-(j+1)^{2\tau}/2+(j+1)^\tau j^\tau }}{\tau (j+1)^{\tau-1}((j+1)^\tau-j^\tau)}.
\end{equation} Finally, we show that
\begin{equation}\label{eq:part3_largeJ}
    \frac{e^{(j+1)^\tau j^\tau }}{e^{(j+1)^{2\tau}/2+j^{2\tau}/2}}\left(1+\frac{1}{\tau (j+1)^{\tau-1}((j+1)^\tau-j^\tau)}\right)\\
    \le  
    \frac{e^{2^\tau }}{e^{1/2+2^{2\tau}/2}}\left(1+\frac{1}{\tau 2^{\tau-1}(2^\tau-1)}\right).
\end{equation} First, we note that $e^{ -(j+1)^{2\tau}/2-j^{2\tau}/2+(j+1)^\tau j^\tau}
    \le  
    e^{-1/2-2^{2\tau}/2+2^\tau }$, which follows directly by showing that $x \mapsto e^{(x+1)^\tau x^\tau -(x+1)^{2\tau}/2-x^{2\tau}/2}$ is non-increasing for $x \ge 1$, and hence bounded by the value at $x=1$. Indeed, since $e^{(x+1)^\tau x^\tau -(x+1)^{2\tau}/2-x^{2\tau}/2}=e^{-((x+1)^\tau -x^{\tau})^2/2}$ and since $x \mapsto (x+1)^\tau-x^\tau$ is non-decreasing for $x \ge 1$, the statement follows. Hence,~\eqref{eq:part3_largeJ} follows if we can show that $x \mapsto \tau (x+1)^{\tau-1}((x+1)^\tau-x^\tau)$ is non-decreasing for all $x \ge 1$, and hence attains its minimal value for $x \ge 1$ at $x=1$. Since $\tau \ge 1$, both $x \mapsto (x+1)^{\tau-1}$ and $x \mapsto (x+1)^\tau-x^\tau$ are non-negative and non-decreasing functions for all $x \ge 1$, and thus it follows directly that $x \mapsto \tau (x+1)^{\tau-1}((x+1)^\tau-x^\tau)$ is non-decreasing for all $x \ge 1$.

Altogether, for $j \in \{1\ld d\}$, applying~\eqref{eq:overall-bound_toshow},~\eqref{eq:up_bound_alpha_j_terms},~\eqref{eq:part1_largeJ},~\eqref{eq:part2_largeJ} and~\eqref{eq:part3_largeJ} yields 
\begin{align}\label{eq:bound_i_j_dependence}
    &\sum_{i\in \{1\ld d\}\setminus\{j\}} e^{-\lambda_i^2/2-\lambda_j^2/2}(\cosh(\lambda_i\lambda_j)-1) \\
   & \qquad \le 
    \frac{1}{2e^{1/2}}\left( e^{2^\tau -2^{2\tau}/2}\left(2+\frac{1}{\tau 2^{\tau-1}(2^\tau-1)}\right)
    +
    e^{3^\tau-3^{2\tau}/2}\left( 1+\frac{1}{\tau 3^{\tau-1}(3^\tau-1)}\right)\right).
\end{align} Hence, together with~\eqref{eq:general_assumption_rho_dep} this concludes the proof of part~(i).

\textbf{Part~(ii)}. Assume now that $(G_k)_{k \in \N}$ is a fGn. Recall that the aim is to show for any $\tau \ge 1.302$ that $\inf_{n,d \in \N}\sigma_*^2(\bm{\Lambda}_n)>0$ for all $H \in (0,1/2)$. Throughout this part of the proof, we assume that $H\in (0,1/2)$. Initially, the ensuing inequality is proven: 
    \begin{equation}\label{eq:fGn_rho_asymp}
        |\rho(v)| \le 2^{2-2H}H(1-2H)|v|^{2H-2}, \qquad \text{ for all }v \in \Z\setminus\{-1,0,1\}.
    \end{equation} 
    Due to symmetry, it suffices to show~\eqref{eq:fGn_rho_asymp} for $v \ge 2$. Since the $G_k$'s are negatively correlated (i.e.\ $\rho(v)<0$) when $H \in (0,1/2)$, it follows for $v \ge 2$, that
    \begin{align*}
        |\rho(v)|&=\frac{1}{2}\left(v^{2H}-(v+1)^{2H}\right)+\frac{1}{2}\left(v^{2H}-(v-1)^{2H} \right)=H(1-2H)\int_0^1 \int_{-t}^t (v+x)^{2H-2}\D x \D t\\
        &\le H(1-2H)\int_0^1 2t (v-t)^{2H-2} \D t \le H(1-2H)(v-1)^{2H-2}\le 2^{2-2H}H(1-2H)v^{2H-2}.
    \end{align*} 

    Thus, applying~\eqref{eq:fGn_rho_asymp}, it follows that
    \begin{align}
    \begin{aligned}\label{eq:bound_rho_bound}
        \sum_{v=1}^{n-1}\rho(v)^2
        &=
        \rho(1)+\sum_{v=2}^{n-1}\rho(v)^2 \le \rho(1)^2+\rho(2)^2+2\int_2^\infty \rho(x)^2 \D x\\
        &\le (1-2^{2H-1})^2+\left(2^{2H}-\frac{3^{2H}}{2}-\frac{1}{2}\right)^2+2^{4-4H}H^2(1-2H)^2\int_2^\infty x^{4H-4} \D x\\
        &=(1-2^{2H-1})^2+\left(2^{2H}-\frac{3^{2H}}{2}-\frac{1}{2}\right)^2+\frac{2 H^2(1-2H)^2}{3-4H}, \quad \text{ for all }n \ge 3. 
    \end{aligned}
    \end{align} For all $H\in (0,1/2)$, it holds that 
    \begin{equation}\label{eq:limits_H_dep}
        1+2(1-2^{2H-1})^2+2\left(2^{2H}-\frac{3^{2H}}{2}-\frac{1}{2}\right)^2+\frac{4 H^2(1-2H)^2}{3-4H} \in (1, 3/2),
    \end{equation} where the function in~\eqref{eq:limits_H_dep} is decreasing, attaining the value $3/2$ at $H=0$ and the value $1$ at $H=1/2$. Hence, by~\eqref{eq:bound_infty_norm_corr_1},~\eqref{eq:bound_rho_bound} and~\eqref{eq:limits_H_dep}, it follows that
    \begin{align}\label{eq:infty_norm_bound_Lambda_n}
        \begin{aligned}
            \|\bm{\Lambda}_n-\bm{I}_d\|_\infty
        &\le 
        \frac{3\Bigg( e^{2^\tau -2^{2\tau}/2}\left(2+\frac{1}{\tau 2^{\tau-1}(2^\tau-1)}\right)
    +
    e^{3^\tau-3^{2\tau}/2}\left( 1+\frac{1}{\tau 3^{\tau-1}(3^\tau-1)}\right)\Bigg)}{4e^{1/2}(1-e^{-1})(1-e^{-2^{2\tau}})}.
        \end{aligned}
    \end{align} Note that the bound in~\eqref{eq:infty_norm_bound_Lambda_n} does not depend on $H$, and hence we can choose an arbitrary $H \in (0,1/2)$. Using~\eqref{eq:infty_norm_bound_Lambda_n}, it follows that $\|\bm{\Lambda}_n-\bm{I}_d\|_\infty<1$ for all $\tau \ge 1.302$. Indeed, this holds, since $\|\bm{\Lambda}_n-\bm{I}_d\|_\infty<1$ holds for $\tau =1.302$ and since the upper bound in~\eqref{eq:infty_norm_bound_Lambda_n} is decreasing in $\tau$. To verify this, it suffices to prove that $\tau \mapsto e^{p^\tau-p^{2\tau}/2}$ is a decreasing function in $\tau$ for $p=2$ and $p=3$. By monotonicity of the exponential function, this is equivalent to showing that $\tau \mapsto p^\tau-p^{2\tau}/2$ is decreasing. Hence, we consider the derivative and show its negative:
    \begin{equation*}
        \frac{\D}{\D \tau}(p^\tau-p^{2\tau}/2) = p^\tau \log(p)-p^{2\tau}\log(p)=-p^\tau (p^\tau-1)\log(p)< 0.
    \end{equation*} The last inequality follows, since $p=2$ or $p=3$ and $\tau>1$, implying that $p^\tau >1$ and $\log(p)>0$.
\end{proof}

\subsection{Proof of Corollary~\ref{cor:Breuer--Major_roc}}
\begin{proof}[Proof of Corollary~\ref{cor:Breuer--Major_roc}]
    One of the main differences from Theorem~\ref{thm:main_mult_clt_techncial_thm} is that $\varphi_i=\varphi$ for all $i\in \{1\ld d\}$, i.e.\ $a_{i,q}=a_q$ 
    for all $i=1\ld d$ and $q \ge 2$.

    The proof of Corollary~\ref{cor:Breuer--Major_roc} follows the same steps as the proof of Theorem~\ref{thm:main_mult_clt_techncial_thm} down to~\eqref{eq:f_n,i_kernels}, hence $S_{n,i} \eqd \sum_{ \ell=2}^\infty I_\ell \left(f_n(i,\ell)\right)$, where the kernels $f_n(i,\ell)$ are given by $f_n(i,\ell)=a_\ell n^{-1/2} \sum_{k=\floor{nt_{i-1}}+1}^{\floor{nt_i}}
    e_k^{\otimes \ell} \in \mH^{\odot \ell} $ for all $n \in \N$ and $i \in \{1\ld d\}$. For these kernels, let $$\Delta_n(i,j)\coloneqq \E[\langle -DL^{-1}S_{n,i}, DS_{n,j} \rangle_{\mH}]-\langle -DL^{-1}S_{n,i}, DS_{n,j} \rangle_{\mH}\eqd \sum_{\ell,q=2}^\infty \delta(i,j,q,\ell,n),$$ for all $n \in \N$ and $i,j=1\ld d$, where $$\delta(i,j,q,\ell,n) 
        \coloneqq 
        \frac{1}{\ell}\big( \E\big[\big\langle D I_{\ell}(f_n(i,\ell)), DI_{q}(f_n(j,q)) \big\rangle_{\mH}\big]-\big\langle DI_{\ell}(f_n(i,\ell)), DI_{q}(f_n(j,q)) \big\rangle_{\mH}\big).$$ Thus, it follows from~\eqref{eq:triangle_inequ_sec5} \&~\eqref{eq:Product_hermite_form_delta}, together with~\cite[Prop.~A.1 \& Prop.~A.2]{MR3911126}, that
    \begin{equation*}
        \E\bigg[\max_{1 \le i,j \le d}| \Delta_n(i,j)|\bigg]\le \sum_{q,\ell \ge 2}M_{\ell+q-2} \log^{(q+\ell-2)/2}(2d^2-1+e^{(\ell+q-2)/2-1})\max_{1\le i,j \le d} \|\delta(i,j,q,\ell,n)\|_2,
    \end{equation*} for all $n \in \N$, where $M_{\ell+q-2}\coloneqq(4e/(\ell+q-2))^{(\ell+q-2)/2}$ for all $\ell,q \ge 2$. In a similar fashion to the proof of Theorem~\ref{thm:main_mult_clt_techncial_thm}, we note that $\|\delta(i,j,q,\ell,n)\|_2^2=\var(\langle -DL^{-1}I_\ell(f_n(i,\ell)),DI_q(f_n(j,q))\rangle_\mH)$, and use the definition of the kernels $f$ together with the chain rule to see
    \begin{align*}
        \langle -DL^{-1}I_\ell(f_n(i,\ell)),DI_q(f_n(j,q))\rangle_\mH &\eqd \frac{a_\ell a_q q}{n} \sum_{k=\floor{nt_{i-1}}+1}^{\floor{nt_i}}\sum_{r=\floor{nt_{j-1}}+1}^{\floor{nt_j}}
        H_{\ell-1}(G_k)H_{q-1}(G_r) \rho(k-r),
    \end{align*} for all $n \in \N$, $i,j=1\ld d$ and $\ell,q \ge 2$. We can now bound $\|\delta(i,j,q,\ell,n)\|_2^2$ explicitly in terms of $i,j,q,\ell$. Indeed, as in~\cite[Eq.~(21)]{MR4488569}, it follows that
    \begin{align*}
        \|\delta(i,j,q,\ell,n)\|_2^2 
        \le \frac{a_\ell^2 a_q^2 q^2}{ n^2} \sum_{k,k',r,r'=1}^n |\cov(H_{\ell-1}(G_k)H_{q-1}(G_r),H_{\ell-1}(G_{k'})H_{q-1}(G_{r'})) \rho(k-r)\rho(k'-r')|.
    \end{align*} Next, applying the bounds in~\eqref{eq_Gebelin'_inequality},~\eqref{eq:Cauchy_schwarz},~\eqref{eq:proof_bound_indep_d} and \eqref{eq:ineq_Youngs_applic}, it follows, for $K_n\coloneqq 2\|\rho_n\|_{\ell^1(\Z)}^{3/2}n^{-1/2}$ for all $n \in \N$, that
    \begin{equation*}
        \E\bigg[\max_{1 \le i,j \le d}|\Delta_n(i,j)|\bigg] \le K_n\sum_{q,\ell=2}^\infty  M_{\ell+q-2} \wt C_{q,\ell}q \log^{(\ell+q-2)/2}(2d^2-1+e^{(\ell+q)/2})\max_{1 \le i,j \le d}|a_{i,\ell} a_{j,q}|.
    \end{equation*} This is the exact same bounds as in~\eqref{eq:L1norm_max_Delta_n}, and the remainder of the proof therefore follows as in in the proof of Theorem~\ref{thm:main_mult_clt_techncial_thm}, by choosing $a_{i,q}=a_q$ for all $i \in \{1\ld d\}$ and $q \ge 2$. This yields that there exists a $C_{\bm{\theta}}>0$, independent of $d$ and $n$, such that, 
    \begin{align*}
    d_\HR(\bm{S}_n,\bm{Z})&\le 
    C_{\bm{\theta}} \frac{\Delta(\bm{S}_n)}{\sigma_*^2}\log_+(d) \log_+\left(\frac{\Delta(\bm{S}_n)}{\sigma_{*}^2}\right), \quad \text{ for all }n,d \in \N,
    \end{align*} where $\Delta(\bm{S}_n)= \|\rho_n\|_{\ell^1(\Z)}^{3/2}
    e^{r\log_+^{1/(2\beta)}(d)} n^{-1/2} \log_+(d)^{-1}$. Next, following the inequalities~\eqref{eq:inequality_first_step},~\eqref{eq:inequ_split_constant} and~\eqref{eq:inequality_psi(d)_proof} from the proof of Theorem~\ref{thm:main_theorem}, the statement of Corollary~\ref{cor:Breuer--Major_roc} follows.
\end{proof}

\section{Proofs of Results form Sections~\ref{subsec:generality_of_setting} \&~\ref{sec:ext_results_examples}}\label{sec:proofs_sec_2}

Recall the definition of \emph{Borel-isomorphism}, where we say that two measurable spaces $(A,\mathcal{A})$ and $(B,\mathcal{B})$ are isomorphic if and only if there exists a bijective function $f:A \to B$ such that both $f$ and $f^{-1}$ are measurable. Hence, two metric spaces $(A,\rho_1)$ and $(B,\rho_2)$ are Borel isomorphic if and only if they are isomorphic with their $\sigma$-algebras of Borel sets (\cite[Sec.~13.1]{MR1932358}), in this case we write $A \cong B$.
\begin{proof}[Proof of Proposition~\ref{prop:Phi_G_recovers_all_dist}]
\textbf{Part~(i)} Due to independence, it suffices to prove the equality in distribution for the marginals, which follows if we can prove the following statement. Let $d \ge 1$, $G \sim \mathcal{N}(0,1)$ and $\bm{X}$ be a random vector in $\R^d$ with an arbitrary distribution. Then a measurable function $f: \R \to \R^d$ exists such that $f(G) \eqd \bm{X}$.

We now prove the statement above. Assume that $d = 1$ (use the notation $\bm{X}^{(1)}=X$) and that $G \sim \mathcal{N}(0,1)$. Inverse transform sampling for a uniform random variable $U \sim U(0,1)$, gives a function $T:\R \to \R$ where $T$ and $T^{-1}$ are measurable such that $T(U) \eqd G$. Moreover, we also get a measurable function $g:(0,1) \to \R$ such that $g(U)\eqd X$. Thus, $X \eqd g(U) \eqd g\circ T^{-1}(G)$, and choosing $\varphi=g \circ T^{-1}$ yields a measurable function from $\R$ to $\R$ such that $X \eqd \varphi(G)$, concluding the case where $d=1$.

Assume next that $d \ge 2$ and $G \sim \mathcal{N}(0,1)$. We want to show, for $\bm{X}$ having an arbitrary distribution in $\R^d$, that there exists a measurable function $\varphi:\R \to \R^d$ such that $\bm{X}\eqd \varphi(G)$. By~\cite[Thm~13.1.1]{MR1932358}, it follows that $\R$ and $\R^d$ are Borel-isomorphic, i.e.\ $\R \cong \R^d$. This implies that there exists a bijective function $f: \R \to \R^d$ such that both $f$ and $f^{-1}$ are measurable. From the $d=1$ case, we know that there exists a function $h: \R \to \R$ such that $f^{-1}(\bm{X})\eqd h(G)$, and hence 
\begin{equation*}
    \bm{X}\eqd f\circ h(G)=\varphi(G), \quad \text{where } \varphi=f\circ h.
\end{equation*}  Hence, we have constructed a measurable function $\varphi : \R \to \R^d$, concluding the proof of part~(a).

\textbf{Part~(ii)} Due to~\cite[Thm~2.1]{MR2733224}, it follows that if a function $f\in L^2(\gamma)$ with $f(x)=\sum_{q=0}^\infty a_q H_q(x)$ satisfy $\sum_{q=0}^\infty (q+1)^{2n}q! a_q^2<\infty$ for all $n \in \N_0$, then $f \in \mathscr{S}(\R)$, i.e. in the Schwartz space of rapidly decreasing smooth real-valued functions. Hence, since $\mathscr{S}(\R) \subset C^\infty(\R)$, it follows that $\Phi \in C^\infty$, if, for all $i=1\ld d$, $\sum_{q=0}^\infty (q+1)^{2n}q! a_{i,q}^2<\infty$ for all $n \in \N_0$. By assumption~\eqref{eq:main_assump_a_i,q}, it follows that $\sum_{q=0}^\infty (q+1)^{2n}q! a_{i,q}^2 \le c^2 \sum_{q=2}^\infty (q+1)^{2n} (q!)^{1-2\beta} e^{2\kappa q}$. If $\beta=1/2$ and $\kappa<0$, then $((q+1)^{2n}e^{2\kappa q})_{q \ge 2} \in \ell^1(\N)$ and similarly, if $\beta>1/2$, then Stirling's inequality~\eqref{eq:Stirling's_formula} implies that $((q+1)^{2n}(n!)^{1-2\beta}e^{2\kappa q})_{q \ge 2} \in \ell^1(\N)$, since $1-2\beta<0$. Hence, $ c^2 \sum_{q=2}^\infty (q+1)^{2n} (q!)^{1-2\beta} e^{2\kappa q}< \infty$ in all cases, and we therefore always have that $\Phi \in C^\infty$.

\textbf{Part~(iii)} Since $\bm{S}_n \cid \bm{Z}$ as $n \to \infty$, Portmanteau's theorem~\cite[Thm~5.25]{MR4226142}, implies, for every open set $U \subset \R^d$, that $\liminf_{n\to\infty}\p(\bm{S}_n\in U) \ge \p(\bm{Z}\in U)$. Since $\bm{\Sigma}$ is invertible (i.e.\ $\det \bm{\Sigma}>0$), and since the Gaussian $\bm{Z}$ has density which is continuous and strictly positive on $\R^d$, it follows that 
\begin{equation*}
    \p(\bm{Z} \in U) = \int_U \frac{1}{(2\pi)^{d/2}\sqrt{\det\Sigma}}
\exp\!\Big(-\tfrac12 \bm{x}^\tra \bm{\Sigma}^{-1} \bm{x}\Big) \D \bm{x}>0,
\end{equation*} for all non-empty open sets $U \subset \R^d$.
\end{proof}

\begin{proof}[Proof of Proposition~\ref{prop:example_gen_d_dim}]
    Consider $\Phi(x)=(\varphi_1(x),\varphi_2(x))$, where $\varphi_1(x)=\cos(x)-e^{-1/2}$ and $\varphi_2(x)=\sin(x)-xe^{-1/2}$. By the proof of Corollary~\ref{cor:real_part_charac_fun}, it follows that $\Phi$ satisfies~\eqref{eq:main_assump_a_i,q}. Hence, it remains to prove that the support of the law of $\bm{S}_2$ has a non-empty interior in $\R^2$. Since multiplication with $1/\sqrt{2}$ is a homeomorphism, it suffices to show that the support of the law of $\sqrt{2}\bm{S}_2$ has a non-empty interior in $\R^2$.
    
    Consider $n=2$, and let $F(x,y)=\Phi(x)+\Phi(y)$ where $\sqrt{2}\bm{S}_2=F(G_1,G_2)$. Let $JF(x_0,y_0)$ be the Jacobian of $F$ at $(x_0,y_0)$, where straightforward calculations show that $\det( JF(0,\pi/2)) = 1 - e^{-1/2} \ne 0$, making $JF(0,\pi/2)$ intervible at this point. By the inverse function theorem, there exist neighbourhoods $U \subset \mathbb{R}^2$ of $(0,\pi/2)$ and
    $V \subset \mathbb{R}^2$ of $F(0,\pi/2)$ such that
    $F : U \to V$ is a $C^1$-diffeomorphism. In particular, $V$ is a non-empty open
    subset of $\mathbb{R}^2$ contained in the image of $F$. 
    
    Since $(G_1,G_2)$ is a centered bivariate normal vector with identity covariance matrix, it follows that $(G_1,G_2)$ has a strictly positive density on all of $\R^2$, and hence $\p((G_1,G_2) \in A)>0$ for all open $A \subset \R^2$. Next, since $F: U \to V$ is a $C^1$-diffeomorphism, it follows for all open balls $B \subset V$ that $F^{-1}(B)$ is a non-empty open subset of $\R^2$, and hence
    \[
        \p(\sqrt{2}\bm{S}_2 \in B)
        = \p(F(G_1,G_2) \in B)
        = \p((G_1,G_2) \in F^{-1}(B)) > 0.
    \] Thus, the support of the law of $\sqrt{2}\bm{S}_2$ has a non-empty interior in $\R^2$, concluding that the support of the law of $\sqrt{2}\bm{S}_2$ has non-empty interior in $\R^2$.
\end{proof}

\begin{proof}[Proof of Corollary~\ref{cor:main_theorem_n_dep_d}]
Assume now that $d\in \N$ depends on $n$ such that $d=d(n)\le n^\lambda$ for a fixed arbitrary $\lambda>0$. Recall for any $\beta \in (1/2,1]$ and $\kappa \in \R$, that there for any $\epsilon>0$ exists a $c_\epsilon>0$ such that $\psi_{\beta,\kappa}(d) \le c_\epsilon d^{\epsilon}$ (by~\eqref{eq:bounds_psi}) for any $d \ge 2$. Moreover, for $\beta=1/2$ and $\kappa < \Upsilon $, we have $\psi_{\beta,\kappa}(d)=d^{r}\log(d)$ for all $d \ge 2$. 

Part $(i)$: We start by proving the bound for $d_\HR$. From Corollary~\ref{cor:LRD_SRD_fixed cov:mat}(i), there exists a $C_{\bm{\theta}}\in (0,\infty)$, such that $$d_\HR(\bm{S}_n,\bm{Z})\le 
    C_{\bm{\theta}}
    \psi_{\beta,\kappa}(n^\lambda)n^{-1/2} \log(n) \log_+\left(\,\overline{\overline{\sigma}} \, \wt\sigma_{*}^2/\,\underline{\underline{\sigma}}\,)\right)\wt\sigma_*^{-2},$$ for all $n \in \N$. Assume now that $\beta \in (1/2,1]$ and $\kappa \in \R$. Applying~\eqref{eq:bounds_psi}, we see for any $\epsilon>0$ that there exists a $c_{\epsilon}>0$ such that 
    \begin{equation*}
    d_\HR(\bm{S}_n,\bm{Z})\le 
    c_\epsilon C_{\bm{\theta}}
    n^{\epsilon \lambda-1/2} \log(n)\log_+\left(\,\overline{\overline{\sigma}} \, \wt\sigma_{*}^2/\,\underline{\underline{\sigma}}\,)\right)\wt\sigma_*^{-2}, \quad \text{ for all }n \in \N.
    \end{equation*}
    Thus, for any $\epsilon \in (0,1/(2\lambda))$ we can choose a $\epsilon_\lambda\in (\epsilon\lambda,1/2)$, in which case there exists a constant $K_{\epsilon,\lambda} \in (0,\infty)$, which only depends on $\epsilon$ and $\lambda$, such that
    \begin{equation*}
    d_\HR(\bm{S}_n,\bm{Z})\le 
    K_{\epsilon,\lambda} C_{\bm{\theta}} 
    n^{\epsilon_\lambda-1/2}\log_+\left(\,\overline{\overline{\sigma}} \, \wt\sigma_{*}^2/\,\underline{\underline{\sigma}}\,)\right)\wt\sigma_*^{-2}, \quad \text{ for all }n \in \N.
    \end{equation*} Assume next, that $\beta=1/2$, and note in this case that $\psi_{\beta,\kappa}(d)n^{-1/2}\log(n)=\lambda n^{\lambda r-1/2}\log^2(n)$. Note that $r$ can become arbitrarily small if we choose a $\kappa$ small enough by the definition of $r$. Therefore, choosing $\kappa < \min\{-\log(4e^{1/(2e)}\lambda),\log\big(W(e^{-1 + 1/(2 e)})/e^{1/(2e)}\big)\}/2-\log(24)/2-5/(4e)$, implies $r\lambda <1/2$, by the definition of $r$ from~\eqref{defn:r_constant-2}. Using this choice of $\kappa$, yields the existence of a $\delta \in (r\lambda,1/2)$ and a constant $\wt K_{\epsilon,r,\lambda}$, which depends on $r$, $\epsilon$ and $\lambda$, such that
    \begin{equation*}
    d_\HR(\bm{S}_n,\bm{Z})\le 
    \wt K_{\epsilon,r,\lambda} C_{\bm{\theta}}
    n^{\delta-1/2}\log_+\left(\,\overline{\overline{\sigma}} \, \wt\sigma_{*}^2/\,\underline{\underline{\sigma}}\,)\right)\wt\sigma_*^{-2}, \quad \text{ for all }n \in \N.
    \end{equation*}

    Next, we show the bound for $d_\CD$. From Corollary~\ref{cor:LRD_SRD_fixed cov:mat}(i), there exists a $C_{\bm{\theta},\rho,m}\in (0,\infty)$, such that $$d_\CD(\bm{S}_n,\bm{Z})\le 
    C_{\bm{\theta}} n^{\lambda \cdot 65/24-1/2} ((\sigma_*^2)^{-3/2}+1), \quad \text{ for all }n \in \N.$$ Hence, if $\lambda <12/65$, then there exists a constant $c_\lambda$, such that $n^{\lambda\cdot 65/24-1/2} \le c_\lambda n^{-\zeta_\lambda}$ for some $\zeta_\lambda \in (0,1/2)$ for all $n \in \N$. 

    Part $(ii)$:  From Corollary~\ref{cor:LRD_SRD_fixed cov:mat}(ii), it follows directly that 
    \begin{equation}
    d_\HR(\bm{S}_n,\bm{Z})
    \le 
        C_{\bm{\theta},\wt c,\mu} 
        \psi_{\beta,\kappa}(n^\lambda)\frac{L(n)^{3/2}}{n^{(3\mu-2)/2}}\log_+\bigg(\frac{L(n)^{3/2}}{n^{(3\mu-2)/2}}\bigg) \frac{\log_+\left(\,\overline{\overline{\sigma}} \, \wt\sigma_{*}^2/\,\underline{\underline{\sigma}}\,)\right)}{\wt\sigma_*^2}, 
      \quad \text{ for all }n \in \N.
    \end{equation}

    Assume that $\beta \in (1/2,1]$ and $\kappa \in \R$. Applying~\eqref{eq:bounds_psi}, we see for any $\epsilon>0$ that there exists a $c_\epsilon\in (0,\infty)$, such that 
    \begin{equation*}
    d_\HR(\bm{S}_n,\bm{Z})\le 
    c_\epsilon C_{\bm{\theta},\wt c,\mu} 
    n^{\epsilon\lambda-(3\mu-2)/2} L(n)^{3/2} \log_+\bigg(\frac{L(n)^{3/2}}{n^{(3\mu-2)/2}}\bigg)\frac{\log_+\left(\,\overline{\overline{\sigma}} \, \wt\sigma_{*}^2/\,\underline{\underline{\sigma}}\,)\right)}{\wt\sigma_*^2}, \quad \text{ for all }n \in \N.
    \end{equation*} Thus, for any $\epsilon\in (0,(3\mu-2)/(2\lambda))$ we can choose a $\delta\in (\epsilon\lambda,(3\mu-2)/2) $, such that there exists a constant $K_{\epsilon,\lambda}\in (0,\infty)$, which depends on $\epsilon$ and $\lambda$, yielding
    \begin{equation*}
       d_\HR(\bm{S}_n,\bm{Z})\le K_{\epsilon,\lambda} C_{\bm{\theta},\wt c,\mu}
       n^{\delta-(3\mu-2)/2}\frac{\log_+\left(\,\overline{\overline{\sigma}} \, \wt\sigma_{*}^2/\,\underline{\underline{\sigma}}\,)\right)}{\wt\sigma_*^2}, \quad \text{ for all }n \in \N.
    \end{equation*} 
    
    Next, assume $\beta=1/2$, implying that 
    \begin{equation}
    d_\HR(\bm{S}_n,\bm{Z})
    \le 
        C_{\bm{\theta},\wt c,\mu} 
        n^{r\lambda-(3\mu-2)/2} \log_+(n^\lambda)L(n)^{3/2}\log_+\bigg(\frac{L(n)^{3/2}}{n^{(3\mu-2)/2}}\bigg) \frac{\log_+\left(\,\overline{\overline{\sigma}} \, \wt\sigma_{*}^2/\,\underline{\underline{\sigma}}\,)\right)}{\wt\sigma_*^2}, 
    \end{equation} for all $n \in \N$. By the definition of $r$ from~\eqref{defn:r_constant-2}, choosing 
    $$\kappa< \min\{\log((3\mu-2)/(4e^{1/(2e)}\lambda))/2-\log(24)/2-5/(4e),\Upsilon\},$$ implies that $r<(3\mu-2)/(2\lambda)$. Thus, due to the slow variation of $L$ and the fact that any polynomial growth bounds logarithmic growth, there exists a $\delta \in (r\lambda,(3\mu-2)/2)$ and $K \in (0,\infty)$ (which depends on $r$ and $\lambda$), such that
    \begin{equation*}
       d_\HR(\bm{S}_n,\bm{Z})\le K C_{\bm{\theta},\wt c,\mu}
       n^{\delta-(3\mu-2)/2}\frac{\log_+\left(\,\overline{\overline{\sigma}} \, \wt\sigma_{*}^2/\,\underline{\underline{\sigma}}\,)\right)}{\wt\sigma_*^2}, \quad \text{ for all }n \in \N.
    \end{equation*}

    Finally, we show the bound for $d_\CD$. From Corollary~\ref{cor:LRD_SRD_fixed cov:mat}(ii), there exists a $C_{\bm{\theta},\wt c,\mu}\in (0,\infty)$, such that 
    \begin{equation}
    d_\CD(\bm{S}_n,\bm{Z}) 
      \le
      C_{\bm{\theta},\wt c,\mu} n^{\lambda \cdot 65/24-(3\mu-2)/2} L(n)^{3/2} \big((\sigma_*^2)^{-3/2}+1\big), 
      \quad \text{ for all }n,d \in \N.
    \end{equation} Hence, if $\lambda <(3\mu-2)\cdot 12/65$, then there exists a constant $c_\lambda$, such that $n^{\lambda\cdot 65/24-(3\mu-2)/2} \le c_\lambda n^{-\zeta_\lambda}$ for some $\zeta_\lambda \in (0,(3\mu-2)/2)$ for all $n \in \N$.
\end{proof}

\begin{proof}[Proof of Lemma~\ref{lem:c_infty_main_result}]
Due to~\cite[Prop.~1.4.2(v)]{MR2962301}, it follows that
    \begin{equation}\label{eq:general_form_a_i_q}
        a_{i,q}
        =\frac{1}{q!} \int_\R \varphi_i^{(q)}(x)  \gamma(\D x), \quad \text{ for all }i=1\ld d \text{ and }q \ge 2.
    \end{equation} 
    Assume that $\beta=1/2$ and $\kappa < \Upsilon$, and recall by assumption ii) that $\big|\int_\R \varphi_i^{(q)}(x)  \gamma(\D x)\big| \le c e^{\kappa q}\sqrt{q!}$ for all $q \ge 2$. Thus, it follows from~\eqref{eq:general_form_a_i_q} that $|a_{i,q}|\le c e^{\kappa q}/\sqrt{q!}$ for all $q \ge 2$ as desired. 
    Next, assume that $\beta \in (1/2,1]$ and $\kappa \in \R$, and use assumption ii) to see that $\big|  \int_\R \varphi_i^{(q)}(x)  \gamma(\D x)\big| \le c e^{\kappa q}(q!)^{1-\beta}$ for all $q \ge 2$. From~\eqref{eq:general_form_a_i_q} it now directly follows that $|a_{i,q}|\le c e^{\kappa q}/(q!)^{\beta}$ for all $q \ge 2$, concluding the proof.
\end{proof}

\begin{proof}[Proof of Lemma~\ref{lem:example_beta=1/2}]
    Let $M=\Upsilon$ or $M=-\log(3)/2$.
    
    Part~(i). Assume that $\varphi_i$ is the modification of the density of a $\mathcal{N}(0,\sigma_i^2)$-distribution given in Lemma~\ref{lem:example_beta=1/2}~(i). From~\cite[Thm~2.9]{Hermite_Davis}, it follows that
\begin{equation*}
    \varphi_i(x)=\sum_{q \ge 1} \frac{(-1)^q}{q!2^q\sqrt{2\pi(\sigma_i^2+1)^{2q+1}}}H_{2q}(x)=
    \sum_{q \ge 2} \frac{\1_{\text{even}}(q)(-1)^{q/2}}{(q/2)!2^{q/2}\sqrt{2\pi(\sigma_i^2+1)^{q+1}}}H_{q}(x),
\end{equation*} for all $x\in \R$, with $a_{i,q}=\1_{\text{even}}(q)(-1)^{q/2}((q/2)!2^{q/2}\sqrt{2\pi(\sigma_i^2+1)^{q+1}})^{-1}$ for all $q \ge 2$ and $i=1\ld d$. Here we used the function $\1_{\text{even}}(q)\coloneqq (1+(-1)^q)/2$ for all $q \in \N$. Since $m_i=2$, it follows that $\E[\varphi_i(G_1)]=0=\E[\varphi_i(G_1)G_1]$ as desired. For this $a_{i,q}$, the aim is to find a constant $c \in (0,\infty)$ and a $\kappa<M$ such that $|a_{i,q}| \le ce^{\kappa q}/\sqrt{q!}$ for $q \ge 2$ and $i=1\ld d$. If $q$ is odd, this is trivial, so we will restrict to the case of $q$ even. Using Stirling's inequality~\eqref{eq:Stirling's_formula}, it follows that
\begin{align*}
    |a_{i,q}| =\frac{1}{(q/2)!2^{q/2}\sqrt{2\pi(\sigma_i^2+1)^{q+1}}} &\le \frac{1}{\sqrt{2\pi(\sigma_i^2+1)}}\frac{e^{q/2}}{\sqrt{\pi q}q^{q/2}(\sigma_i^2+1)^{q/2}}\\
    &\le \frac{e^{1/24}}{(4\pi^3)^{1/4}
    e^{-2\Upsilon}}\left( \frac{1}{\min_{1\le i \le d}\sigma_i^2+1}\right)^{q/2}(q!)^{-1/2},
\end{align*} for all $q \ge 2$ and $i=1\ld d$. Note next that
\begin{equation*}
    \left( \frac{1}{\min_{1\le i \le d}\sigma_i^2+1}\right)^{q/2}=\exp\left\{-q\log(\min_{1 \le i \le d}\sigma_i^2+1)/2\right\}=e^{\kappa q},
\end{equation*} where $\kappa = -\log(\min_{1 \le i \le d}\sigma_i^2+1)/2<M$, since $\min_{1\le i \le d}\sigma_i >e^{-2M}-1$. Altogether, it follows that $|a_{i,q}|\le ce^{\kappa q}/\sqrt{q!}$ for all $q\ge 2$ and $i=1\ld d$ where $c=e^{1/24}(4\pi^3)^{-1/4}e^{2M}$, implying that~\eqref{eq:main_assump_a_i,q} holds.

Part~(ii). Assume now that $\varphi_i$ is the modification of the cumulative distribution function of a $\mathcal{N}(0,\sigma_i^2)$-distribution given in Lemma~\ref{lem:example_beta=1/2}(ii). From~\cite[Thm~2.10 \& Eq.~(2-42)]{Hermite_Davis}, it follows that
\begin{equation*}
    \varphi_i(x)=
    \sum_{q =2}^\infty \frac{\1_{\text{odd}}(q)(-1)^{(q-1)/2}}{q((q-1)/2)!2^{(q-1)/2}\sqrt{2\pi(\sigma_i^2+1)^{q}}}H_{q}(x),
\end{equation*} for all $x\in \R$, with $a_{i,q}=\1_{\text{odd}}(q)(-1)^{(q-1)/2}\big(q((q-1)/2)!2^{(q-1)/2}\sqrt{2\pi(\sigma_i^2+1)^{q}}\big)^{-1}$ for all $q \ge 2$ and $i=1\ld d$. Here we used the functions $\1_{\text{odd}}(q)\coloneqq (1-(-1)^q)/2$ for all $q \in \N$. We now show, that there exists a constant $c \in (0,\infty)$ and a $\kappa<M$ such that $|a_{i,q}| \le ce^{\kappa q}/\sqrt{q!}$ for $q \ge 2$ and $i=1\ld d$. If $q$ is even, it is trivial, so we restrict to $q$ odd. Using Stirling's inequality~\eqref{eq:Stirling's_formula}, it follows that
\begin{align*}
    |a_{i,q}|&=\frac{1}{q((q-1)/2)!2^{(q-1)/2}\sqrt{2\pi(\sigma_i^2+1)^{q}}} \le \frac{1}{\sqrt{2\pi}} \frac{e^{(q-1)/2}}{q\sqrt{\pi(q-1)}(q-1)^{(q-1)/2}(\sigma_i^2+1)^{q/2}}\\
    &\le \frac{e^{1/24}}{(2\pi^{3})^{1/4} q(q-1)^{1/4}} \frac{1}{\sqrt{(q-1)!}(\sigma_i^2+1)^{q/2}}\le \frac{e^{1/24}}{(2^5\pi^{3})^{1/4}} (q!)^{-1/2} \left(\frac{1}{\min_{1\le i \le d}\sigma_i^2+1}\right)^{q/2},
\end{align*} for all $q \ge 2$ and $i=1\ld d$. The proof now follows from the same steps as in case i) above, except that in this case $c=e^{1/24}(2^5\pi^{3})^{-1/4}$, implying that~\eqref{eq:main_assump_a_i,q} holds.
\end{proof}

\begin{proof}[Proof of Corollary~\ref{cor:finite_hermite_series}]
    We start by proving~\eqref{eq:d_R_bound_finite expansion}. The proof follows most of the main steps of the proof of Theorem~\ref{thm:main_mult_clt_techncial_thm}. Note that all steps in the proof of Theorem~\ref{thm:main_mult_clt_techncial_thm} hold until~\eqref{eq_defn_delta_i,j,q,l}, where the only difference is that we have a finite Hermite expansion of $\varphi_i$, resulting in the summation being over $\ell \in \{2 
    \ld N\}$ instead of $\ell \ge 2$. As in the proof of Theorem~\ref{thm:main_mult_clt_techncial_thm}, we choose $M_{2N-2}\coloneqq 
    (2e/(N-1))^{N-1}$. Moreover,~\cite[Prop~A.1]{MR3911126} implies that $\Delta_n(i,j)$ is a sub-$(2N-2)$th chaos random variable relative to scale $M_{2N-2}\|\Delta_n(i,j)\|_2$. Hence,~\cite[Prop~A.2]{MR3911126} yields
    \begin{equation*}
        \E\bigg[\max_{1 \le i,j \le d}|\Delta_n(i,j)|\bigg] \le M_{2N-2} \log^{N-1}(2d^2-1+e^{N-2})\max_{1\le i,j \le d} \|\Delta_n(i,j)\|_2, \quad \text{ for all }n \in \N.
    \end{equation*} 

We now bound $\|\Delta_n(i,j)\|_2^2$ explicitly in terms of $i,j,n,N$ and $d$, where we first note that  
\begin{equation*}
    \|\Delta_n(i,j)\|_2^2
    =\var\left(\sum_{\ell=2}^N\sum_{q=2}^N \langle -DL^{-1}I_\ell(f_n(i,\ell)), D I_q(f_n(j,q)) \rangle_{\mH} \right),
\end{equation*} for all $n \in \N$ and $i,j=1 \ld d$. As in the proof of Theorem~\ref{thm:main_mult_clt_techncial_thm}, it holds that
\begin{equation*}
        \langle -DL^{-1}I_\ell(f_n(i,\ell)),DI_q(f_n(j,q))\rangle_\mH  \eqd \frac{a_{i,\ell} a_{j,q} q}{n} \sum_{k=1}^n \sum_{r=1}^n H_{\ell-1}(G_k)H_{q-1}(G_r) \rho(k-r).
    \end{equation*} 
    Thus, for all $n \in \N$ and $i,j =1 \ld d$,
    \begin{align*}
         \|\Delta_n(i,j)\|_2^2 &\le \sum_{\ell,\ell',q,q'=2}^N
         \Bigg(\frac{ |a_{i,\ell} a_{i,\ell'} a_{j,q}a_{j,q'}| q q'}{ n^2} \\
         &\quad\times \sum_{k,k',r,r'=1}^n |\cov(H_{\ell-1}(G_k)H_{q-1}(G_r),H_{\ell'-1}(G_{k'})H_{q'-1}(G_{r'})) \rho(k-r)\rho(k'-r')|\Bigg).
    \end{align*} Next, bounding the inner part of the sum above as in Theorem~\ref{thm:main_mult_clt_techncial_thm} implies that
\begin{align*}
       \max_{1 \le i,j\le d} \|\Delta_n(i,j)\|_2^2 &\le \frac{4\|\rho_n\|_{\ell^1(\Z)}^3}{n} \sum_{\ell,\ell',q,q'=2}^N q q' \wt C_{q,\ell} \wt C_{q',\ell'}  \max_{1 \le i,j\le d}|a_{i,\ell} a_{i,\ell'} a_{j,q}a_{j,q'}|  \\
        &\le  \frac{4\|\rho_n\|_{\ell^1(\Z)}^3}{n} \left(\sum_{\ell,q=2}^N q\wt C_{q,\ell}  \max_{1 \le i,j\le d}|a_{i,\ell} a_{j,q}|  \right)^2, \quad \text{ for all }n \in \N,
    \end{align*} where $ \wt C_{q,\ell} \coloneqq \E[H_{\ell-1}(G_1)^4]^{1/4} \E[H_{q-1}(G_1)^4]^{1/4}$. Thus, applying Lemma~\ref{lem:mom_bound_Hermite} on $\wt C_{q,\ell}$, and the assumption that there exists a $c>0$, $\beta\in [1/2,1]$ and $\kappa \in \R$ (if $\beta=1/2$ then $\kappa \le 0$), such that $|a_{i,q}| \le c e^{\kappa q} (q!)^{-\beta}$ for all $i=1\ld d$ and $q \ge 2$, yields the existence of a $K>0$ such that
\begin{align*}
    &\sum_{\ell,q=2}^N q\wt C_{q,\ell}  \max_{1 \le i,j\le d}|a_{i,\ell} a_{j,q}|\\
    &\le K \sum_{q,\ell=2}^N \frac{e^{\kappa q}}{ (q!)^{\beta}}\frac{e^{\kappa \ell}}{ (\ell!)^{\beta}} q (\ell-1)^{(\ell-1)/2}e^{\ell(\log(3)/2-1/2)}(\ell-1)^{1/4} (q-1)^{(q-1)/2}e^{q(\log(3)/2-1/2)}(q-1)^{1/4}.
\end{align*} Using Stirling's inequality~\eqref{eq:Stirling's_formula} together with the bounds $q/(q-1)^{1/4} \le 2^{1/4}q^{3/4}$ and $(q-1)^{q/2} \le q^{q/2}$ for all $q \ge 2$, it follows that there exists a $\wt K >0$ such that
\begin{align}
\sum_{\ell,q=2}^N q\wt C_{q,\ell}  \max_{1 \le i,j\le d}|a_{i,\ell} a_{j,q}| 
&\le
\wt K \left(\sum_{q=2}^N q^{3/4} (q!)^{1/2-\beta}e^{q (\kappa+\log(3)/2)}\right)^2.
\end{align} If $\beta=1/2$ and $\kappa<-\log(3)/2$, then $\Psi(N,1/2)\coloneqq \sum_{q \ge 2} q^{3/2}e^{q(\kappa+\log(3)/2)}<\infty$, and we may bound the expression in the display above by a finite constant. For $\kappa \in [-\log(3)/2,0]$, the function $q \mapsto q^{3/4}e^{q (\kappa+\log(3)/2)}$, and hence it follows that
\begin{align}
\sum_{\ell,q=2}^N q\wt C_{q,\ell}  \max_{1 \le i,j\le d}|a_{i,\ell} a_{j,q}| 
   \le  
\Psi(N,\beta)\coloneqq \begin{dcases}
   \wt K \left(\sum_{q =2}^N q^{3/4} (q!)^{1/2-\beta}e^{q (\kappa+\log(3)/2)}\right)^2,  \text{ if }\beta>1/2, \\
    \wt K N^{7/2}e^{N(2\kappa+\log(3))}
    , \text{ if } \beta=1/2,\, \kappa \in [-\log(3)/2,0].
\end{dcases}
\end{align} Note that $\Psi(N,\beta)$ is a finite constant for $\beta>1/2$, since it is bounded by $\Psi(\infty,\beta)<\infty$. Altogether, there exists a $\wh K>0$ such that 
    \begin{align*}
        \E\bigg[\max_{1 \le i,j \le d}|\Delta_n(i,j)|\bigg] &\le \wh K\frac{\|\rho_n\|_{\ell^1(\Z)}^{3/2}}{\sqrt{n}}
        \left(\frac{2e \log(2d^2-1+e^{N-2})}{N-1}\right)^{N-1}\Psi(N,\beta).
    \end{align*}
    
Hence, using all of the above, together with~\eqref{eq:inf_d_HR_distance} where $\bm{D}\in \mathcal{D}$ such that $\bm{D}^{1/2}\bm{\Sigma}_n\bm{D}^{1/2}=\bm{\Lambda}_n$, there exists a $\wt{C}_{\bm{\theta}}>0$ such that, it follows for all $n,d \in \N$ holds that
\begin{align*}
    d_\HR(\bm{S}_n,\bm{Z}_n)&\le 
    \wt{C}_{\bm{\theta}} \log(d) \frac{\Delta(d,n,N)}{\sigma_*^2}\left(\left|\log\left(\frac{\Delta(d,n,N)}{\sigma_{*}^2}\right)\right|\vee 1\right), \quad \text{ where }\\
    \Delta(d,n,N)&= \frac{\|\rho_n\|_{\ell^1(\Z)}
    ^{3/2}}{\sqrt{n}} 
    \left(\frac{2e \log(2d^2-1+e^{N-2})}{N-1}\right)^{N-1}\Psi(N,\beta).
    \end{align*}
    Next, using the same steps as in~\eqref{eq:inequ_split_constant} from the proof of Theorem~\ref{thm:main_theorem}, there exists a $C_{\bm{\theta}}>0$, such that
    \begin{align*}
        d_\HR(\bm{S}_n,\bm{Z}_n) &\le C_{\bm{\theta}}\frac{\log_+(\sigma^2_*)}{\sigma^2_*}\log_+(d)\Delta(d,n,N)\log_+(\Delta(d,n,N)), \quad \text{ for all }n,d \in \N.
    \end{align*}

    Next, we prove~\eqref{eq:d_c_bound_finite expansion}. To do so, follow the proof Theorem~\ref{thm:Ext_to_convex_dist} down to~\eqref{eq:finite_herm_exp_proof_d_C}, where the only difference is that we sum only for $q,\ell \in \{2,\ldots, N\}$. Hence, it follows that
    \begin{multline*}
        \sqrt{\E[\|\bm{M}_{\bm{S}_n}-\cov(\bm{S}_n)\|_{\mathrm{H.S.}}^2]}\le \wh C d \frac{\|\rho_n\|_{\ell^1(\Z)}^{3/2}}{\sqrt{n}} \Theta  \\
        \coloneqq \wh C d \frac{\|\rho_n\|_{\ell^1(\Z)}^{3/2}}{\sqrt{n}} \sum_{2 \le q,\ell \le N} \frac{e^{\kappa (q+\ell)}}{(q!\ell!)^\beta} \ell^{(\ell-1)/2}e^{(\ell-1)(\log(3)/2-1/2)}\ell^{1/4} q^{(q-1)/2}e^{(q-1)(\log(3)/2-1/2)}q^{5/4}.
    \end{multline*} Hence, since the sum is independent of $d$, it only remains to show how the sum depends on $N$. Let $w=\log(3)/2-1/2$, and note the multiplicative structure, given by
    \begin{align*}
        \Theta(N,\beta) \coloneqq \left(\sum_{q = 2}^N \frac{e^{\kappa q}}{(q!)^\beta} q^{(q-1)/2}e^{w(q-1)}q^{5/4}  \right) \left(\sum_{\ell=2}^N \frac{e^{\kappa \ell}}{(\ell!)^\beta} \ell^{(\ell-1)/2}e^{w(\ell-1)}\ell^{1/4} \right).
    \end{align*} In the case of $\beta>1/2$, we know from the proof of Theorem~\ref{thm:Ext_to_convex_dist}, that $\Theta(N,\beta)\le \Theta(\infty,\beta)<\infty$. Hence, we may in the case $\beta>1/2$ always bound $\Theta(N,\beta)\le \Theta(\infty,\beta)$, and hence have no $N$-dependence in the bound. 

    If $\beta=1/2$ and $\kappa < -\log(3)/2$, we have again from the proof of Theorem~\ref{thm:Ext_to_convex_dist}, that $\Theta(N,1/2)\le \Theta(\infty,1/2)<\infty$, and we can therefore use this to achieve an upper bound which does not depend on $N$. Finally, if $\beta=1/2$ and $\kappa \in[-\log(3)/2,0]$, we consider for $a \in \{1,5\}$ the function $\wt{\Theta}(N,1/2,a)
\coloneqq \sum_{q = 2}^N e^{\kappa q}(q!)^{-1/2} \, q^{(q-1)/2} e^{w(q-1)} q^{a/4}$. To find the exact behaviour of $\wt{\Theta}(N,1/2,a)$, we apply Stirling's inequality~\eqref{eq:Stirling's_formula}, which implies
\begin{equation*}
    e^w (2\pi)^{1/4}\wt{\Theta}(N,1/2,a)
    \le
    \sum_{q=2}^N
    e^{(\kappa + \log(3)/2) q}
    q^{(a-3)/4} = \begin{dcases}
    \sum_{q=2}^N
    q^{(a-3)/4}, & \!\!\!\text{if } \kappa=-\log(3)/2,\\
    \sum_{q=2}^N
    e^{(\kappa + \log(3)/2) q}
    q^{(a-3)/4}, & \!\!\!\text{if } \kappa\in (-\log(3)/2,0]. 
    \end{dcases} 
\end{equation*} In the case $\kappa \in (-\log(3)/2,0]$, the function $q \mapsto e^{(\kappa+\log(3)/2)q}q^{(a-3)/4}$ is increasing for $a=5$. Moreover, for $a=1$, we use that $e^{(\kappa + \log(3)/2) q}
    q^{(a-3)/4} \le e^{(\kappa + \log(3)/2) q}$. Hence, 
\begin{equation*}
    e^w (2\pi)^{1/4}\wt{\Theta}(N,1/2,a)
    \le \begin{dcases}
    \sum_{q=2}^N
    q^{(a-3)/4}, & \!\!\!\text{if } \kappa=-\log(3)/2,\\
    N e^{(\kappa+\log(3)/2)N}(\1_{\{a=1\}}+\1_{\{a=5\}}\sqrt{N}), & \!\!\!\text{if } \kappa\in (-\log(3)/2,0]. 
    \end{dcases} 
\end{equation*} Hence, altogether, it follows that $\Theta(N,\beta)\le \Psi(N,\beta)\coloneqq \Theta(\infty,\beta)<\infty$ if $\beta>1/2$, and if $\beta=1/2$, we have that
\begin{equation}\label{eq:bound_M_s_n_stirling_d_C_corrected}
    \Theta(N,1/2) \le \Psi(\infty,1/2)\coloneqq  \begin{dcases}
        \Theta(\infty,1/2)<\infty, & \text{ if }\kappa<-\log(3)/2,\\
        \left( \sum_{q=2}^N q^{1/2}\right)\left( \sum_{\ell=2}^N \ell^{-1/2}\right), & \text{ if }\kappa=-\log(3)/2, \\
        N^{5/2} e^{2(\kappa+\log(3)/2)N}, & \text{ if } \kappa \in (-\log(3)/2,0].
    \end{dcases}
\end{equation}

Finally, using the ideas as in the proof of Theorem~\ref{thm:Ext_to_convex_dist}, when going from~\eqref{eq:d_c_before_correlation} to \eqref{eq:d_c_after_correlation}, we may conclude for $\Psi(N,\beta)$ as defined in~\eqref{eq:bound_M_s_n_stirling_d_C_corrected}, that there exists a uniform constant $C_{\bm{\theta}}>0$, which depends only on $\bm{\theta}$, such that
    \begin{equation*}
        d_\CD(\bm{S}_n,\bm{Z}_n) \le C_{\bm{\theta}} d^{65/24} \Psi(N,\beta) n^{-1/2} \|\rho_n\|_{\ell^1(\Z)}^{3/2} \big((\sigma_*^2)^{-3/2}+1\big),
    \end{equation*} for all $n,d \in \N$, concluding the proof.
\end{proof}

\section*{Acknowledgements}

\thanks{
\noindent ABO and DKB are supported by AUFF NOVA grant AUFF-E-2022-9-39. DKB would like to thank the Isaac Newton Institute for Mathematical Sciences, Cambridge, for support and hospitality during the programme Stochastic systems for anomalous diffusion, where work on this paper was undertaken. This work was supported by EPSRC grant EP/Z000580/1.}

\printbibliography

\appendix
\section{Proof of Lemma~\ref{lem:extend_fang_koike}}\label{app:proof_lemma}
\begin{proof}[Proof of Lemma~\ref{lem:extend_fang_koike}]
    We start by extending to the case $d=2$. Assume that $\bm{F}$ has a Stein's kernel $\bm{\tau^F}$ and that $\bm{Z}\sim \mathcal{N}_2(0,\bm{\Sigma})$ with $\sigma_*^2>0$. The aim is to show that the $d=3$ case directly shows the case $d=2$ by adding an independent standard normal to $\bm{F}$. Indeed, let $\wt Z \sim \mathcal{N}(0,1)$ be an independent standard normal, and let $\wh{\bm{F}}\coloneqq (\bm{F}^\intercal , \wt Z)^\intercal$. Moreover, we define
    \begin{equation}
       \bm{\tau}^{\wh{\bm{F}}}(\wh{\bm{F}}) 
        \coloneqq \begin{pmatrix}
            \bm{\tau^F} & \begin{matrix}
                0 \\ 0 
            \end{matrix}\\ \begin{matrix}
                0 & 0
            \end{matrix} & 1
        \end{pmatrix} \quad \text{ and }\quad
        \wh{\bm{Z}} \sim \mathcal{N}_3(\bm{0},\wh{\bm{\Sigma}}), \quad \text{ where } \wh{\bm{\Sigma}}=\begin{pmatrix}
            \bm{\Sigma} & \begin{matrix}
                0 \\ 0 
            \end{matrix}\\ \begin{matrix}
                0 & 0
            \end{matrix} & 1
        \end{pmatrix}.
    \end{equation} Hence, to use~\eqref{eq:log_dim_upperbound} for $d=3$, we want to show that $\wh{\bm{F}}$ has Stein's kernel given by $\bm{\tau}^{\wh{\bm{F}}}(\wh{\bm{F}})$. Thus, we need to verify that $\E[|\bm{\tau}_{i,j}^{\wh{\bm{F}}}(\wh{\bm{F}})|]<\infty$ for any $i,j \in \{1\ld 3\}$ and that $\bm{\tau}^{\wh{\bm{F}}}(\wh{\bm{F}})$ satisfies~\eqref{eq:defn_steins_kernel} for $\wh{\bm{F}}$. By construction of $\bm{\tau}^{\wh{\bm{F}}}(\wh{\bm{F}})$, it follows directly that $\E[|\bm{\tau}_{i,j}^{\wh{\bm{F}}}(\wh{\bm{F}})|]<\infty$ for any $i,j \in \{1\ld 3\}$, since $\bm{\tau^F}(\bm{F})$ is a Stein's Kernel. To verify that $\bm{\tau}^{\wh{\bm{F}}}(\wh{\bm{F}})$ satisfies~\eqref{eq:defn_steins_kernel} for $\wh{\bm{F}}$, we note that $\bm{\tau^F}$ is a Stein's kernel, which by~\eqref{eq:defn_steins_kernel} implies $\sum_{j=1}^2 \E[\partial_j f(\bm{F})F_j]=\sum_{i,j=1}^2 \E[\partial_{ij}f(\bm{F})\bm{\tau}_{i,j}^{\bm{F}}(\bm{F})]$. Applying this equality and the definition of $\bm{\tau}_{i,j}^{\wh{\bm{F}}}(\wh{\bm{F}})$, it follows that~\eqref{eq:defn_steins_kernel} for $\wh{\bm{F}}$ and $\bm{\tau}^{\wh{\bm{F}}}(\wh{\bm{F}})$ reduces to
    \begin{equation}
        \E[\partial_3 f(\wh{\bm{F}})\wt Z]=\sum_{i,j\in \{1,2,3\}:i\vee j=3} \E[\partial_{ij}f(\wh{\bm{F}})\bm{\tau}_{i,j}^{\wh{\bm{F}}}(\wh{\bm{F}})]=\E[\partial_{33}f(\wh{\bm{F}})],
    \end{equation} which holds by~\cite[Lem.~1.1.1]{MR2962301}. Hence, all assumptions of~\cite[Thm~1.1]{MR4312842} are satisfied for $d=3$, and it follows that
    \begin{equation}
        d_\HR(\wh{\bm{F}},\wh{\bm{Z}})\le 
    \frac{C\log(3) \Delta_{\wh{\bm{F}}}}{\sigma_*^2(\wh{\bm{\Sigma}})
    } \log_+\left(\frac{\underline{\sigma}(\wh{\bm{\Sigma}})
    \Delta_{\wh{\bm{F}}}}{\overline{\sigma}(\wh{\bm{\Sigma}})
    \sigma_{*}^2(\wh{\bm{\Sigma}})
    }\right), \, \text{with } \Delta_{\wh{\bm{F}}}\coloneqq \E\bigg[\max_{1 \le i,j \le 3}|\wh{\bm{\Sigma}}_{i,j}-\bm{\tau}^{\wh{\bm{F}}}_{i,j}(\wh{\bm{F}})|\bigg].
    \end{equation} Now, by definition of $\wh{\bm{\Sigma}}$, it follows that $\sigma_*^2(\wh{\bm{\Sigma}})=\sigma_*^2(\bm{\Sigma})\wedge 1$, $\underline{\sigma}(\wh{\bm{\Sigma}})=\underline{\sigma}(\bm{\Sigma})\wedge 1$ and $\overline{\sigma}(\wh{\bm{\Sigma}})=\overline{\sigma}(\bm{\Sigma}) \vee 1$. Finally, by definition of $d_\HR$ and since $(\bm{\tau}^{\wh{\bm{F}}}(\wh{\bm{F}}) )_{i,j}=(\wh{\bm{\Sigma}})_{i,j}$ for all $i,j\in \{1,2,3\}$ where $i\vee j=3$, it follows that 
    \begin{equation}
        d_\HR(\bm{F},\bm{Z})
        \le 
        d_\HR(\wh{\bm{F}},\wh{\bm{Z}})
        \le 
    \frac{C\log(3) \Delta_{\bm{F}}}{\sigma_*^2(\bm{\Sigma})\wedge 1
    } \log_+\left(\frac{(\underline{\sigma}(\bm{\Sigma})\wedge 1)
    \Delta_{\bm{F}}}{(\overline{\sigma}(\bm{\Sigma})\vee 1)
    (\sigma_{*}^2(\bm{\Sigma})\wedge 1)
    }\right), \, \text{with } 
    \end{equation} $\Delta_{\bm{F}}\coloneqq \E\big[\max_{1 \le i,j \le 2}|\wh{\bm{\Sigma}}_{i,j}-\bm{\tau}^{\bm{F}}_{i,j}(\bm{F})|\big]$, concluding the proof for $d=2$. The same method can be used for the proof of $d=1$, which will be omitted here for brevity of the proof.
\end{proof}

\end{document}